\documentclass[11pt,a4paper]{article}

\RequirePackage{dsfont}
\RequirePackage{amsmath, mathrsfs}
\RequirePackage{amssymb}
\RequirePackage{amsthm}
\RequirePackage{enumitem}
\RequirePackage{mathtools}

\usepackage{xcolor}
\usepackage{graphicx}
\usepackage{import}
\usepackage{xifthen}
\usepackage{pdfpages}
\usepackage{transparent}
\usepackage{hyperref}
\usepackage{mathrsfs}
\usepackage{bm}
\usepackage{caption}

\graphicspath{{../figures/}}

\newcommand{\R}{\mathbb{R}}
\newcommand{\Z}{\mathbb{Z}}
\newcommand{\C}{\mathbb{C}}
\newcommand{\N}{\mathbb{N}}

\newcommand{\scH}{\mathscr{H}}
\newcommand{\scG}{\mathscr{G}}
\newcommand{\scC}{\mathscr{C}}
\newcommand{\scF}{\mathscr{F}}
\newcommand{\scL}{\mathscr{L}}
\newcommand{\scQ}{\mathscr{Q}}
\newcommand{\scB}{\mathscr{B}}

\newcommand{\scN}{\mathscr{N}}
\newcommand{\scE}{\mathscr{E}}
\newcommand{\scA}{\mathscr{A}}

\newcommand{\sC}{\mathcal{C}}

\newcommand{\sH}{\mathcal{H}}

\newcommand{\sV}{\mathcal{V}}

\newcommand{\seg}{\mathop\mathrm{seg}}
\newcommand{\rad}{\mathop\mathrm{rad}}
\newcommand{\supp}{\text{supp}}

\newcommand{\diam}{\mathop\mathrm{diam}}

\newcommand{\crd}{\mathrm{crd}}
\newcommand{\prp}{\mathop\mathrm{perp}}

\newcommand{\dist}{\mathop\mathrm{dist}}
\newcommand{\tb}[1]{\tilde{\beta}(#1)}

\newcommand{\spn}{\mathop\mathrm{span}}

\newcommand{\Diam}{\mathop\mathsf{Diam}\nolimits}
\newcommand{\Domain}{\mathop\mathsf{Domain}\nolimits}
\newcommand{\Image}{\mathop\mathsf{Image}\nolimits}
\newcommand{\Edge}{\mathop\mathsf{Edge}\nolimits}
\newcommand{\Line}{\mathop\mathsf{Line}\nolimits}
\newcommand{\Crd}{\mathop\mathsf{Crd}\nolimits}
\newcommand{\Start}{\mathop\mathsf{Start}\nolimits}
\newcommand{\End}{\mathop\mathsf{End}\nolimits}
\newcommand{\Max}{\mathop\mathsf{Max}\nolimits}
\newcommand{\Drift}{\mathop\mathsf{Drift}\nolimits}

\def\XX{{\mathbb{X}}}

\usepackage{amsthm}
\newtheorem{thm}{Theorem}[section]
\newtheorem{lem}[thm]{Lemma}
\newtheorem{cor}[thm]{Corollary}
\newtheorem{prop}[thm]{Proposition}

\newtheorem{maintheorem}{Theorem}

\newenvironment{claim}[1]{\par\noindent\textbf{Claim #1:}}{}
\newenvironment{claimproof}[1]{\par\noindent\textbf{Proof #1:}}{\hfill $\blacksquare$}

\theoremstyle{definition}
\newtheorem{defn}[thm]{Definition}
\newtheorem{remark}[thm]{Remark}

\numberwithin{figure}{section}
\numberwithin{equation}{section}

\title{The Traveling Salesman Theorem for Jordan Curves in Hilbert Space}
\date{}
\author{Jared Krandel}

\begin{document}
\maketitle
\begin{abstract}
     Given a metric space $X$, an Analyst's Traveling Salesman Theorem for $X$ gives a quantitative relationship between the length of a shortest curve containing any subset $E\subseteq X$ and a multi-scale sum measuring the ``flatness'' of $E$. The first such theorem was proven by Jones for $X = \R^2$ and extended to $X = \R^n$ by Okikiolu, while an analogous theorem was proven for Hilbert space, $X = H$, by Schul. Bishop has since shown that if one considers Jordan arcs, then the quantitative relationship given by Jones' and Okikioulu's results can be sharpened. This paper gives a full proof of Schul's original necessary half of the traveling salesman theorem in Hilbert space and provides a sharpening of the theorem's quantitative relationship when restricted to Jordan arcs analogous to Bishop's aforementioned sharpening in $\R^n$.
\end{abstract}
\tableofcontents
\begin{section}{Introduction}

Given a metric space $X$ and a set $E\subseteq X$, how can one tell if there is a curve $\gamma$ of finite length containing $E$? If one does exist, how can one estimate its length in terms of the geometry of $E$, and how can one construct such a curve with length as short as possible? The problem of answering these questions in $X$ is commonly referred to as the \textit{Analyst's Traveling Salesman Problem} for $X$. The study of these problems began when Peter Jones introduced and solved the problem in the standard Euclidean plane $\R^2$ \protect\cite{Jo90}. Okikiolu later extended the result to curves in $\R^n$ \protect\cite{Ok92}, and Schul managed to give an analogue of Okikiolu's and Jones's results in Hilbert space $H$ \protect\cite{Sc07}. Full solutions have been given for sets in Carnot groups \protect\cite{Li22} and graph inverse limit spaces \protect\cite{DS16}, and for Radon measures in $\R^n$ \cite{BS17} and Carnot groups \cite{BLZ22}. Partial results are also available in Banach spaces \cite{BM22a}, \cite{BM22b} and general metric spaces \cite{Ha05}, \cite{DS21}. 

Many authors have also studied traveling salesman-type problems for higher-dimensional sets. This includes H\"{o}lder curves \cite{BNV19},\cite{BZ20}, $C^{1,\alpha}$ surfaces \cite{Ghi20}, lower content $d$-regular sets in $\R^n$ \cite{AS18} and Hilbert space \cite{Hyd21}, analogues of Jordan curves in higher dimensions \cite{Vi20}, and even general sets in $\R^n$ \cite{Hyd22}. Many of these approaches are closely tied to results on parameterization of Reifenberg flat-type sets in $\R^n$ \cite{DT12} and Banach spaces \cite{ENV19}.

One of the central matters in traveling salesman problems is finding a specific quantitative relationship between the Hausdorff measure of the set in question and some measure of its local geometry. The traditional traveling salesman theorems in $\R^n$ and $\ell_2$ provide a relationship which holds for general subsets of the ambient space. Therefore, it seems plausible that one could find a tighter relationship when one restricts their attention to a more geometrically regular class of subsets. A result in this direction was recently achieved by Bishop \cite{Bi22} as part of his study of Weil-Petersson curves \cite{Bi20}. His result is a sharpening of this quantitative relationship for the class of Jordan arcs in $\R^n$. Bishop posed a natural question: Does a similar relationship exist for Jordan arcs in Hilbert space?  

This paper has two primary goals: First, provide a full proof of Schul's necessary condition in the Hilbert space traveling salesman theorem, filling in gaps and correcting errors present in the original presentation in \cite{Sc07}. Second, we answer Bishop's question in the affirmative by providing an analogous sharpening of the Hilbert space traveling salesman theorem when restricted to Jordan arcs.

The proof of our analogue of Bishop's result differs significantly from Bishop's proof because the latter relies heavily on dimension-dependent estimates. We use dimension-independent pieces of Bishop's argument where possible, but largely focus on implementing an extension of the Hilbert space methods introduced in \cite{Sc07}. The first goal is motivated by the discovery of several technical errors in Schul's original proof as presented in \cite{Sc07}. The proof presented here has largely the same outline and general structure as Schul's proof while implementing several new ideas to correct the identified errors. We also mention that the work on the traveling salesman problem in Banach spaces by Badger and McCurdy \cite{BM22a}, \cite{BM22b} provides another proof of the Hilbert space necessary condition via methods which diverge more significantly from ours and those of Schul's original proof.

\begin{subsection}{Overview}
Jones' solution to the traveling salesman problem in $\R^2$ is based on measuring how close a subset $E\subseteq \R^2$ is to being linear locally. In order to do this, he defined what is now called \textit{Jones' beta number}. 
\begin{defn}[\textit{Jones' beta number}]
Fix a Hilbert space $H$ and let $E,Q\subseteq H$ where $Q$ has finite diameter. We define the $\beta$-number for $E$ in the ``window'' $Q$ by
\begin{equation*}
    \beta_{E}(Q) \vcentcolon= \inf_L\sup_{x\in Q\cap E}\frac{\text{dist}(x,L)}{\diam(Q)},
\end{equation*}
where $L$ ranges over all affine lines in $H$.
\end{defn}
 One can interpret the number $\beta_E(Q)\diam(Q)$ as the radius of the minimal width cylinder in $H$ which contains the set $E\cap Q$. The factor of $\diam(Q)$ on the right-hand side in the definition ensures that $\beta_E(Q)$ is scale-invariant. We always have $0 \leq \beta_E(Q) \leq 1$ where $\beta_E(Q) = 0$ implies $E\cap Q\subseteq L$ for some line $L$ while $\beta_E(Q) \geq \epsilon$ for some constant $\epsilon > 0$ implies that for every choice of $L$, there exists a point in $E\cap Q$ of distance $\epsilon\diam(Q)$ from $L$. Jones and Okikiolu used these numbers to characterize subsets of rectifiable curves in $\R^2$ and $\R^n$ respectively in the following theorem:
\begin{thm}{(Jones for $n=2$ \protect\cite{Jo90}, Okikiolu for $n\geq 2$ \protect\cite{Ok92})}\label{t:Rn-tst}
Let $E\subseteq \R^n$. $E$ is contained in a rectifiable curve if and only if
\begin{equation*}
    \beta_E^2(\R^n) \vcentcolon= \diam(E) + \sum_{Q\in\Delta(\R^n)}\beta_{E}(3Q)^2\diam(Q) < \infty
\end{equation*}
where $\Delta(\R^n)$ is the set of all dyadic cubes in $\R^n$ and $3Q$ is the cube with the same center as $Q$ but three times the side length. If $\Sigma$ is a connected set of shortest length containing $E$, then
\begin{equation}\label{e:Rn-necessary}
    \beta_\Sigma^2(\R^n) \lesssim_n \sH^1(\Sigma)
\end{equation}
and
\begin{equation}\label{e:Rn-sufficient}
      \beta_E^2(\R^n) \gtrsim_n \sH^1(\Sigma).
\end{equation}
\end{thm}
It is important to take $3Q$ rather than $Q$ so that the family $\{3Q\}_{Q\in\Delta(\R^n)}$ ``covers'' $\R^n$ sufficiently well. More precisely, any subset $B\subseteq \R^n$ is contained in a cube $3Q$ of comparable diameter, while there may not exist such a standard dyadic cube $Q$ with this property. The exponent $2$ appears in Theorem \ref{t:Rn-tst} because the Pythagorean theorem allows one to estimate the difference in length between a line segment and a slight perturbation of the segment by a small distance $d$ in a perpendicular direction by a factor proportional to $d^2$ (See Remark \ref{rem:pythagorean-theorem}). We recommend the reader sees the introduction of \cite{BM22a} for further intuition on the behavior of $\beta$-numbers for subsets of rectifiable curves in $H$ (and in Banach spaces). 
\begin{remark}[The Pythagorean theorem and triangle inequality excess]\label{rem:pythagorean-theorem}
\begin{figure}[h]
    \centering
    \includegraphics[scale=0.82]{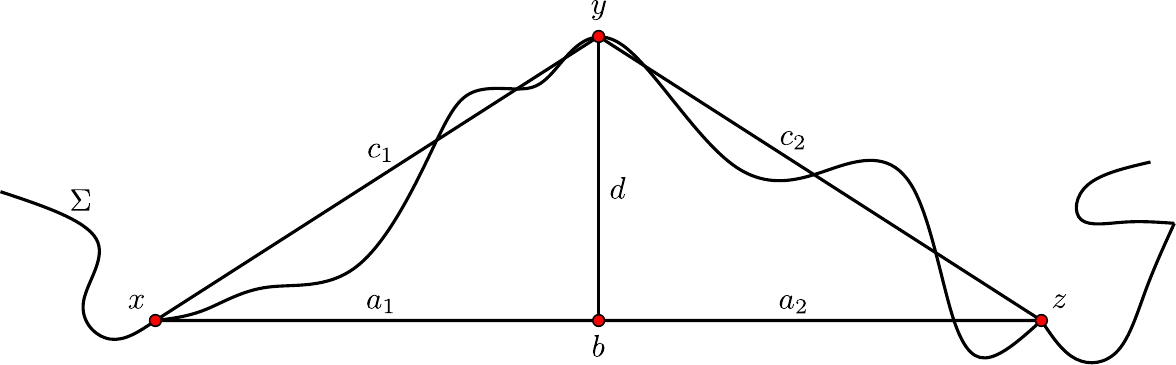}
    \caption{}
    \label{fig:Pythagorean-theorem}
\end{figure}
Let $x,y,z\in\Sigma$. Applying the Pythagorean theorem in Figure \ref{fig:Pythagorean-theorem} gives $d^2 = c_i^2 - a_i^2 = (c_i-a_i)(c_i+a_i)$ for $i=1,2$. If we assume that $R>0$ is such that $c_i\simeq a_i\simeq R$, then
\begin{equation*}
    2d^2 = (c_1-a_1)(c_1+a_1) + (c_2-a_2)(c_2+a_2) \simeq 2R(c_1+c_2-a_1-a_1).
\end{equation*}
If we further assume that there exists a dyadic cube $Q$ such that $x,y,z\in 3Q, \diam(Q)\simeq R$, and $d\simeq\beta_\Sigma(Q)\diam(Q)$, then
\begin{equation*}
    \beta_\Sigma(Q)^2\diam(Q) \simeq c_1 + c_2 - a_1 - a_2 = |x-y|+|y-z|-|x-z|.
\end{equation*}
This final equality demonstrates why beta numbers are often said to measure the \textit{triangle inequality excess} inside a cube $3Q$.
\end{remark}

Bishop's results say that if we restrict our attention to \textit{Jordan arcs} then \eqref{e:Rn-necessary} and \eqref{e:Rn-sufficient} can be improved.
\begin{defn}[\textit{Jordan arcs and curves}]
For a metric space $X$, we say that $\Gamma\subseteq X$ is a \textit{rectifiable arc} if $\sH^1(\Gamma) < \infty$ and there exists a surjective continuous map $\gamma:I\rightarrow\Gamma$ for some closed interval $I\vcentcolon=[a,b]\subseteq\R$. We also refer to the map $\gamma$ as a rectifiable arc. Whenever we refer to such $\Gamma$, we implicitly have a particular parameterizing map $\gamma$ in mind and really mean the pair $(\Gamma,\gamma)$. We refer to the points $\gamma(a),$ and $\gamma(b)$ as the \textit{endpoints} of $\gamma$ (or $\Gamma$) and we define
\begin{equation*}
    \crd(\Gamma) \vcentcolon= \crd(\gamma) \vcentcolon= \dist(\gamma(a),\gamma(b)).
\end{equation*}
We refer to $\crd(\Gamma)$ as the \textit{chord length} of $\Gamma$ and refer to the line (segment) which passes through the two endpoints of $\Gamma$ as its \textit{chord} or its \textit{chord line (segment)}. We additionally define
\begin{equation*}
    \ell(\Gamma)\vcentcolon=\ell(\gamma)
\end{equation*}
as the length of $\Gamma$ (also see \ref{e:param-length}). We call $\Gamma$ a \textit{Jordan arc} if we can take $\gamma$ to be bijective. We call $\Gamma$ a \textit{Jordan curve} if there exists a map $\gamma:[a,b]\rightarrow \Gamma$ which is injective on $[a,b)$, but has $\gamma(a) = \gamma(b)$. In general, any rectifiable arc $\gamma$ with $\gamma(a) = \gamma(b)$ is called \textit{closed}. For Jordan arcs and curves $\ell(\Gamma) = \sH^1(\Gamma)$, but this does not hold for general rectifiable arcs (for more on this, see Remark \ref{rem:gen-arcs}).
\end{defn}
The chord length is an integral part of Bishop's following improvement:
\begin{thm}[\protect\cite{Bi22} Theorem 1.2] \label{t:bistst}
Let $\Gamma\subseteq \R^n$ be a Jordan arc. Then
\begin{equation}\label{e:bistst}
    \sum_{Q\in\Delta(\R^n)}\beta_\Gamma(3Q)^2\diam(Q) \simeq_n \ell(\Gamma) - \crd(\Gamma).
\end{equation}
\end{thm}
This result has the following corollary for Jordan curves:
\begin{cor}[\cite{Bi22} Corollary 1.3]
If $\Gamma\subseteq\R^n$ is a Jordan curve, then 
\begin{equation*}
    \sum_{Q\in\Delta(\R^n)}\beta_\Gamma(3Q)^2\diam(Q) \simeq_n\ell(\Gamma).
\end{equation*}
\end{cor}
One can think of deriving this by applying Theorem \ref{t:bistst} while taking $\crd(\Gamma)\rightarrow0$. The significance of the inequalities in \eqref{e:bistst} is easier to see if we compare each directly with its corresponding inequality in Theorem \ref{t:Rn-tst}. 
\begin{remark}[The $\lesssim$ improvement in \eqref{e:bistst}]\label{rem:lesssim}
The inequality \eqref{e:Rn-necessary} and the $\lesssim$ direction of \eqref{e:bistst} applied to a Jordan arc $\Gamma$ are respectively equivalent to
\begin{align}
    \sum_{Q\in\Delta(\R^n)}\beta_\Gamma(3Q)^2\diam(Q) &\leq C_1(n)\ell(\Gamma), \label{e:jones-nec}\\ 
    \sum_{Q\in\Delta(\R^n)}\beta_\Gamma(3Q)^2\diam(Q) &\leq C_2(n)(\ell(\Gamma) - \crd(\Gamma))\label{e:bis-nec}
\end{align}
for some constants $C_1(n),C_2(n) > 0$ depending on $n$ where we used the fact that $\diam(\Gamma) \leq \ell(\Gamma)$. This means \eqref{e:bistst} improves the inequality by replacing $\ell(\Gamma)$ with $\ell(\Gamma)-\crd(\Gamma)$ on the right-hand side. One can see how this improvement manifests by considering the case when $\Gamma$ is a line segment. In this case, $\beta_{\Gamma}(3Q) = 0$ for all $Q\in\Delta(\R^n)$ while $\crd(\Gamma) = \ell(\Gamma)$. This means \eqref{e:jones-nec} becomes $0 \leq C_1(n)\ell(\Gamma)$ while \eqref{e:bis-nec} becomes $0 \leq 0$. Bishop's improvement pulls the slack out of Jones and Okikiolu's inequality by recognizing that $\crd(\Gamma)$ is a global measure that accounts for how much $\Gamma$ looks like its chord segment: the beta numbers do not see the ``component'' of $\Gamma$ along its chord.
\end{remark}

\begin{remark}[The $\gtrsim$ improvement in \ref{e:bistst}]\label{rem:gtrsim}
The inequality \eqref{e:Rn-sufficient} and the $\gtrsim$ direction of \eqref{e:bistst} applied to a Jordan arc $\Gamma$ and $E\subseteq \Gamma$ are respectively equivalent to
\begin{align}
    (1+\delta)\diam(E) + C(\delta,n)\sum_{Q\in\Delta(\R^n)}\beta_E(3Q)^2\diam(Q) &\geq \ell(\Gamma), \label{e:jones-suf}\\
    \crd(\Gamma) + C_2(n)\sum_{Q\in\Delta(\R^n)}\beta_\Gamma(3Q)^2\diam(Q) &\geq \ell(\Gamma).\label{e:bis-suf}
\end{align}
where in \eqref{e:jones-suf}, $\delta > 0$ is arbitrary but $C(\delta,n)\rightarrow\infty$ as $\delta\rightarrow 0$. This is not exactly as was stated in \eqref{e:Rn-sufficient}, but Jones shows that one can arrange the inequality this way. This means \eqref{e:bis-suf} changes \eqref{e:jones-suf} by replacing the $C_1(n)\diam(E)$ term on the left-hand side with $\crd(\Gamma)$ and exchanges the $\beta_E$ numbers for the (larger) $\beta_\Gamma$ numbers. 

Bishop \cite{Bi22} shows how this change manifests by considering the example set $E = \{0,1,i\beta\}\subseteq\C$ with $0 < \beta \ll 1$. One can calculate that $\sum_{Q\in\Delta(\R^2)}\beta_E(3Q)^2\diam(Q) \leq C\beta^2$ while $\diam(E) = \sqrt{1 + \beta^2} \leq 1 + c\beta^2$ and the shortest curve $\Gamma = [i\beta,0]\cup[0,1]$ containing $E$ satisfies $\ell(\Gamma) = 1 + \beta$. By taking $\beta\rightarrow 0$, this configuration gives a family of examples showing that one cannot take $\delta = 0$ in \eqref{e:jones-suf}. The only cubes $Q$ which contribute to $\sum\beta_E(3Q)^2\diam(Q)$ are those for which $E\subseteq 3Q$ because $E$ only contains three points total. But, for any cube $Q$ which contains $0$, $\beta_\Gamma(3Q) \gtrsim 1$ so that one can show that $\sum\beta_\Gamma(3Q)^2\diam(Q) \geq 1 + c\beta$. The point is that the connectedness of $\Gamma$ adds more geometry with more locations and scales to measure curvature at, increasing the contribution of the $\beta$-numbers and loosening the requirement on the $\diam(E)$ term.
\end{remark}

In each of the previously stated traveling salesman theorems, the implicit constants in the given inequalities increase exponentially as $n\rightarrow\infty$. The constants' exponential blowup can be attributed to the exponential increase in the relative number of dyadic cubes on each scale as $n$ increases. To formulate a version of the traveling salesman theorem in Hilbert space $H$, Schul invented a replacement for the set of dyadic cubes called a \textit{multiresolution family}.
\begin{defn}[\textit{Multiresolution family}]
Fix a connected set $\Sigma\subseteq H$. For $\epsilon > 0$, we call a set $E\subseteq\Sigma$ an \textit{$\epsilon$-net} of $\Sigma$ if both
\begin{enumerate}[label=(\roman*)]
    \item For all $x,y\in E$, $|x-y| > \epsilon$, \text{and} \label{i:net-space}
    \item $\Sigma \subseteq \bigcup_{x\in E}B(x,\epsilon)$ \label{i:net-cover}
\end{enumerate}
Any subset satisfying \ref{i:net-space} can be extended to an $\epsilon$-net since it can be extended to satisfy \ref{i:net-cover} by adding a maximal number of appropriately spaced points. Fixing an integer $n_0$, let $X_{n_0}\subseteq \Sigma$ be a $2^{-n_0}$-net. The extension property implies the existence of a sequence of $2^{-n}$-nets $\{X_{n}\}$ satisfying $X_{n}\subseteq X_{n+1}$. Fix a constant $A > 1$ and put
\begin{equation}\label{e:mrf}
    \scH \vcentcolon= \{B(x,A2^{-n}):x\in X_{n},\ n\in\Z,\ n\geq n_0\}
\end{equation}
where $B(x,A2^{-n})$ is the closed ball of radius $A2^{-n}$ around $x$. We call $\scH$ a \textit{multiresolution family} for $\Sigma$ and refer to an element $Q\in\scH$ as a ball. Given $Q = B(x,A2^{-n})$, define $x_Q\vcentcolon=x$ and $\rad(Q) \vcentcolon= A2^{-n}$. For $\lambda > 0$, we let $\lambda Q \vcentcolon= B(x_Q,\lambda\rad(Q))$. The starting point $n_0$ is of no consequence; all of the results obtained will be independent of it.
\end{defn}
\begin{remark}[Reduction to $H = \ell_2$]
If $\Sigma\subseteq H$ is a closed, connected set with $\sH^1(\Sigma) < \infty$, then $\Sigma$ is compact and hence a separable subset of $H$. This gives the existence of a countable set of vectors $v_1,v_2,\ldots$ such that $\Sigma\subseteq \overline{\spn\{v_1,v_2,\ldots\}}=\vcentcolon V$. Hence, $\Sigma$ is contained in the separable subspace $V\subseteq H$ which is isometric (via a linear transformation) to $\ell_2$. Therefore, it suffices to fix $H = \ell_2$ in the following theorems concerning Hilbert space.
\end{remark}
The important difference between a multiresolution family and the set of dyadic cubes is that the former is centered on the set, while the latter is a partition of the ambient space. In infinite dimensional space (and general metric spaces), it is necessary to concentrate on the intrinsic properties of the set in question rather than where the set happens to lie relative to pre-defined pieces of the ambient space. We can now state Schul's result:
\begin{thm}[\protect\cite{Sc07} Theorem 1.1, Theorem 1.5] \label{t:hstst}
Let $E\subseteq \ell_2$ and let $\scH$ be a multiresolution family for $E$ with inflation factor $A > 200$. Then $E$ is contained in a rectifiable curve if and only if 
\begin{equation*}
    \beta_E^2(\scH) \vcentcolon= \diam(E) + \sum_{Q\in\scH}\beta_E(Q)^2\diam(Q) < \infty.
\end{equation*}
If $\Sigma\subseteq H$ is a connected set of shortest length containing $E$, then 
\begin{equation}\label{e:HS-nec}
    \beta_\Sigma^2(\scH) \lesssim_A \sH^1(\Sigma)
\end{equation}
and
\begin{equation}\label{e:HS-suf}
     \sH^1(\Sigma) \lesssim_A \beta_E^2(\scH).
\end{equation}
\end{thm}
The exponent 2 can again be attributed to the fact that the Pythagorean theorem holds in Hilbert space. The inflation factor $A$ given in the definition of $\scH$ is the analogue of taking $3Q$ rather than $Q$ in the Euclidean space traveling salesman theorems. Schul's proof of \eqref{e:HS-suf} closely parallels Jones' constructive proof of \eqref{e:Rn-sufficient}, replacing dimension-dependent estimates with dimension-independent estimates needed. We mention here that Badger, Naples, and Vellis provide a refined constructive proof of this result in \cite{BNV19} which produces a nice sequence of parameterizations which they used to prove traveling salesman sufficient conditions for H\"{o}lder curves. 

On the other hand, Schul's proof of \eqref{e:HS-nec} differs significantly from Jones's original proof, incorporates some key ideas from Okikiolu's proof in $\R^n$, and introduces several ingenious new constructions to remove dimension-dependent estimates. Unfortunately, several errors have since been discovered in the original presentation in \cite{Sc07}, leaving gaps in the proof. The results of this paper will fill in these gaps, providing a full proof of \eqref{e:HS-nec} in parallel with the following new results:
\begin{maintheorem}\label{t:thmA}
Let $\Gamma\subseteq \ell_2$ be a Jordan arc. For any multiresolution family $\scH$ associated to $\Gamma$ with inflation factor $A > 200$, we have
\begin{equation}\label{e:thmA}
    \sum_{Q\in\scH}\beta_\Gamma(Q)^2\diam(Q) \lesssim_A \ell(\Gamma) - \crd(\Gamma).
\end{equation}
\end{maintheorem}
The second main result is the other side of the inequality in Theorem \ref{t:bistst} for $\ell_2$:
\begin{maintheorem}\label{t:thmB}
Let $\Gamma\subseteq \ell_2$ be a Jordan arc. For any multiresolution family $\scH$ associated to $\Gamma$ with inflation factor $A > 30$, we have
\begin{equation}\label{e:thmB}
    \sum_{Q\in\scH}\beta_\Gamma(Q)^2\diam(Q) \gtrsim_A \ell(\Gamma) - \crd(\Gamma).
\end{equation}
\end{maintheorem}
These results are to Hilbert space what Bishop's Theorem \ref{t:bistst} is to Euclidean space. One can again look to Remarks \ref{rem:lesssim} and \ref{rem:gtrsim} to gain intuition about the nature of these improvements over the estimates in Theorem \ref{t:hstst}. 
\begin{remark}[General rectifiable arcs]\label{rem:gen-arcs}
Theorems \ref{t:thmA} and \ref{t:thmB} raise a natural question: Do similar results hold for general rectifiable arcs? In this case, we must be careful about the definitions. If $\gamma:[0,\ell(\gamma)]\rightarrow\Sigma$ is any constant arc length parameterization of a compact, connected set $\Sigma\subseteq\ell_2$, then we interpret $\ell$ as the pushforward of Lebesgue measure onto $\Sigma$ which does not necessarily coincide with $\sH^1|_{\Sigma}$ as it does for a Jordan arc or curve. If $\Gamma'$ is a rectifiable arc, then $\ell(\Gamma') \geq \sH^1(\Gamma')$, so Theorem \ref{t:thmA} is weaker than the more natural inequality
\begin{equation}\label{e:rect-arc-beta}
    \sum_{Q\in\scH}\beta_{\Gamma'}(Q)^2\diam(Q) \lesssim_A \sH^1(\Gamma') - \crd(\Gamma').
\end{equation}
Whether or not \eqref{e:rect-arc-beta} holds remains open. In Remark \ref{rem:rect-almost-flat} we give some ideas on how one might modify some of our methods in this direction.
\end{remark}

\end{subsection}
\begin{subsection}{Related Results and Questions}
\begin{subsubsection}{Weil-Petersson Curves}
Theorem \ref{t:bistst} arose as an improvement to the traveling salesman theorem necessary to connect some of the geometric characterizations of \textit{Weil-Petersson} curves discovered in \cite{Bi20} (a few out of 26 total definitions given!). The Weil-Petersson curves are defined to be the closure of smooth curves in $\R^2$ in the Weil-Petersson metric on universal Teichm\"{u}ller space introduced in \cite{TT06} by Takhtajan and Teo for studying problems related to string theory. This class of curves has also been studied in relation to computer vision \cite{FKL14}, \cite{FN17}, \cite{SM06}, and Schramm-Loewner evolutions \cite{Wa19a}, \cite{Wa19b}.

The following result gives the aforementioned characterizations when $n=2$. We say that a curve $\Gamma$ is \textit{chord-arc} if any two points $x,y\in \Gamma$ are connected by a subarc $\gamma \subseteq\Gamma$ with $\ell(\gamma) \leq C|x-y|$ for some constant $C$ independent of $x$ and $y$.
\begin{thm}[\cite{Bi22} Theorem 1.4]\label{t:bis-tst-char}
The following are equivalent for a closed Jordan curve $\Gamma\subseteq\R^n,\ n\geq 2$,
\begin{enumerate}[label=(\roman*)]
    \item $\Gamma$ satisfies
    \begin{equation*}
        \sum_{Q\in\Delta(\R^n)}\beta_\Gamma(3Q)^2 < \infty.
    \end{equation*}
    \label{i:beta-squared}
    
    \item $\Gamma$ is chord-arc, and for any dyadic decomposition of $\Gamma$, the inscribed polygons $\{\Gamma_n\}$ defined by taking the $n$-th generation points as vertices satisfy
    \begin{equation*}
        \sum_{n=1}^\infty 2^n[\ell(\Gamma) - \ell(\Gamma_n)] < \infty
    \end{equation*}
    with a bound that is independent of the choice of decomposition.\label{i:dyadic-decomp}
    
    \item $\Gamma$ has finite M\"{o}bius energy. That is,
    \begin{equation*}
        \mathrm{\mathop{M\ddot{o}b}}(\Gamma)\vcentcolon=\int_\Gamma\int_\Gamma\left(\frac{1}{|x-y|^2} - \frac{1}{\ell(x,y)^2}\right)dxdy < \infty
    \end{equation*}
    where $\ell(x,y)$ is the length of the shortest arc contained in $\Gamma$ connecting $x$ and $y$ and the integration is with respect to arc length measure. \label{i:mob-energy}
\end{enumerate}
\end{thm}
The M\"{o}bius energy in \ref{i:mob-energy} was one of several functionals introduced by O'Hara to study knots \cite{Ju91}, \cite{Ju92}. One can interpret \ref{i:beta-squared} as a bound on the total curvature of the curve $\Gamma$ over all locations and scales. The missing factor of $\diam(Q)$ when compared with the sums that appear in the traveling salesman theorems makes this condition much harder to satisfy in general. For instance, a curve satisfying \ref{i:beta-squared} cannot have a ``corner'' (conical type singularity) because this would give an infinite collection of cubes $Q$ such that $\beta_\Gamma(3Q) \gtrsim 1$. In \ref{i:dyadic-decomp}, a dyadic decomposition is an ordered collection of points contained in $\Gamma$ which divide $\Gamma$ into $2^n$ intervals of equal length. We let $\Gamma_n$ be the polygon with these points as vertices. Hence, we interpret \ref{i:dyadic-decomp} as measuring the rate of convergence of the length of inscribed polygonal approximations to $\Gamma$ to the length of $\Gamma$ itself. The term $\ell(\gamma) - \crd(\gamma)$ a subarc $\gamma$ can be expected to appear because it measures exactly this form of difference in length. 

One of the corollaries of Theorem \ref{t:bistst} that Bishop uses to prove Theorem \ref{t:bis-tst-char} translates directly to our setting:
\begin{cor}[\cite{Bi22} Corollary 5.2 in $\R^n$]
If $\Gamma\subseteq\ell_2$ is a closed Jordan curve and $S\vcentcolon=\sum_{Q\in\scH}\beta_\Gamma(Q)^2 < \infty$, then $\Gamma$ is chord-arc, i.e.. any pair of points $z,w\in\Gamma$ are connected by a subarc $\gamma$ with $\ell(\gamma)\lesssim |z-w|$.
\end{cor} One can check that Bishop's proof of the $\R^n$-version is independent of the dimension $n$ so that this result follows if one replaces usages of Theorem \ref{t:bistst} there with Theorem \ref{t:thmB}. For more on how the traveling salesman theorem applies to Weil-Petersson curves and related subjects, the reader should see Section 4 of \cite{Bi20}. The rest of the paper gives connections between these curves and a plethora of objects such as conformal maps, Schwarzian derivatives, quasiconformal mappings, Sobolev spaces, and minimal surfaces in hyperbolic 3-space.
\end{subsubsection}

\begin{subsubsection}{Traveling salesman in Banach spaces}
A separate, related branch of research is that of traveling salesman problems in more general metric spaces. Recent success has been achieved by Matthew Badger and Sean McCurdy \cite{BM22a}, \cite{BM22b} in attaining traveling salesman-type necessary and sufficient conditions in Banach spaces. Much of their work was inspired by the paper of Edelen, Naber, and Valtorta \cite{ENV19} which implemented the Reifenberg algorithm in Banach spaces.

Roughly speaking, \cite{ENV19} gave a Banach space version of Reifenberg's topological disk theorem \cite{Re60}, which states that any subset $\Sigma\subseteq\R^n$ which is sufficiently bilaterally close to an affine $k$-dimensional plane at all locations in $\Sigma$ and all sufficiently small scales is locally homeomorphic to an open subset of $\R^k$, hence is locally topologically a $k$-dimensional disk. Edelen, Naber, and Valtorta extended this result to infinite-dimensional Banach spaces, and gave a traveling salesman-type application in the form of a structure theorem for measures in Banach spaces (\cite{ENV19} Theorem 2.1). They give a sufficient condition on a Borel measure $\mu$ to be well concentrated around a $k$-dimensional set in terms of the pointwise boundedness of a sum of integral beta numbers $\beta_\mu^k$ which measure how close $\mu$ is locally to a $k$-dimensional affine plane. An important aspect of their result which is particularly relevant to Badger and McCurdy's work is that the exponent on $\beta_\mu^k$ appearing in Edelen, Naber, and Valtorta's sum differs based on the geometric structure of the Banach space. The exponent $2$ appears in the Hilbert space case, but one must make other assumptions on the geometry in more general Banach spaces to say something stronger.

Indeed, Badger and McCurdy use the well-studied notions of \textit{modulus of smoothness} and \textit{modulus of convexity} to estimate the triangle inequality excess (recall Remark \ref{rem:pythagorean-theorem}) from above and below respectively. They apply their results to prove necessary and sufficient conditions in $\ell_p$ spaces for $1 < p < \infty$. A major difference between $\ell_p,\ p\not=2$ and $\ell_2$ is that the sharp necessary and sufficient conditions they prove in $\ell_p$ using the standard Jones beta number diverge from one another. One reason this result might be expected is that the triangle inequality excess for orthogonally (in the $\ell_2$ sense) perturbed vectors differs based on the direction of the perturbed vector.

To illustrate this point, if $e_1,e_2$ are standard unit basis vectors for $\ell_p$ and $0 < \delta \ll 1$, then 
\begin{equation*}
    |e_1 + \delta e_2|_p - |e_1|_p = \left( 1 + \delta^p \right)^{1/p} - 1 \simeq_p \delta^p.
\end{equation*}
On the other hand, suppose we take a ``diagonal'' vector $v = \frac{1}{2^{1/p}}(e_1 + e_2)$ and perturb it by the orthogonal (in the $\ell_2$ sense) vector $w = \frac{1}{2^{1/p}}(e_1-e_2)$. We have
\begin{equation*}
    |v+\delta w|_p - |v|_p = \frac{1}{2^{1/p}}\left((1+\delta)^p + (1-\delta)^p  \right)^{1/p} - 1 \simeq_p \delta^2.
\end{equation*}
The length gain by small orthogonal perturbation in $\ell_p$ varies depending on the direction of the perturbed vector in contrast to the $\ell_2$ case.

\begin{thm}[\cite{BM22a} Theorem 1.6](sharp sufficient conditions in $\ell_p$)
Let $1 < p < \infty$. If $E\subseteq \ell_p$ and $S_{E,\min(p,2)}(\scG) < \infty$ for some multiresolution family $\scG$ for $E$ with inflation factor $A_\scG \geq 240$, then $E$ is contained in a curve $\Gamma$ in $\ell_p$ with
\begin{equation*}
    \sH^1(\Gamma) \lesssim_{p,A_{\scG}} S_{E,\min(p,2)}(\scG).
\end{equation*}
The exponent $\min(p,2)$ on beta numbers is sharp.
\end{thm}
\begin{thm}[\cite{BM22a} Theorem 1.7](sharp necessary conditions in $\ell_p$)
Let $1 < p < \infty$. If $\Sigma\subseteq \ell_p$ is a connected set and $\scH$ is a multiresolution family for $\Sigma$ with inflation factor $A_\scH > 1$, then
\begin{equation*}
    S_{\Sigma,\max(2,p)}(\scH) \lesssim_{p,A_{\scH}} \sH^1(\Sigma).
\end{equation*}
The exponent $\max(2,p)$ on beta numbers is sharp.
\end{thm}
Badger and McCurdy's results give a similar proof of Theorem \ref{t:hstst} by taking the case $p=2$ in their above results.
\begin{remark}[Banach space Jordan arcs]
Given the results of this paper, it is natural to ask whether there is any analogue of Theorems \ref{t:thmA} and \ref{t:thmB} in $\ell_p$. That is, for a Jordan arc $\Gamma\subseteq\ell_p$ and multiresolution family $\scH$ for $\Gamma$, could one show that
\begin{equation*}
     S_{\Gamma,\max(2,p)}(\scH) \lesssim_{p,A_{\scH}} \ell(\Gamma) - \crd(\Gamma)
\end{equation*}
or
\begin{equation*}
    S_{\Gamma,\min(p,2)}(\scH) \gtrsim_{p,A_{\scH}} \ell(\Gamma) - \crd(\Gamma)
\end{equation*}
by combining the methods of \cite{BM22a}, \cite{BM22b} and those given here? If these inequalities do not hold, can one find a different geometric function of the endpoints of $\Gamma$ which could replace $\crd(\Gamma)$? What about for general rectifiable arcs?
\end{remark}
\end{subsubsection}

\begin{subsubsection}{Traveling salesman in general metric spaces}
Some success has also been achieved in the setting of general metric spaces by Hahlomaa \cite{Ha05} and David and Schul \cite{DS21}. Because there is no ambient linear structure in a general metric space which one can use to define the standard Jones beta number, the work in metric spaces uses replacements which directly measure the triangle inequality excess. Hahlomaa originally defined a general \textit{metric beta number} using the notion of Menger curvature, but this definition is equivalent to the following given by David and Schul. Let $E$ be a metric space, $p\in E$ and $r > 0$. Let $Q=B(p,r)$ and define the \textit{metric beta number} by
\begin{align*}
    \beta_\infty^E(Q)^2 \vcentcolon= r^{-1}\sup\{&\dist(x,y) + \dist(y,z) - \dist(z,x):\\ &x,y,z\in E\cap B(p,r)
    \text{ and } \dist(z,y) \leq \dist(y,z) \leq \dist(z,x)\}.
\end{align*}
If $E$ is $\ell_2$, then this is proportional to the normalized length difference between the line segment $[x,z]$ and its perturbed version given by $[x,y]\cup[y,z]$. The exponent $2$ is added in the definition as a convention to preserve the form of Theorem \ref{t:Rn-tst}. Hahlomaa was the first to give a sufficient condition in general metric spaces:
\begin{thm}[\cite{Ha05} Theorem 5.3]
Let $E$ be a metric space and let $\scG$ be a multiresolution family for $E$ with inflation factor $A\simeq 1$. If
\begin{equation*}
    \beta_\infty^E(\scG)\vcentcolon= \diam(E) + \sum_{Q\in\scG}\beta_\infty^E(Q)^2\diam(Q) < \infty,
\end{equation*}
then there exists a set $F\subseteq[0,1]$ and a surjective Lipschitz map $f:F\rightarrow E$ with Lipschitz constant $\mathrm{\mathop{Lip}}(f)\lesssim \beta_\infty^E(\scG)$.
\end{thm}
See \cite{Sc07b} Example 3.3.1 for a counterexample to the converse to Hahlomaa's result in $\R^2$ with the $\ell_1$ metric. Schul notes however that this counterexample is not fully satisfactory, as Hahlomaa's result can be strengthened, for instance, by defining the metric beta number to be a supremum taken over more restrictive triples. In any case, David and Schul have recently achieved a partial converse to this result. Their result concerns \textit{doubling} metric spaces. We say that a metric space is doubling if there exists a constant $N$ such that every ball of radius $r > 0$ can be covered by at most $N$ balls of radius $\frac{r}{2}$.
\begin{thm}[\cite{DS21} Theorem A]
Let $\Sigma$ be a connected doubling metric space with doubling constant $N$ and let $\scH$ be a multiresolution family for $\Sigma$ with inflation factor $A > 1$. For every $\epsilon > 0$,
\begin{equation*}
    \diam(Q) + \sum_{Q\in\scH}\beta_\infty^\Sigma(Q)^{2+\epsilon}\diam(Q) \lesssim_{\epsilon,A,N}\sH^1(\Sigma).
\end{equation*}
\end{thm}
The authors conjecture that the doubling hypothesis can be dropped by utilizing the techniques of \cite{Sc07}.
\begin{remark}[Metric space Jordan arcs]
It would again be interesting to know whether these results could be strengthened in the special case of a Jordan arc. That is, let $\Gamma$ be a metric space which is the image of a continuous injective map $\gamma:[0,1]\rightarrow\Gamma$. Suppose $\scG$ is a multiresolution family for $\Gamma$. Do the na\"{i}ve inequalities
\begin{equation*}
    \sum_{Q\in\scG}\beta_\infty^\Gamma(Q)^2\diam(Q) \gtrsim_A \ell(\Gamma) - \crd(\Gamma)
\end{equation*}
or, for $\Gamma$ with doubling constant $N$,
\begin{equation*}
    \sum_{Q\in\scG}\beta_\infty^\Gamma(Q)^{2+\epsilon}\diam(Q) \lesssim_{\epsilon,A,N}\ell(\Gamma) - \crd(\Gamma)
\end{equation*}
hold? As our methods rely heavily on linear structure, these seem further from proof than the proposed extension to Banach space. But even if these do not hold, can one find a different geometric function of the endpoints of $\Gamma$ to replace $\crd(\Gamma)$ in the equations above? What about for general rectifiable arcs?
\end{remark}
\end{subsubsection}
\end{subsection}
\begin{subsection}{Acknowledgments}
The author would first like to thank Raanan Schul for many helpful discussions. This result would not have been possible without both Raanan's edits of several early drafts and the author's numerous conversations with Raanan about \protect\cite{Sc07} and the new ideas introduced here. 

This result would also not exist without the wonderful course ``Topics in Real Analysis: The Weil-Petersson class, traveling salesman theorems and minimal surfaces in hyperbolic space" given by Chris Bishop in Fall 2020 at Stony Brook University in which the results from \protect\cite{Bi20} and \protect\cite{Bi22} were presented, and the problem solved here was first presented to the author.

The author would finally like to thank the anonymous referee for a multitude of insightful comments which greatly improved the content and style of the paper. The author also thanks the referee for also pointing out the problems with several aspects of Schul's original proof of \cite{Sc07} Lemma 3.28 and encouraging the author to write out the details of Schuls' arguments in full.
\end{subsection}

\begin{subsection}{Preliminaries}
\begin{subsubsection}{Parameterizations of finite-length continua and arcs associated to a parameterization in Hilbert space}
From this point on, fix a connected, compact set $\Sigma\subseteq\ell_2$ and a rectifiable Jordan arc $\Gamma\subseteq\ell_2$. We are guaranteed that $\Gamma$ has an injective arc length parameterization $\gamma:[0,\ell(\Gamma)]\rightarrow\Gamma$ by definition. It is a vital fact that we also have access to an arc length parameterization of $\Sigma$. We deduce the existence of this map as a consequence of the more general results on parameterization of finite-length continua in metric spaces carried out by Alberti and Ottolini \cite{AO17}.

Let $X$ be a metric space and $I = [a,b]\subseteq \R$ be a closed interval. Following Alberti and Ottolini, for a continuous map $\gamma:I\rightarrow X$ (often referred to as a \textit{path}) and a point $x\in X$, define the \textit{multiplicity} of $\gamma$ at $x$ as
\begin{equation*}
    m(\gamma,x)\vcentcolon=\#(\gamma^{-1}(x))
\end{equation*}
where for any set $A$, $\#A$ denotes the cardinality of $A$. We define the \textit{length} of $\gamma$ as
\begin{equation}\label{e:param-length}
    \ell(\gamma) =  \ell(\gamma,I)\vcentcolon= \sup\left\{ \sum_{i=1}^n d(\gamma(t_{i-1}),\gamma(t_i)) : n\geq0,\ t_0 < t_1 < \ldots < t_n,\ t_j\in I \text{ for all $j$}\right\}.
\end{equation}
Additionally, $\gamma$ has \textit{constant speed} if there exists a finite constant $c$ such that
\begin{equation*}
    \ell(\gamma,[t_0,t_1]) = c(t_1-t_0) \text{ for every $[t_0,t_1]\subseteq I$}.
\end{equation*}
We will refer to $\gamma$ as an \textit{arc length} parameterization if $\gamma$ has constant speed with $c = 1$. We will only consider constant speed parameterizations, and given a fixed parameterization $\gamma:I\rightarrow\Sigma$ with constant speed $c$, we define a finite Borel measure $\ell$ supported on $\Sigma$ by
\begin{equation*}
    d\ell\vcentcolon= c\gamma_*(dt)
\end{equation*}
where $\gamma_*$ denotes the pushforward measure so that $\ell(A) = c\int_{\gamma^{-1}(A)}dt$. Alberti and Ottolini prove the following general parameterization result:
\begin{thm}[\cite{AO17} Theorem 4.4]\label{t:AO-param}
Let $X$ be a connected, compact metric space with $\sH^1(X) < \infty$. Then there exists a path $\gamma:[0,1]\rightarrow X$ with the following properties:
\begin{enumerate}[label=(\roman*)]
    \item $\gamma$ is closed, Lipschitz, surjective, and has degree zero;
    \item $m(\gamma,x) = 2$ for $\sH^1$-a.e. $x\in X$, and $\ell(\gamma) = 2\sH^1(X)$; and,
    \item $\gamma$ has constant speed, equal to $2\sH^1(X)$.
\end{enumerate}
\end{thm}
See \cite{AO17} Section 4.1 for the definition of degree zero. (Essentially, the path passes through almost every point thee same number of times in one direction as in the opposite direction.) Fix a multiresolution family $\scH$ for $\Sigma$, and let $\gamma$ be a constant speed parameterization of $\Sigma$, the existence of which is guaranteed by Theorem \ref{t:AO-param}. We will use $\gamma$ to properly study the geometry of $\Sigma$ inside of the balls of $\scH$.

\begin{defn}[Arcs]
We define an \textit{arc} $\tau\vcentcolon=\gamma|_{[a,b]}$ to be the restriction of $\gamma$ to a subinterval $[a,b]\subseteq I$. Given a ball $Q\subseteq\ell_2$, we define the family of arcs of $\Sigma$ inside $Q$ as
\begin{equation*}
    \Lambda(Q) \vcentcolon= \{\gamma|_{[a,b]}:[a,b]\subseteq[0,1],\ [a,b]\ \text{is a connected component of } \gamma^{-1}(2Q\cap\Sigma)\}.
\end{equation*}
These are arcs inside $2Q$ which intersect $Q$ in the style of \cite{BM22a}. Fix an arc $\tau$ as above. Further following \cite{BM22a}, we use bold terms to refer to operators acting on arcs and define
\begin{equation*}
    \Domain(\tau) \vcentcolon= [a,b],\ 
    \Image(\tau) \vcentcolon= \tau(\Domain(\tau)),\ 
    \Diam(\tau) \vcentcolon= \diam(\Image(\tau)),
\end{equation*}
\begin{equation*}
    \Edge(\tau) \vcentcolon= [\tau(a),\tau(b)],\ 
    \Line(\tau) \vcentcolon= \{\tau(a) + t(\tau(b) - \tau(a)) : t\in\R\},\ 
    \Crd(\tau) \vcentcolon= |\tau(a) - \tau(b)|,
\end{equation*}
\begin{equation*}
    \Start(\tau) \vcentcolon= \tau(a),\ \End(\tau)\vcentcolon= \tau(b).
\end{equation*}
where $[\tau(a),\tau(b)]\subseteq\ell_2$ is the line segment connecting the endpoints of $\tau$, and hence $\Line(\tau)$ is the line passing through the endpoints of $\tau$. We will often use the term \textit{arc} to refer to both $\tau$ and $\Image(\tau)$, but the referent should be clear from context. If $\xi = \gamma|_{[c,d]}$ for $[c,d]\subseteq[a,b]$, then we will often call $\xi$ a \textit{subarc} of $\tau$. For two general arcs $\xi$ and $\tau$ we define shorthand notation by defining (in the sense of logical formulas)
\begin{align*}
    (\xi \subseteq \tau) \vcentcolon= (\Domain(\xi)\subseteq\Domain(\tau)) \text{ and } (x\in\tau) \vcentcolon= (x\in \Image(\tau)),
\end{align*}
and we define
\begin{align*}
        \xi\cap\tau \vcentcolon= \gamma|_{\Domain(\tau)\cap\Domain(\xi)} \text{ and } \xi\cup\tau \vcentcolon= \gamma|_{\Domain(\tau)\cup\Domain(\xi)}.
\end{align*}
If $E\subseteq \ell_2$ and $\mu$ is a Borel measure on $\ell_2$ then we set
\begin{align*}
    \tau\cap E \vcentcolon= \Image(\tau)\cap E \text{ and } \mu(\tau) \vcentcolon= \mu(\Image(\tau)). 
\end{align*}
\end{defn}

\begin{defn}[Almost flat and non-flat arcs]
In order to measure the flatness of an arc, we define the \textit{arc beta number}
\begin{equation*}
    \tilde{\beta}(\tau) \vcentcolon= \sup_{x\in\Image(\tau)}\frac{\text{dist}(x,\Edge(\tau))}{\Diam(\tau)}.
\end{equation*}
We also set
\begin{align*}
    \Max(\tau) &\vcentcolon= \{y\in\Image(\tau) : \tb{\tau}\Diam(\tau) = \dist(y,\Edge(\tau))\}\not=\varnothing,\\
    \Drift(\tau) &\vcentcolon= \tb{\tau}\Diam(\tau) = \dist(y,\Edge(\tau)) \text{ for any $y\in\Max(\tau)$}.
\end{align*}
Fix a constant $\epsilon_2 > 0$ whose specific value will be set in Section \ref{sec:constants} Given a ball $Q\in\scH$, we define the set of \textit{almost flat arcs} for $Q$ as
\begin{equation*}
    S(Q) \vcentcolon= \{\tau\in\Lambda(Q):\tilde{\beta}(\tau)\leq\epsilon_2\beta_\Sigma(Q)\}
\end{equation*}
and refer to any arc $\tau\in S(Q)$ as an \textit{almost flat arc}. We will commonly refer to $\Lambda(Q)\setminus S(Q)$ as the set of \textit{non-flat arcs}. For any collection of arcs $\mathscr{T}$, we define
\begin{equation*}
    \beta_{\mathscr{T}}(Q) \vcentcolon= \beta_{\cup_{\tau\in\mathscr{T}}\Image(\tau)}(Q),
\end{equation*}
and for a single arc $\tau$ we set $\beta_\tau(Q)\vcentcolon=\beta_{\Image(\tau)}(Q)$.
\end{defn}
Considering different configurations of almost flat and non-flat arcs will give us useful ways of classifying balls $Q\in \scH$. For $\eta\in S(Q)$, we get that $\Image(\eta)$ lies \emph{very} close to $\Edge(\eta)$ on the scale of $\Diam(\eta) \simeq \diam(Q)$, so that in many cases one can think of $\eta$ as a line segment. The parameter $\epsilon_2$ will be fixed small enough such that this approximation will work well on all small scales relative to $\diam(Q)$ which are relevant to our almost flat analysis in Section \ref{sec:almost-flat-arcs}. 
\end{subsubsection}
\begin{subsubsection}{The division of $\scH$}

\begin{figure}[h]
    \centering
    \includegraphics[scale=0.5]{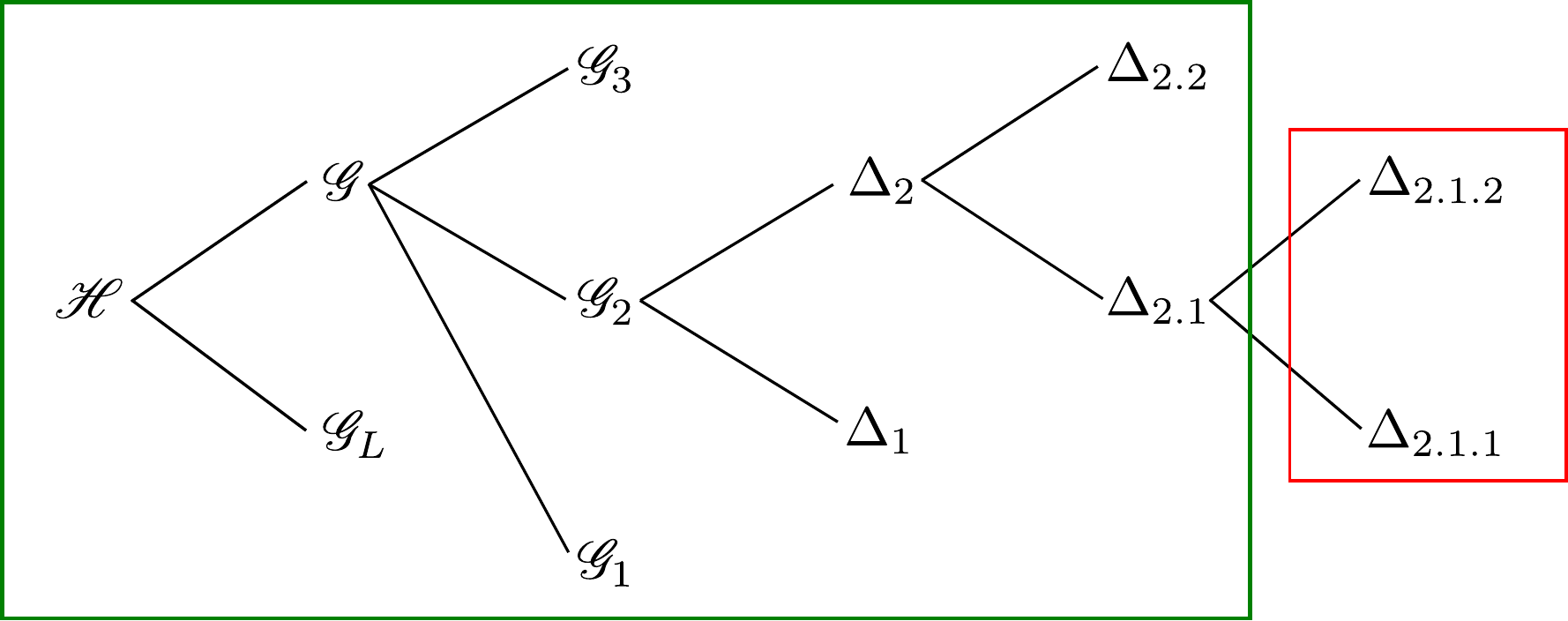}
    \caption{A tree denoting the subfamilies of the multiresolution family $\scH$. Those in the green rectangle were considered by Schul in his proof of Theorem \ref{t:hstst}. The in the red rectangle are new introductions made in the proof of Theorem \ref{t:thmA}}.
    \label{fig:CollectionDiagram}
\end{figure}
We begin classifying balls of $\scH$ based on their geometry by splitting off the large balls and balls with $\beta_\Sigma(Q) = 0$: Define
\begin{align*}
    \scG_L &\vcentcolon= \{Q\in \scH : \Gamma \cap (H\setminus12Q)=\varnothing \text{ or } \beta_\Sigma(Q) = 0\},\\
    \scG &\vcentcolon= \scH\setminus\scG_L.
\end{align*}
Next, we extend the family $\scG$ by considering $\scG^\lambda = \{\lambda Q\ :\ Q\in\scG\}$ for $\lambda\in\{1,2,8,12\}$ together. For any ball $Q\in \scG^1\cup\scG^2\cup\scG^8\cup\scG^{12}$, choose a subarc $\gamma_Q\ni x_Q$ such that we always have $\gamma_Q\subseteq\gamma_{2Q}\subseteq\gamma_{8Q}\subseteq\gamma_{12Q}$. Let $\epsilon_1 > 0$ be small (to be fixed in Section \ref{sec:constants}) and partition $\scG^\lambda = \scG_1^\lambda \cup \scG_2^\lambda \cup \scG_3^\lambda$ where
\begin{align*}
    \scG_1^\lambda &= \{Q\in\scG: \tilde{\beta}(\gamma_{\lambda Q}) > \epsilon_2\beta(\lambda Q)\},\\
    \scG_2^\lambda &= \{Q\in\scG: \tilde{\beta}(\gamma_{\lambda Q}) \leq \epsilon_2\beta_\Sigma(\lambda Q) ; \beta_{S_{\lambda Q}}(\lambda Q) > \epsilon_1\beta_\Sigma(Q)\},\\
    \scG_3^\lambda &= \{Q\in\scG: \tilde{\beta}(\gamma_{\lambda Q}) \leq \epsilon_2\beta_\Sigma(\lambda Q) ; \beta_{S_{\lambda Q}}(\lambda Q) \leq \epsilon_1\beta_\Sigma(Q)\}.
\end{align*}
\begin{figure}[h]
    \centering
    \includegraphics[scale=0.7]{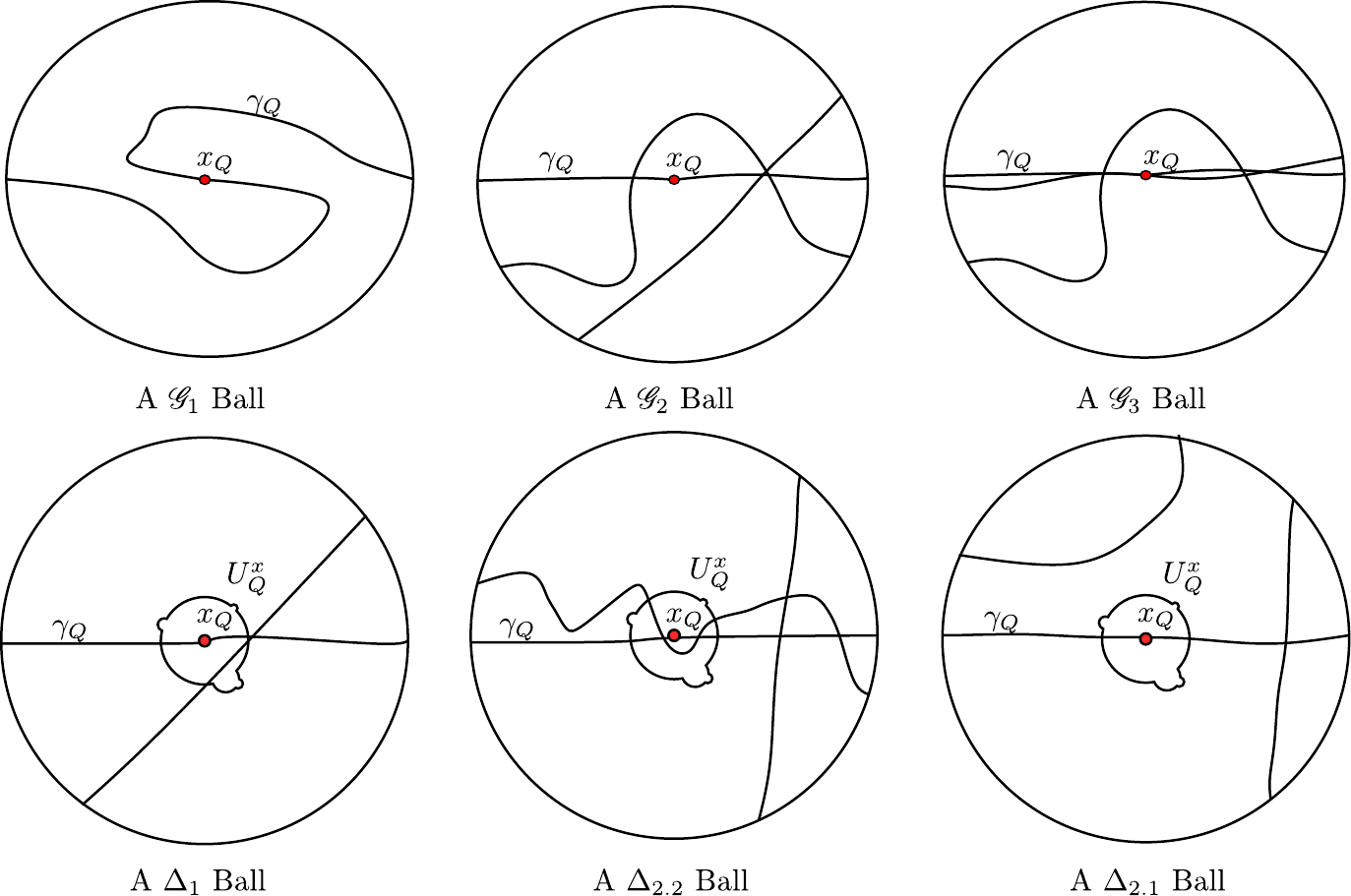}
    \caption{Examples of balls in $\scG_1,\scG_2,\scG_3,\Delta_1,\Delta_{2.1}$ and $\Delta_{2.2}$.}
    \label{fig:BallExamples}
\end{figure}
See Figure \ref{fig:CollectionDiagram} for examples. This decomposition is slightly different than Schul's in the style of $\cite{BM22a}$. Notice that for any $Q\in\scG^1\cup\scG^2\cup\scG^{8}\cup\scG^{12}$, either
\begin{enumerate}[label=(\roman*)]
    \item $Q\in\bigcup_{\lambda}\scG_1^\lambda\cup\scG_2^\lambda,$ \text{ or }
    \item $Q\in\bigcap_{\lambda}\scG_3^\lambda$.
\end{enumerate}
Hence, it suffices to consider the collections $\scG_1 = \cup_\lambda \scG_1^\lambda,\ \scG_2 = \cup_\lambda\scG_2^\lambda$, and $\scG_3 = \cap_\lambda \scG_3^\lambda$, and our total family of collections of balls is now
\[\{\scG_L, \scG_1, \scG_2, \scG_3\}.\]
The collections $\scG_L$, $\scG_1$, and $\scG_3$ will be handled as they are, but $\scG_2$ needs further refinement. To describe the refinement of $\scG_2$, we will define a ball-like set associated to each ball $Q\in\scG$ called its ``core''.

\begin{defn}[Cores \cite{Sc07}]
Let $Q\in\scG$ with $Q = B(x_Q,A2^{-n})$. For any $c\in\R,\ 0 < c < \frac{1}{4A}$ and $J\in\Z,\ J \geq 10$, define $U_Q^{c,J,0} \vcentcolon= cQ$. Let $i \geq 0$ and set
\begin{equation*}
    U_Q^{c,J,i+1}\vcentcolon= U_Q^{c,J,i} \cup \bigcup_{\substack{x_{Q'}\in X_{n+J(i+1)} \\ cQ'\cap U_Q^{c,J,i}\not=\varnothing}} cQ'.
\end{equation*}
We then define the \textit{core} of $Q$ with dilation factor $c$ and scaling factor $J$ to be
\begin{equation*}
    U^{c,J}_Q \vcentcolon= \bigcup_{i\geq 0} U_Q^{c,J,i}.
\end{equation*}
For $J$ as fixed in Section \ref{sec:constants} and $c_0\vcentcolon=\frac{1}{64A}$, we define three successively larger cores for $Q$ as
\begin{align*}
    U_Q \vcentcolon= U_Q^{c_0,J}, \quad U_Q^x \vcentcolon= U_{Q}^{8c_0,J}, \quad U_Q^{xx} \vcentcolon= U_Q^{16c_0,J}.
\end{align*}
These are the concrete families we will work with.
\end{defn}
The cores have nice separation and inclusion properties which will allow us to work with some families of balls more easily. These are given in the following proposition:
\begin{prop}[properties of core families, cf. \cite{Sc07} Lemma 3.19]\label{p:cores}
Let $J \geq 10$ and $c < \frac{1}{4A}$. Fix $1 \leq j \leq J$ and define
\begin{equation*}
    \scQ_j \vcentcolon= \{Q\in\scG : Q = B(x_Q,A2^{-n}),\ x\in X_n,\ n\equiv j \mod J\}.
\end{equation*}
Let $Q,Q'\in \scQ_j$, with $Q = B(x_Q,A2^{-n}),\ Q' = B(x_{Q'},A2^{-m})$ and corresponding cores $U^c_Q\vcentcolon= U_Q^{c,J},\ U^c_{Q'}\vcentcolon= U_{Q'}^{c,J}$. Then
\begin{enumerate}[label=(\roman*)]
    \item $cQ \subseteq U^c_Q \subseteq (1+2^{-J+2})cQ$,\label{i:core-diam}
    \item If $n=m$, then either $Q = Q'$ or $\dist(U^c_{Q},U^c_{Q'}) \geq 2^{-n-1}$, \text{and}\label{i:core-sep}
    \item If $n > m$ and $U^c_Q\cap U^c_{Q'}\not=\varnothing$, then $U^c_{Q'}\subsetneq U^c_Q$.\label{i:core-nest}
\end{enumerate}
\begin{proof}
For a proof of \ref{i:core-diam}, one can apply \cite{Sc07b} Lemma 2.16 to show that any point $y\in U_Q$ satisfies
\begin{equation*}
    \dist(y,U_Q) \leq c\rad(Q)\sum_{k=0}^\infty(2\cdot2^{-J})^k \leq (1+2^{-J+2})c\rad(Q).
\end{equation*}
For \ref{i:core-sep}, notice that because $x_Q,x_{Q'}\in X_n$, we know $|x_Q - x_{Q'}| \geq 2^{-n}$. This combined with the second inclusion in \ref{i:core-diam} completes the proof. Property \ref{i:core-nest} follows from the definition.
\end{proof}
\end{prop}
\begin{remark}
Property \ref{i:core-diam} above implies the following containments:
\begin{align*}
    c_0Q \subseteq\ &U_Q \subseteq (1 + \epsilon_1)c_0Q \subseteq 2c_0Q,\\
    8c_0Q \subseteq\ &U_Q^x \subseteq (1 + \epsilon_1)8c_0Q \subseteq 9c_0Q,\\
    16c_0Q \subseteq\ &U_Q^{xx} \subseteq (1 + \epsilon_1)16c_0Q \subseteq 17c_0Q,
\end{align*}
where the penultimate containment in each line follows from the fact that $2^{-J+2} < \epsilon_1$ (see Section \ref{sec:constants}). It is usually best to think of cores as small perturbations of balls.
\end{remark}
Using the cores, we can now refine the family $\scG_2$. Let $C_U > 0$ and define
\begin{align*}
    \Delta_1 &\vcentcolon= \{Q\in\scG_2 : C_U\beta_{S(Q)}(U_Q^x) > \beta_{S(Q)}(Q)\},\\
    \Delta_{2} &\vcentcolon= \scG_2 \setminus \Delta_1.
\end{align*}
The constant $C_U$ will be fixed in Section \ref{sec:constants}. We further divide $\Delta_2$ by defining
\begin{align*}
    \Delta_{2.2} &\vcentcolon= \{Q\in\Delta_2 : \exists\tau\in\Lambda(Q)\setminus S(Q),\ \tau\cap U_Q^x\not= \varnothing\},\\
    \Delta_{2.1} &\vcentcolon= \Delta_2 \setminus \Delta_{2.2}.
\end{align*}
We note here that it suffices to assume $\scG_2\subseteq\scG_2^1$:
\begin{remark}[Reduction of $\scG_2$ to $\scG^1_2$]
Suppose that we have proven the inequality
\begin{align}\label{e:G2-reduction}
    \sum_{Q\in\scG_1}\beta_\Sigma(Q)^2\diam(Q) + \sum_{Q\in\scG_3}\beta_\Sigma(Q)^2\diam(Q) + \sum_{Q\in\scG^1_2}\beta_\Sigma(Q)^2\diam(Q) \lesssim_A \sH^1(\Sigma)
\end{align}
for any multiresolution family with arbitrary inflation factor $A > 200$. Recall that $Q\in\scG_2$ implies $Q\in\scG_2^\lambda$ for some $\lambda\in\{1,2,8,12\}$. We would like to show that $\sum_{Q\in\scG^\lambda_2}\beta_\Sigma(Q)^2\diam(Q) \lesssim_A \sH^1(\Sigma)$. Let $\lambda\scH$ be the multiresolution family with the same net points as $\scH$ but inflation factor $\lambda A$. Define $\hat{Q} \vcentcolon= \lambda Q$. Then $\hat{Q}\in\lambda\scH$ and $\tb{\gamma_{\hat{Q}}} \leq \epsilon_2\beta_\Sigma(\hat{Q})$. If $\beta_{S(\hat{Q})}(\hat{Q}) > \epsilon_1\beta_\Sigma(\hat{Q})$, then $\hat{Q}\in \scG_2^1(\lambda \scH)$, the subfamily $\scG^1_2$ defined relative to the multiresolution family $\lambda \scH$. Otherwise, $\beta_{S(\hat{Q})}(\hat{Q}) \leq \epsilon_2\beta_\Sigma(\hat{Q})$, implying $\hat{Q}\in\scG^1_3(\lambda \scH)\subseteq \scG_1(\lambda\scH)\cup\scG_3(\lambda\scH)$. Therefore, the desired inequality follows from \eqref{e:G2-reduction} applied to the multiresolution family $\lambda\scH$. A similar argument shows that the same reduction holds for $\Gamma$ where the right side of \eqref{e:G2-reduction} is replaced by $\ell(\Gamma) - \crd(\Gamma)$.
\end{remark}
With this remark, we are justified in assuming $\scG_2\subseteq\scG_2^1$ and need not worry about factors of $\lambda$ in our analysis of $\scG_2$ in Section \ref{sec:almost-flat-arcs}. See examples of balls in these families in Figure \ref{fig:BallExamples}. These are all of the subcollections necessary to prove \eqref{e:HS-nec}, the Hilbert space necessary condition. When we restrict to the case of a Jordan arc, we will need to further divide the family $\Delta_{2.1}$. This is carried out in Section \ref{subsec:almost-flat-arcs-gamma} (also see Figure \ref{fig:CollectionDiagram} for a full diagram of the divisions).
\end{subsubsection}
\end{subsection}

\begin{subsection}{Constants}\label{sec:constants}
In this section, we fix the values of constants used in the proof of Theorem \ref{t:thmA} and give general descriptions of their purposes and where the values come from. We fix
\begin{align*}
    \epsilon_1 &\vcentcolon= 10^{-10},\\
    C_U &\vcentcolon= 100A\epsilon_1^{-1},\\
    J &\vcentcolon= -\log_2(10^{-3}\epsilon_1c_0),\\
    \epsilon_3 &\vcentcolon= (100A)^{-1}\epsilon_1^2,\\
    \epsilon_2 &\vcentcolon= \min((10^5AC_U)^{-1}\epsilon_1^2, 100^{-1}c_0\epsilon_3^2).
\end{align*}
We first fix $\epsilon_1$, a catch-all, small reference parameter. Next, we fix $C_U$, the constant used in the definition of $\Delta_1$. It is fixed small enough here to facilitate \eqref{e:delta1-betatilde} in the proof of Lemma \ref{l:delta1-cores}, ensuring that $\beta_{\gamma_Q}(U_Q^x) \lesssim_{\epsilon_1} \beta_{S(Q)}(U_Q^x)$. We now fix the ``jump'' parameter $J$. This is fixed large enough so that for any balls $Q,Q'\in\scQ_j$, the ``thinned'' family gotten by skipping $J$ scales in the multiresolution family $\scH$, $\diam(Q') < \diam(Q)$ implies
\begin{equation}
    \diam(2Q') \leq 2^{-J+1}\diam(Q) \leq 100^{-1}\epsilon_1c_0\diam(2Q) < \epsilon_1\diam(U_Q).
\end{equation}
This means any future generation ball $Q'$ is very small even on the scale of $U_Q$ (this is important in Section \ref{subsec:sigma-delta2.1}, for instance). We can also conclude from Proposition \ref{p:cores} that
\begin{equation*}
     c\diam(Q) \geq \frac{1}{1+2^{-J+2}}\diam(U^c_Q) \geq (1-\epsilon_1)\diam(U^c_Q), \text{ and }
\end{equation*}
\begin{equation*}
    \diam(U^c_Q) \leq (1+2^{-J+2})c\diam(Q) \leq  (1+\epsilon_1)c\diam(Q).
\end{equation*}
for $c < \frac{1}{4A}$. We next fix $\epsilon_3$, a constant introduced in Section \ref{subsec:almost-flat-arcs-gamma} to define the families $\Delta_{2.1.1}$ and $\Delta_{2.1.2}$. This constant is fixed small in terms of $\epsilon_1$ to facilitate the final estimate in the proof of Lemma \ref{l:straightcenter}. It is fixed small in terms of $A$ to ensure that $100\epsilon_3\diam(Q) < \frac{c_0}{4}\diam(Q)$ to facilitate the neighborhood inclusion needed in the proof of Lemma \ref{l:projection-packing}. The constant $\epsilon_2$ is fixed last. It is fixed small in terms of all of the previous parameters to ensure that almost flat arcs stay close to their edge segments on all of the needed scales, i.e., relative to $C_U^{-1}\epsilon_1\diam(U_Q)$ needed in estimates for $\Delta_1$ and relative to $\epsilon_3\diam(U_Q)$ needed in estimates for $\Delta_{2.1.2}$. We have not attempted to optimize these.
\end{subsection}
\end{section}
\begin{section}{Large-scale balls: \texorpdfstring{$\scG_L$}{}}\label{sec:large-balls}
The goal of this section is to prove the following proposition:
\begin{prop}[cf. \cite{Sc07} Lemma 3.9, cf. \cite{BM22a} Lemma 3.27]
We have 
\begin{equation}\label{e:gen-big-balls}
    \sum_{Q\in\scG_L}\beta_\Sigma(Q)^2\diam(Q) \lesssim_A \ell(\Sigma) \text{ and } \sum_{Q\in\scG_L}\beta_\Gamma(Q)^2\diam(Q) \lesssim_A \ell(\Gamma) - \crd(\Gamma).
\end{equation}
\end{prop}
\begin{proof}
We first prove the $\Sigma$ inequality in \eqref{e:gen-big-balls}. Since $\Sigma\subseteq 12Q$ for any $Q\in\scG_L$, we know $\diam(Q) \geq \frac{\diam(\Sigma)}{12}$. For $k \geq 0$, define
\begin{equation*}
    \mathscr{B}_k \vcentcolon= \left\{Q\in\scG_{L} : \frac{\diam(\Sigma)}{12}2^k \leq \diam(Q) < \frac{\diam(\Sigma)}{12}2^{k+1}\right\}
\end{equation*}
and let $N_k = \#\mathscr{B}_k$. The net spacing for $Q\in\scB_k$ must be at least $\frac{\diam(Q)}{2A} \geq \frac{\frac{\diam(\Sigma)}{12}2^k}{2A} \geq \frac{\diam(\Sigma)}{24A}$. Since $N_k$ is maximal when $\Sigma$ is a line segment with net points separated by distance greater than $\frac{\diam(\Sigma)}{24A}$ along length $\sH^1(\Sigma)$, we get
\begin{equation}\label{e:ballnum}
    N_k \leq 1 + \frac{24A\sH^1(\Sigma)}{\diam(\Sigma)} \leq \frac{48A\sH^1(\Sigma)}{\diam(\Sigma)}.
\end{equation}
Now, to estimate beta numbers, observe that for any ball $Q\in\scB_k$ we have the trivial bound $\beta_\Sigma(Q) \leq \frac{\diam(\Sigma)}{\diam(Q)} \leq 12\cdot2^{-k}$ so that 
\begin{equation*}
    \beta_\Sigma(Q)^2\diam(Q) \leq 144\cdot2^{-2k}\frac{\diam(\Sigma)}{12}2^{k+1} \leq 12\diam(\Sigma)2^{-k+1}
\end{equation*}
We now put this all together:
\begin{align*}
       \sum_{Q\in\scG_L}\beta_\Sigma(Q)^2\diam(Q) &\leq \sum_{k=0}^\infty\sum_{Q\in\scB_k}\beta_\Sigma(Q)^2\diam(Q)
       \leq \sum_{k=0}^\infty N_k\cdot(12\diam(\Sigma)2^{-k+1})\\
       &\lesssim_A \sH^1(\Sigma)\sum_{k=0}^\infty2^{-k} \lesssim \sH^1(\Sigma).  
\end{align*}
This completes the proof of the $\Sigma$ inequality in \eqref{e:gen-big-balls}. In order to prove the $\Gamma$ inequality in \eqref{e:gen-big-balls}, we first note that it suffices to assume
\begin{equation}\label{e:small-excess}
    \ell(\Gamma) - \crd(\Gamma) \leq \epsilon_1\ell(\Gamma).
\end{equation}
Indeed, otherwise \eqref{e:gen-big-balls} would imply
\begin{equation*}
    \sum_{Q\in\scG_L}\beta_\Gamma(Q)^2\diam(Q) \lesssim_A \ell(\Gamma) < \frac{1}{\epsilon_1}(\ell(\Gamma) - \crd(\Gamma))
\end{equation*}
as desired. The only modifications we need to the above proof for this case are improved upper bounds for $N_k$ and $\beta_\Gamma(Q)$. Our assumption \eqref{e:small-excess} implies
\begin{equation*}
    \ell(\Gamma) - \diam(\Gamma) \leq \ell(\Gamma) - \crd(\Gamma) \leq \epsilon_1 \ell(\Gamma) \implies \frac{\ell(\Gamma)}{\diam(\Gamma)} \leq \frac{1}{1-\epsilon_1} \leq 2
\end{equation*}
so that \eqref{e:ballnum} implies
\begin{equation*}
    N_k \leq \frac{48A\ell(\Gamma)}{\diam(\Gamma)} \leq 96A.
\end{equation*}
We now give a new estimate for $\beta_\Gamma(Q)$. Assume without loss of generality that the endpoints $x\text{ and }z$ of $\Gamma$ satisfy $x \vcentcolon= 0$ and $z \vcentcolon= \crd(\Gamma)e_1$ so that the chord segment of $\Gamma$ lies along the $e_1$ axis. Define $\pi:\ell_2\rightarrow\R$ to be the othogonal projection onto the $e_1$-axis and let $\pi^\perp:\ell_2\rightarrow\ell_2$ be the projection onto the orthogonal subspace of the $e_1$-axis. Let $y\in\Gamma$ be a point satisfying
\begin{equation*}
    |\pi^{\perp}(y)| = \sup_{u\in\Gamma}|\pi^{\perp}(u)|.
\end{equation*}
Define $b \vcentcolon= \pi(y)e_1$. The two triples of points $x,b,y$ and $z,b,y$ form right triangles with common altitude length $d\vcentcolon= |\pi^\perp(y)| = |y-b|$. Let $a_1\vcentcolon= |x-b|,\ a_2\vcentcolon=|b-z|$ be the lengths of the bases of these triangles and let $c_1\vcentcolon=|x-y|,\ c_2\vcentcolon=|y-z|$ be the lengths of their hypotenuses (See Figure \ref{fig:Pythagorean-theorem} for a picture). Applying the Pythagorean theorem to each triangle gives
\begin{align*}
    d^2 &= c_i^2 - a_i^2 = (c_i-a_i)(c_i+a_i) \leq 2\diam(\Gamma)(c_i-a_i)
\end{align*}
for $i = 1,2$ . Summing these inequalities over $i$ gives
\begin{align*}
    d^2 &= \diam(\Gamma)(c_1+c_2-a_1-a_2) \leq 2\diam(\Gamma)(\ell(\Gamma) - \crd(\Gamma))
\end{align*}
where we used the fact that $a_1 + a_2 \geq \crd(\Gamma)$ and $c_1 + c_2 \leq \ell(\Gamma)$ because $\Gamma$ is a connected set containing $x,y,$ and $z$. Now, if $Q\in\scB_k$, then $\diam(Q) \geq \frac{\diam(\Gamma)}{12}2^k$ and the definition of $\beta_\Gamma(Q)$ implies $\beta_\Gamma(Q) \leq \frac{d}{\diam(Q)}$ using the $e_1$ axis as an approximating line. This means
\begin{equation}\label{e:crd-beta-bound}
    \beta_\Gamma(Q)^2\diam(Q) \leq \frac{d^2}{\diam(Q)} \leq 2\diam(\Gamma)(\ell(\Gamma) - \crd(\Gamma))\cdot\frac{12\cdot2^{-k}}{\diam(\Gamma)}\leq 24(\ell(\Gamma) - \crd(\Gamma))2^{-k}.
\end{equation}
Therefore, 
\begin{align*}
    \sum_{Q\in\scG_L}\beta_\Gamma(Q)^2\diam(Q) &\leq \sum_{k=0}^\infty\sum_{Q\in\scB_k}\beta_\Gamma(Q)^2\diam(Q)
    \leq \sum_{k=0}^\infty N_k\cdot(24(\ell(\Gamma) - \crd(\Gamma))2^{-k})\\
    &\lesssim_A (\ell(\Gamma) - \crd(\Gamma))\sum_{k=0}^\infty 2^{-k} \lesssim \ell(\Gamma) - \crd(\Gamma).  
\qedhere\end{align*}
\end{proof}
\end{section}
\begin{section}{Non-flat arcs: \texorpdfstring{$\scG_1,\scG_3,\Delta_{2.2}$}{}} \label{sec:non-flat-arcs}
The goal of this section is to prove the following proposition:
\begin{prop}\label{p:non-flat}
Set $\mathscr{N}\vcentcolon= \scG_1 \cup \scG_3 \cup \Delta_{2.2}$. We have
\begin{equation}\label{e:param-non-flat}
    \sum_{Q\in\scN}\beta_{\Sigma}(Q)^2\diam(Q) \lesssim_A \ell(\gamma) - \crd(\gamma).
\end{equation}
In particular,
\begin{equation}\label{e:non-flat-balls}
    \sum_{Q\in\scN}\beta_\Sigma(Q)^2\diam(Q) \lesssim_A \sH^1(\Sigma) \text{ and } \sum_{Q\in\scN}\beta_\Gamma(Q)^2\diam(Q) \lesssim_A \ell(\Gamma) - \crd(\Gamma).
\end{equation}
\end{prop}
\begin{remark}
Both inequalities in \eqref{e:non-flat-balls} follow from \eqref{e:param-non-flat}. For the first, Theorem \ref{t:AO-param} gives a parameterization $\gamma$ of $\Sigma$ such that $\ell(\gamma) \leq 2\sH^1(\Sigma)$. For the second, $\Gamma$ comes with an injective parameterization $\gamma$ for which $\ell(\Gamma) = \ell(\gamma)$ and $\crd(\Gamma) = \crd(\gamma)$ by definition.
\end{remark}

Recall that $\scN$ consists of balls $Q$ which have $\beta_\Sigma(Q)\lesssim_{\epsilon_2} \tb{\tau_Q}$ for some $\tau_Q\in\Lambda(Q)$. That is, their beta number is dominated by the beta-tilde number of some arc they contain. Our strategy to prove \eqref{e:non-flat-balls} is to construct an appropriate mapping $Q\mapsto \tau_Q$ and prove that the associated sum $\sum_{Q\in\scN}\tb{\tau_Q}^2\diam(Q)$ is controlled. The first subsection below develops the general method for building an appropriate mapping and proving that the associated sum is controlled, while the second subsection applies the results of the first to proving \eqref{e:non-flat-balls}.

\begin{subsection}{Filtration construction and properties of \texorpdfstring{$\tilde{\beta}$ }{}}
It turns out to be most appropriate to derive bounds for sums over $\tb{\tau_Q}$ by only considering certain nice families of arcs called \textit{filtrations}. 
\begin{defn}[Filtrations \cite{Ok92}, \cite{Sc07}]
A \textit{filtration} of $\gamma$ is a family of subarcs $\scF = \bigcup_{n=0}^\infty\scF_n$ of $\gamma$ of whose constituent subfamilies $\{\scF_n\}_{n\geq0}$ satisfy the following:
\begin{enumerate}[label=(\roman*)]
    \item For all $\tau'\in\scF_{n+1}$, there exists a unique $\tau\in\scF_n$ such that $\Domain(\tau')\subseteq\Domain(\tau)$, \label{i:fil-tree}
    \item There exist constants $\underline{A},A > 0,\ \rho < 1$ such that for all $n\geq 0$ and $\tau\in\scF_n$, $\underline{A}\rho^{-n} \leq \Diam(\tau) \leq A\rho^{-n}$,\label{i:fil-diam}
    \item For all $\tau,\tau'\in\scF_n$, either $\tau = \tau'$ or $\#(\Domain(\tau)\cap\Domain(\tau')) \leq 1$,\text{ and} \label{i:fil-overlap}
    \item $\bigcup_{\tau\in\scF_0} \Domain(\tau) = \bigcup_{\tau'\in\scF_n} \Domain(\tau')$.\label{i:fil-partition}
\end{enumerate}
\end{defn}
We are interested in constructing filtrations with constituent arcs associated to subfamilies of $\scN$ because of the following lemma:
\begin{lem}\protect(\cite{Ok92}, cf. \cite{Sc07} Lemma 3.11)\label{l:nonflat}
Let $\mathscr{F}$ be a filtration for $\gamma$. Then
\begin{equation}\label{e:fil-sigma}
    \sum_{\tau\in\mathscr{F}}\tilde{\beta}(\tau)^2\Diam(\tau) \lesssim_A \ell\left(\bigcup_{\tau\in\mathscr{F}_0}\tau\right) - \sum_{\tau\in\mathscr{F}_0}\Crd(\tau).
\end{equation}
If $\bigcup_{\tau\in\scF_0}\tau = \gamma$, then
\begin{equation}\label{e:fil-gamma}
    \sum_{\tau\in\mathscr{F}}\tilde{\beta}(\tau)^2\Diam(\tau) \lesssim_A \ell(\gamma) - \crd(\gamma).
\end{equation}
\end{lem}
\begin{proof}
We refer the reader to the proof of Lemma 3.11 in \cite{Sc07} for the the proof of \eqref{e:fil-sigma}. In order to prove \eqref{e:fil-gamma}, we follow Schul's aforementioned proof to the second to last equation of page 349. Summing this equation over $n$, we replace the first equation on page 350 with
\begin{equation*}
    \sum_{\tau\in\mathscr{F}}\frac{d_\tau^2}{\Diam(\tau)} \lesssim \sup_n\sum_{\tau\in\mathscr{F}_n}\sH^1(I_\tau) - \sum_{\tau\in\mathscr{F}_0}\sH^1(I_\tau) \lesssim \ell\left(\bigcup_{\tau\in\mathscr{F}_0}\tau\right) - \sum_{\tau\in\mathscr{F}_0}\Crd(\tau).
\end{equation*}
Finally, replace the following occurrences of $\ell\left(\bigcup_{\tau\in\mathscr{F}_0}\tau\right)$ on page 350 with $\ell\left(\bigcup_{\tau\in\mathscr{F}_0}\tau\right) - \sum_{\tau\in\mathscr{F}_0}\Crd(\tau)$. The result follows from the fact that $\bigcup_{\tau\in\mathscr{F}_0}\tau = \gamma$ and $\crd(\gamma) \leq \sum_{\tau\in\mathscr{F}_0}\Crd(\tau)$ by the triangle inequality.
\end{proof}

In order to apply this lemma, we must preprocess the collections of dominating arcs $\{\tau_Q\}_{Q\in\scN}$ coming from each the families $\scG_1,\scG_3,\Delta_{2.2}\subseteq\scN$ individually into a bounded number of filtrations. $\protect{\cite{Sc07}}$ provides Lemma 3.13 for this. However, the statement and proof of the lemma as written contain errors which must be addressed.

First, the statement of the lemma makes the following claim: There exists $c_0 > 0$ such that for any arcs $\tau \subseteq \tau'$ with $\Diam(\tau') \leq 2 \Diam(\tau)$, we have $\tb{\tau'} \geq c_0\tb{\tau}$. In general, this is false. For example, Figure \ref{fig:beta-tilde} gives two counterexamples for this claim. The problem is not an issue for the results of the paper; although the claim is not true in general, an inequality of this type does hold for the specific arc families we will use. The proof of the lemma also contains a gap which is fixed in a modified, more general version given below. Before we state the lemma, we give a definition:
\begin{defn}[Augmentations]
Fix an arc $\tau$. We refer to any arc $\tau' \supseteq \tau$ as a \textit{$\tau$-augmentation} if we can write
\begin{equation}\label{e:augment-form}
    \tau' = \eta_1 \cup \tau \cup \eta_2
\end{equation}
where $\eta_1,\eta_2$ are arcs such that 
\begin{equation}\label{e:augment-properties}
    \Diam(\eta_i) \leq \frac{1}{1000}\Diam(\tau) \text{ and } \Domain(\eta_i)\cap\Domain(\tau)\not=\varnothing.
\end{equation}
This also gives $\Diam(\tau') \leq \left( 1 + \frac{1}{100} \right)\Diam(\tau)$.
\end{defn}

\begin{lem}[prefiltration lemma \protect\cite{BM22a}]\label{l:build-filtrations} Let $\XX$ be a metric space and let $f:[0,1]\rightarrow\overline{\Sigma}$ be a continuous parameterization of a set $\overline{\Sigma}\subset\XX$. Assume that $\rho>1$, $0<\underline{A}<A<\infty$, and $J\geq 1$ is any integer such that $\rho^{J}>6A/\underline{A}$. Then for every family $\mathscr{F}^0=\bigcup_{n=n_0}^\infty\mathscr{F}_n^0$ of arcs in $\overline{\Sigma}$ with $\mathscr{F}^0_{n_0}\neq\emptyset$ satisfying \begin{enumerate}[label=(\roman*)]
\item bounded overlap: for every arc $\tau\in\mathscr{F}^0_n$, there exists no more than $C$ arcs $\tau'\in\mathscr{F}^0_n$ such that $\Domain(\tau)\cap\Domain(\tau')\neq\emptyset$ for some constant $C$ independent of $\tau$
\item geometric diameters: for every arc $\tau\in\mathscr{F}^0_n$, we have $\underline{A}\rho^{-n}\leq \Diam(\tau)\leq A\rho^{-n}$,
\end{enumerate} we can construct $5(A/\underline{A})CJ$ or fewer filtrations $\mathscr{F}^1=\bigcup_{n=n_1}^\infty \mathscr{F}^1_n$, $\mathscr{F}^2=\bigcup_{n=n_2}^\infty \mathscr{F}^2_n$, \dots, with starting index $n_j\in \{n_0,n_0+1,\dots,n_0+J-1\}$ for all $j$ and \begin{equation}\label{new-arc-diameters}
\frac{1}{1000}\left(\underline{A}\rho^{(J-1)n_j}\right)\rho^{-Jn} \leq \Diam(\tau)< \left(1+\frac{1}{100}\right)\left(A\rho^{(J-1)n_j}\right) \rho^{-Jn}\quad\text{for all $j$, $\tau\in\mathscr{F}^j_n$, $n\geq n_j$,}
\end{equation} 
such that for every index $n\geq n_0$ and arc $\tau\in\mathscr{F}^0_{n}$, there exists $\mathscr{F}^j$ (in the list of filtrations), an index $N$ with $n-n_j=J(N-n_j)$, and a $\tau$-augmentation $\tau'\in\mathscr{F}^j_N$. The assignment $(n,\tau)\mapsto (\mathscr{F}^j,N,\tau')$ is injective.
\end{lem}

\begin{remark}[Changes to the statement of Lemma \ref{l:build-filtrations}]
In our statement of this lemma, we have first changed \eqref{new-arc-diameters} by replacing a $\frac{1}{4}$ in the diameter lower bound with a $\frac{1}{1000}$ and by replacing a 2 in the corresponding upper bound with $1+\frac{1}{100}$. The result of this change is that the lemma produces a $\tau$-augmentation $\tau'$ such that $\Diam(\tau') \leq \left(1+\frac{1}{100}\right)\Diam(\tau)$ rather than a general extension $\tau'$ such that $\Diam(\tau')\leq 2\Diam(\tau)$ as in the statement in \cite{BM22a}. This improvement can be made as long as $J$ is sufficiently large by carefully following the proof in \cite{BM22a}. In the following paragraph, we give a sketch of how one can justify this change.

Indeed, each filtration $\mathscr{F}^j$ produced in the lemma is composed of essentially two types of arcs: $\tau$-type arcs which are extensions of arcs in $\mathscr{F}^0$ originally passed into the lemma and $\sigma$-type arcs which are the leftover arcs in-between the $\tau$ type arcs. Each arc $\xi\in\mathscr{F}_0$ is extended to a $\tau$-type arc by adding in a chain of arcs of geometrically decreasing diameter, beginning with diameter $\lesssim \rho^{-J}\diam(\xi)$. Hence, each chain can be made to have arbitrarily small diameter compared to $\xi$ as long as we take $J$ small enough. After doing this process to all arcs in one stage of the filtration, the remaining in-between arcs of the curve are broken up appropriately and either added to the filtration themselves or added onto the ends of the recently produced $\tau$-type arcs. By replacing the appropriate factors of $\frac{1}{4}$ in this stage of the proof with $\frac{1}{1000}$, we can ensure each in-between arc is chopped into arcs of no diameter greater than $\left(1+\frac{1}{100}\right)\left(A\rho^{(J-1)n_j}\right) \rho^{-Jn}$ and no less than $\frac{1}{1000}\left(\underline{A}\rho^{(J-1)n_j}\right)\rho^{-Jn}$. Hence, they satisfy the desired bounds and appending these arcs to the previously produced $\tau$-type arcs gives the form $\tau' = \eta_1 \cup \tau \cup \eta_2$.
\end{remark}
If we pass a family of arcs $\mathscr{F}$ into lemma \ref{l:build-filtrations}, we receive a finite family of filtrations $\mathscr{F}_j$ such that for any arc $\tau\in\mathscr{F}$, there exists a filtration $\mathscr{F}_i$ and a unique $\tau$-augmentation $\tau'\in\mathscr{F}_i$. In order to effectively apply the filtration estimate in Lemma \ref{l:nonflat}, we must show that taking the $\tau'$ rather than $\tau$ does not ruin the arc beta number estimate $\tb{\tau_Q} \gtrsim_{\epsilon_2} \beta_\Sigma(Q)$. That is, we would like to show that $\tb{\tau'} \gtrsim \tb{\tau}$ for any $\tau$-augmentation $\tau'$. 

Badger and McCurdy do not need this in \cite{BM22a} because they use $\beta_\tau(\Image(\tau))$ instead of $\tb{\tau}$ as their measure of non-flatness of arcs which requires slightly different definitions of the primary arc families. Here, we take the different approach of showing that mapping $\tau^\prime\mapsto \tau$ given in Lemma \ref{l:build-filtrations} also preserves $\tb{\cdot}$ in the sense that there exists a constant $c > 0$ such that $\tb{\tau} \geq c\tb{\tau^\prime}$ for any arc $\tau^\prime$ in one of the particular families $\mathscr{F}$ which we pass into Lemma \ref{l:build-filtrations}. 

Fix an arc $\tau$ and let $\tau'$ be a $\tau$-augmentation. We begin with the simple observation that if $\tau$ has large Jones beta number, then $\tb{\tau'}\gtrsim\tb{\tau}$ trivially.
\begin{remark}\label{rem:jones-beta-tilde}
Let $\epsilon > 0$ and suppose $\beta_\tau(\Image(\tau)) \geq \epsilon$. Then
by definition, 
\begin{align*}
    \Drift(\tau') &\geq \beta_{\tau'}(\Image(\tau'))\Diam(\tau') \geq \frac{1}{2}\beta_{\tau'}(\Image(\tau))\Diam(\tau') \\
    &= \frac{1}{2}\beta_{\tau}(\Image(\tau))\Diam(\tau') \geq \epsilon\frac{\Diam(\tau')}{2}.
\end{align*}
where the second inequality follows from the fact that $\Image(\tau') \supseteq \Image(\tau)$ and $\Diam(\tau) \geq \frac{1}{2}\Diam(\tau')$. Hence, $\tb{\tau'} \geq \frac{\epsilon}{2} \geq \frac{\epsilon}{2}\tb{\tau}$. 
\end{remark}

This remark means that we can fix a small constant $\epsilon \gtrsim_A 1$ and achieve $\tb{\tau'} \gtrsim_\epsilon \tb{\tau}$ whenever $\tau$ satisfies $\beta_\tau(\Image(\tau)) \geq \epsilon$. It turns out that any remaining arc not covered by this case which we will need to pass into Lemma \ref{l:build-filtrations} will be a member of $\Lambda(Q)$ for some $Q$, meaning its endpoints lie in $\partial(2Q)$ and its image has nonempty intersection with $Q$. This geometric information is enough to conclude the desired bound.
\begin{lem}\label{l:nice-extensions}
There exists $c_1 > 0$ such that for all $Q\in\scG$ and $\tau\in\Lambda(Q)$, any $\tau$-augmentation $\tau'$ satisfies
\begin{equation}\label{e:beta-tilde-below}
    \tilde{\beta}(\tau') \geq c_1\tilde{\beta}(\tau).
\end{equation}
\end{lem}
Our goal for the rest of this section is to prove Lemma \ref{l:nice-extensions}. We begin by distinguishing between \textit{tall} and \textit{wide} arcs.
\begin{defn}[Tall and wide arcs]
Let $\tau:[a,b]\rightarrow\Sigma$ be an arc. One of the following two inequalities holds:
\begin{enumerate}[label=(\roman*)]
    \item $\Crd(\tau) < 100\Drift(\tau)$,\ \text{ or } \label{i:tall-arc}
    \item $\Crd(\tau) \geq 100\Drift(\tau)$\label{i:wide-arc}
\end{enumerate}
If $\tau$ satisfies \ref{i:tall-arc}, then we call $\tau$ \textit{tall}. If $\tau$ instead satisfies \ref{i:wide-arc}, then we call $\tau$ \textit{wide}. Tall arcs are allowed to drift very far from the line segment $\Edge(\tau)$ while wide arcs stay relatively close. Figure \ref{fig:beta-tilde} gives an example of each type.
\end{defn}

\begin{lem}\label{l:tall-extension}
Suppose $\tau$ is tall. Then $\tb{\tau'} \geq \frac{1}{4}\tb{\tau}$.
\end{lem}
\begin{proof}
Let $x,y\in\Image(\tau)$ such that $\Diam(\tau) = \dist(x,y)$. Then
\begin{equation}
    \Diam(\tau) \leq \dist(x,\Edge(\tau)) + \diam(\Edge(\tau)) + \dist(y,\Edge(\tau)) \leq 2\Drift(\tau) + \Crd(\tau) \leq 102\Drift(\tau).
\end{equation}
The augmentation $\tau'$ has the form $\tau' = \eta_1 \cup \tau \cup \eta_2$ where $\Diam(\eta_i) \leq \frac{1}{1000}\Diam(\tau) \leq \frac{102}{1000}\Drift(\tau)$. Hence, $\Drift(\tau') \geq \Drift(\tau) - \Diam(\eta_1) \geq \frac{1}{2}\Drift(\tau)$ and $\Diam(\tau') \leq \Diam(\tau) + \Diam(\eta_1) + \Diam(\eta_2) \leq 2\Diam(\tau)$. Therefore,
\begin{equation*}
    \tb{\tau'} = \frac{\Drift(\tau')}{\Diam(\tau')} \geq \frac{1}{4}\frac{\Drift(\tau)}{\Diam(\tau)} = \frac{1}{4}\tb{\tau}.
\qedhere\end{equation*}
\end{proof}
Hence, tall arcs extended via Lemma \ref{l:build-filtrations} satisfy \eqref{e:beta-tilde-below} with $c_1 = \frac{1}{4}$. With Lemma \ref{l:tall-extension}, we now only need a way of proving \eqref{e:beta-tilde-below} for a wide arc $\tau\in\Lambda(Q)$. The basic idea is as follows. The facts that $\tau$ is wide and $\Image(\tau)\cap Q\not=\varnothing$ mean that $\Edge(\tau)$ must have nonempty intersection with $\frac{3}{2}Q$. It suffices to show that there exists $x\in\Image(\tau)$ such that a fixed fraction of the value of $\dist(x,\Edge(\tau))$ comes from the direction perpendicular to $\Edge(\tau)$ rather than the direction parallel. This is proven in the following lemma:
\begin{lem}\label{l:wide-beta-below}
Suppose $\tau$ is a wide arc, and there exists $\alpha < 1$ and $x\in\Image(\tau)$ such that $\dist(x,\Line(\tau)) \geq \alpha \Drift(\tau)$. Then $\tb{\tau'} \geq \frac{\alpha}{2000}\tb{\tau}$.
\end{lem}
\begin{proof}
Define $B_a \vcentcolon= B\left(\Start(\tau),\frac{\alpha}{1000}\Drift(\tau)\right)$ and $B_b \vcentcolon= B\left(\End(\tau),\frac{\alpha}{1000}\Drift(\tau)\right)$. Suppose first that either $\Edge(\tau') \cap B_a =\varnothing$ or $\Edge(\tau')\cap B_b =\varnothing$. Assume without loss of generality that the latter holds. Then $\End(\tau)\in\Image(\tau')$ so that $\Drift(\tau') \geq \dist(\End(\tau),\Edge(\tau')) \geq \frac{\alpha}{1000}\Drift(\tau)$. This implies
\begin{equation*}
    \tb{\tau'} = \frac{\Drift(\tau')}{\Diam(\tau')} \geq \frac{\alpha}{2000}\frac{\Drift(\tau)}{\Diam(\tau)} = \frac{\alpha}{2000}\tb{\tau}.
\end{equation*}
Now, suppose that $\Edge(\tau')$ has nonempty intersection with both $B_a$ and $B_b$. Assume without loss of generality that $\Line(\tau)$ is the $e_1$-axis. Because $\tau$ is wide, $\Drift(\tau) \leq \frac{1}{100}\Crd(\tau)$ so that $B_a\cap B_b = \varnothing$ and $\Edge(\tau)$ hits both ends of the cylinder of length $\Crd(\tau) - 2\frac{\alpha}{1000}\Drift(\tau) \geq \frac{99}{100}\Crd(\tau)$ and radius $\frac{\alpha}{1000}\Drift(\tau)$ whose central axis is collinear with the $e_1$-axis. Let $\theta$ be the angle between $\Edge(\tau')$ and $\Edge(\tau)$. (We measure this by translating $\Edge(\tau')$ to intersect $\Edge(\tau)$, then measuring the angle in the plane containing $\Edge(\tau)$ and $\Edge(\tau')$.) We conclude
\begin{equation}\label{e:small-slope}
    \tan(\theta) \leq \frac{2\frac{\alpha}{1000}\Drift(\tau)}{\frac{99}{100}\Crd(\tau)} \leq \frac{\alpha}{100}\frac{\Drift(\tau)}{\Crd(\tau)} \leq 10^{-4}\alpha.
\end{equation}
We will derive a lower bound for $\Drift(\tau')$ in terms of $\Drift(\tau)$ by showing that $\Edge(\tau')$ remains much closer to $\Line(\tau)$ than the point $x$ is. We know $\Diam(\tau) \leq 2\Drift(\tau)+\Crd(\tau) \leq 2\Crd(\tau)$ so that the fact that $\tau'$ is a $\tau$-augmentation means that $\End(\tau') \in B(\End(\tau),\frac{1}{1000}\Diam(\tau)) \subseteq B(\End(\tau),\frac{1}{500}\Crd(\tau))$. A similar result holds for $\Start(\tau')$. A very rough estimate gives
\begin{equation*}
    \sup_{y\in\Edge(\tau')}\dist(y,\Line(\tau)) \leq \frac{\alpha}{1000}\Drift(\tau) + 2\Crd(\tau) \tan(\theta) \leq \frac{\alpha}{25}\Drift(\tau).
\end{equation*}
We conclude
\begin{align*}
    \Drift(\tau') &\geq \dist(x,\Edge(\tau')) \geq \dist(x,\Line(\tau)) - \sup_{y\in\Edge(\tau')}\dist(y,\Line(\tau))\\
    &\geq \alpha\Drift(\tau) - \frac{\alpha}{25}\Drift(\tau) \geq \frac{\alpha}{2}\Drift(\tau).
\end{align*}
Therefore, $\tb{\tau'} \geq \frac{\alpha}{4}\tb{\tau}$.
\end{proof}
\begin{remark}\label{rem:cone}
One can derive the existence of a point $x$ as in Lemma \ref{l:wide-beta-below} by showing that $\tau$ lies outside half cones centered at $\Start(\tau)$ and $\End(\tau)$ pointing away from $\Edge(\tau)$ of aperture $2\theta$ such that $\tan(\theta) \geq \alpha$. Indeed, then every point $y\in\Image(\tau)$ satisfies $\dist(y,\Line(\tau)) \geq \alpha\dist(y,\Edge(\tau))$ so that any point $x\in\Max(\tau)$ satisfies $\dist(x,\Line(\tau)) \geq \alpha\Drift(\tau)$.
\end{remark}
\begin{figure}
    \centering
    \includegraphics[scale=0.9]{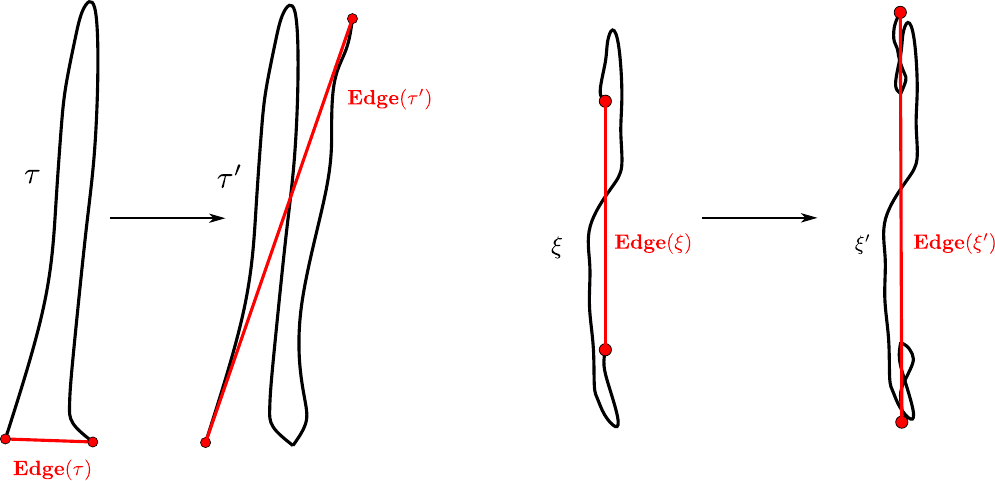}
    \caption{The arc $\tau$ is an example of a tall arc while $\xi$ is an example of a wide arc (neither are drawn to scale, but we hope the ideas are clear). Both of these arcs admit extensions $\tau'$ and $\xi'$ such that $\Diam(\tau') \leq 2\Diam(\tau)$ and $\Diam(\xi') \leq 2\Diam(\xi)$, but $\tau'$ and $\xi'$ have much smaller $\tilde{\beta}$ then $\tau$ and $\xi$. Arcs of the first type are excluded in our analysis by enforcing $\tau$-augmentations to extend only a small distance from the endpoints of $\tau$, while Lemma \ref{l:wide-beta-below} gives conditions for excluding arcs of the second type. (Roughly speaking, if $\xi\in\Lambda(Q)$, then it is not allowed to extend outwards in the direction parallel to its chord line outside of the ball $2Q$.)} 
    \label{fig:beta-tilde}
\end{figure}

With this, we can now give the proof of Lemma \ref{l:nice-extensions}.
\begin{proof}[Proof of Lemma \ref{l:nice-extensions}]
First, suppose that $\Edge(\tau)\cap \frac{3}{2}Q = \varnothing$. Then since $\tau\in\Lambda(Q)$, $\Image(\tau)\cap Q\not= \varnothing$ and we get $\Drift(\tau) \geq \frac{1}{2}\rad(Q) \geq \frac{1}{8}\diam(2Q) \geq \frac{1}{8}\Crd(\tau)$ so that $\tau$ is tall. Lemma \ref{l:tall-extension} implies $\tb{\tau'} \geq \frac{1}{4}\tb{\tau}$ as desired.

Now, suppose that $\Edge(\tau)\cap \frac{3}{2}Q\not=\varnothing$. Our goal is to apply Lemma \ref{l:wide-beta-below}. By Remark \ref{rem:cone}, it suffices to show that there is a $\theta > 0$ independent of $\tau$ such that the cone of aperture $\theta$ centered at $\End(\tau)$ (and $\Start(\tau)$) pointing away from $\Edge(\tau)$ lies entirely outside the ball $2Q$. Intuitively, this is true because the fact that $\Line(\tau)\cap\frac{3}{2}Q\not=\varnothing$ implies that every line in the tangent plane to $\partial(2Q)$ at $\End(\tau)$ makes large angle with $\Line(\tau)$. We supply the full details below.

Let $P$ be any two-dimensional affine plane containing $\Edge(\tau)$ and assume without loss of generality that $2Q = B(0,1),\ \Line(\tau) = \{de_1 + te_2:t\in\R\}$ for some $d < \frac{3}{4}$, and $P = \{de_1 + te_2 + sv\ :\ t,s\in\R,\ |v| = 1,\ v_2 = 0\}$. First, we show that $\Edge(\tau)$ also intersects a central ball in the disk $P\cap 2Q$.
\vspace{1ex}
\begin{claim}{}
$\Edge(\tau)\cap\frac{3}{4}(P\cap 2Q)\not=\varnothing.$
\end{claim}
\vspace{1ex}
\begin{claimproof}{}
$P\cap B(0,1)$ is a disk whose boundary has points satisfying the equation
\begin{equation*}
    (d+sv_1)^2 + t^2 + s^2(1-v_1^2) = 1 \iff (s+dv_1)^2 + t^2 = 1 - d^2(1-v_1^2).
\end{equation*}
This is a circle with radius $\sqrt{1 - d^2(1-v_1^2)}$ and center $(s,t) = (-dv_1, 0)$ which corresponds to the point $de_1 - dv_1v$. We want to show that $de_1\in\frac{3}{4}(P\cap B(0,1))$. But $|de_1 - (de_1 - dv_1v)| = |dv_1| \leq d \leq \frac{3}{4}$, as desired
\end{claimproof}
\vspace{1ex}

Now, it suffices to assume that $\Line(\tau) = \{de_1 + te_2\ :\ t\in\R\}\subseteq\R^2$ and to prove that there is $\theta$ such that the angle between $\Line(\tau)$ and any line tangent to $\partial B(0,1)\subseteq\R^2$ at $x\in\Edge(\tau)\cap\partial B(0,1) = \{(d,\sqrt{1-d^2}),\ (d, -\sqrt{1-d^2})\}$ makes angle greater than $\theta$. But by implicit differentiation of the equation $x^2 + y^2 = 1$, we see that
\begin{equation*}
    \frac{dy}{dx}\bigg|_{(x,y)=(d,-\sqrt{1-d^2})} = -\frac{x}{y}\bigg|_{(x,y)=(d,-\sqrt{1-d^2})} = \frac{\sqrt{1-d^2}}{d} \geq \frac{\sqrt{7}}{3}. 
\end{equation*}
Hence, we can take $\theta$ such that $\arctan{\theta} < \frac{\sqrt{7}}{3} < 2$. We apply Lemma \ref{l:wide-beta-below} with $\alpha = \frac{1}{2}$ to get $\tilde{\beta}(\tau') \geq \frac{1}{4000}\tilde{\beta}(\tau)$. This proves we can take $c_1 = \frac{1}{4000}$.
\end{proof}
\end{subsection}

\begin{subsection}{Bounds on the \texorpdfstring{$\scG_3,\scG_1, \text{ and }\Delta_{2.2}$ sums}{}}
In this subsection, we use the results from the previous subsection to prove Proposition \ref{p:non-flat}. The proofs are mostly adaptations of those for the corresponding lemmas in \cite{Sc07}.

The following three proofs share the same structure, each proving the desired bound for a particular family $\scC\in\{\scG_3,\scG_1,\Delta_{2.2}\}$. In each case, we define a mapping from $Q$ to some associated arc $\tau_Q$. We then show that the collection $\{\tau_Q\}_{Q\in\scC}$ satisfies the geometric diameters and bounded overlap properties necessary to apply Lemma \ref{l:build-filtrations}. This gives a bounded number of filtrations $\scF_{\scC}^j$ such that each $\tau_Q$ has a $\tau_Q$-augmentation $\tau_Q'$ as in the conclusion of Lemma \ref{l:build-filtrations}. The desired bound then follows from applying Lemma \ref{l:nonflat} to each of the filtrations as long as $\beta_\Sigma(Q) \lesssim \tb{\tau_Q'}$. This $\tilde{\beta}$ inequality is achieved by either the fact that $\tau_Q$ uniformly has $\beta_{\tau_Q}(\Image(\tau_Q)) > \epsilon$ or by showing that $\tau_Q$ satisfies the hypotheses of Lemma \ref{l:nice-extensions}. 

\begin{prop}[cf. \cite{Sc07} Lemma 3.16]\label{G3-bound}
\begin{equation*}
    \sum_{Q\in\scG_3}\beta_\Sigma(Q)^2\diam(Q) \lesssim_{J,A,\epsilon_2} \sH^1(\Sigma) \text{ and } \sum_{Q\in\scG_3}\beta_\Gamma(Q)^2\diam(Q) \lesssim_{J,A,\epsilon_2} \ell(\Gamma) - \crd(\Gamma).
\end{equation*}
\end{prop}
\begin{proof}
We begin by defining a new family $\scE\subseteq \scG_3$ and proving the claim for $\scE\cap\scG_3$ in place of $\scG_3$. Define
\begin{align*}
    \scE \vcentcolon= \bigg\{Q\in\scG_3 : &\exists\eta\subseteq\gamma_{12Q},\ \Diam(\eta) \geq \diam(Q)\ \text{and } \\
    &\tb{\eta'} \geq 10^{-6}\epsilon_2\beta_\Sigma(Q)\ \text{for all } \eta'\supseteq \eta\ \text{with } \Diam(\eta) \leq \Diam(\eta') \leq \left(1+\frac{1}{100}\right)\Diam(\eta)\bigg\}.
\end{align*}
We will build an appropriate mapping $Q\mapsto \tau_Q\subseteq \gamma_{12Q}$ to pass into Lemma \ref{l:build-filtrations}. We will then apply Lemma \ref{l:nonflat} to conclude the result.

Now, for any $Q\in\scE$ we have the existence of an arc $\eta_Q\subseteq \gamma_{12Q}$ with $\diam(Q) \leq \Diam(\eta_Q) \leq 24\diam(Q)$ and $\tb{\eta_Q} \geq 10^{-6}\epsilon_2\beta_\Sigma(Q)$. We define $\tau_Q \vcentcolon= \eta_Q$. 

In order to apply Lemma \ref{l:build-filtrations}, we must verify that the family $\{\tau_Q\}_{Q\in\scE}$ has geometric diameters and bounded overlap. The diameter requirement is satisfied by definition, so we only need to prove that $\scE_{Q} \vcentcolon= \{R\in\scE : \diam(R) = \diam(Q),\ \tau_R\cap \tau_Q\not=\varnothing\}$ satisfies $\#(\scE_Q) \leq C$ for some $C$ independent of $Q$. Using the parameterization $\gamma$, we can put a total order on balls with diameter equal to $\diam(Q)$ by setting $R < Q$ if and only if $\gamma_R^{-1}(x_R) < \gamma_Q^{-1}(x_Q)$. Because $X_n$ is finite, there exist balls $R_1,R_2\in\scE_Q$ which are respectively maximal and minimal in $\scE_Q$ with respect to this ordering. By definition, any $R\in\scE_Q$ must satisfy $x_R \in \gamma_{12R_1}\cup \gamma_{12Q} \cup \gamma_{12R_2}$. But, since $\tb{\gamma_{12Q'}}\leq \epsilon_2$ for all $Q'\in\scE$, the set $\Image(\gamma_{12Q'})$ is contained in a cylinder of width at most $\epsilon_2(24\diam(Q'))$ and length at most $24\diam(Q')$. Since net points on the scale of $Q'$ must be separated by distance at least $\frac{\diam(Q')}{2A}$, the net points must be separated by at least distance $\frac{\diam(Q')}{4A}$ along the axis of the cylinder because $\epsilon_2 \ll A^{-1}$. This means the number of net points on each of $\gamma_{12Q'}, \gamma_{12R_1}, \gamma_{12R_2}$ is less than $24\diam(Q')\cdot\frac{\diam(Q')}{4A} + 1\leq 100A$ so that $\#(\scE_Q) \leq 300A$ as desired.

This verifies the geometric diameter and bounded overlap conditions for $\{\tau_Q\}_{Q\in\scE}$. We apply Lemma \ref{l:build-filtrations} to receive a bounded number of filtrations $\scF_{\scE}^j,\ j\in J_{\scE}$ such that for any $\tau_Q$, there exists a $\tau_Q$-augmentation $\tau_Q'\in\scF_{\scE}^j$ for some $j$. Therefore, the definition of $\scE$ also implies that $\tb{\tau_Q'} \geq10^{-6}\epsilon_2\beta_\Sigma(Q)$. Therefore, we have
\begin{equation*}
    \sum_{Q\in\scG_3\cap\scE}\beta_\Sigma(Q)^2\diam(Q) \lesssim_{\epsilon_2} \sum_{j\in J}\sum_{\tau\in\scF_{\scE}^j}\tb{\tau}^2\Diam(\tau) \lesssim_{J,A} \ell(\gamma)-\crd(\gamma)
\end{equation*}
using Lemma \ref{l:nonflat}. This proves the desired inequalities for $\scG_3\cap\scE$. We will now prove this for $\scG_3\setminus\scE$.

Indeed, fix $Q\in\scG_3\setminus\scE$. We look to build an appropriate mapping $Q\mapsto \tau_Q$ to pass into Lemma \ref{l:build-filtrations}. Let $x\in\Sigma\cap Q$ be such that $\dist(x,\bigcup_{\tau\in S(Q)}\Image(\tau))$ is maximal and let $\xi_Q\in\Lambda(Q)\setminus S(Q)$ be such that $x\in\Image(\xi_Q)$. By the definition of $\scG_3$, $\beta_{S(Q)}(Q)\leq\epsilon_1\beta_\Sigma(Q)$ so that
\begin{equation*}
    \beta_{S(Q)\cup\xi_Q}(Q) \geq \frac{1}{3}\beta_\Sigma(Q).
\end{equation*}
Indeed, otherwise $\Sigma\cap Q$ is contained in a cylinder of width $\frac{2}{3}\beta_\Sigma(Q)\diam(Q) + c\epsilon_1\beta_\Sigma(Q)\diam(Q)$ with $\epsilon_1 \ll c$ contradicting the definition of $\beta_\Sigma(Q)$. We set $\tau_Q \vcentcolon= \xi_Q$. We now verify that the family $\{\tau_Q\}_{Q\in\scG_3\setminus\scE}$ has geometric diameters and bounded overlap in order to apply the pre-filtration lemma.

Since $\tau_Q\in\Lambda(Q)$, we know $\tau_Q\cap Q\not=\varnothing$ and $\tau_Q\cap H\setminus 2Q\not=\varnothing$ so that
\begin{equation*}
    2\diam(Q) \geq \Diam(\tau_Q) \geq \rad(2Q) - \rad(Q) = \frac{1}{2}\diam(Q).
\end{equation*}
In order to verify bounded overlap, set $\scG_Q \vcentcolon= \{R\in \scG_3\setminus \scE : \diam R = \diam Q,\ \tau_R\cap\tau_Q\not=\varnothing\}$ so that we want to show $\#(\scG_Q) \leq C$ for $C$ independent of $Q$. Assume first that $\beta_\Sigma(R) \leq \beta_\Sigma(Q)$. Because $\tau_R\cap\tau_Q\not=\varnothing$, we have $2Q\cap 2R\not=\varnothing$ so that $x_R\in 8Q\subseteq 12R$. Let $\gamma_{12R}|_{8Q}$ be a largest diameter subarc of $\gamma_{12R}$ which is in $\Lambda(8Q)$. We want to show that $\gamma_{12R}|_{8Q}\in S_{8Q}$. This is the reason for the addition of $\scE$. Because $R\not\in\scE$ and $\Diam(\gamma_{12R}|_{8Q}) \geq \diam R$, there exists some extension $\eta'\supseteq\gamma_{12R}|_{8Q}$ such that $\tb{\eta'} < 10^{-6}\epsilon_2\beta_\Sigma(R)$. But, $x_R\in 4Q = \frac{1}{2}8Q$ implies we can apply Lemma \ref{l:nice-extensions} to conclude
\begin{equation*}
    \tb{\gamma_{12R}|_{8Q}}\leq 4000\tb{\eta'} \leq \frac{4000}{10^6}\epsilon_2\beta_\Sigma(R) < \epsilon_2\beta_\Sigma(8Q).
\end{equation*}
This proves that $\gamma_{12R}|_{8Q}\in S_{8Q}$. In particular, the fact that $\beta_{S_{8Q}}(8Q) \leq \epsilon_1\beta_\Sigma(Q)$ implies that $x_R$ is contained in a small tube around $\gamma_{8Q}$. We assumed that $\beta_\Sigma(R) \leq \beta_\Sigma(Q)$ for this, but if $\beta_\Sigma(R) > \beta_\Sigma(Q)$, then running the argument for $\gamma_{12Q}|_{8R}$ in place of $\gamma_{12R}|_{8Q}$ proves the same claim with $Q$ and $R$ reversed. In either case, all $R\in\scG_Q$ are contained in a small neighborhood of the almost flat arc $\gamma_{12Q}$, proving $\#(\scG_Q)\leq 100A$.

This verifies the geometric diameters and bounded overlap condition, so we apply Lemma \ref{l:build-filtrations} to get a bounded family of filtrations $\scF^j_{\scG_3},\ j\in J_{\scG_3}$ such that for each $Q\in\scG_3\setminus\scE$, there exists $\tau_Q'\in\scF^j_{\scG_3}$ for some $j$ which is a $\tau_Q$-augmentation. Because $\tau_Q\in\Lambda(Q)\setminus S_{Q}$, we apply Lemma \ref{l:nice-extensions} to conclude 
\begin{equation*}
    \tb{\tau_Q'} \geq \frac{1}{4000}\tb{\tau_Q} \geq \frac{\epsilon_2}{4000}\beta_\Sigma(2Q) \geq \frac{\epsilon_2}{8000}\beta_\Sigma(Q).
\end{equation*}
Therefore, Lemma \ref{l:nonflat} implies
\begin{equation*}
    \sum_{Q\in\scG_3\setminus\scE}\beta_\Sigma(Q)^2\diam(Q) \lesssim_{\epsilon_2} \sum_{j\in J_{\scG_3}}\sum_{\tau\in\scF_{\scG_3}^j}\tb{\tau}^2\Diam(\tau) \lesssim_{J,A} \ell(\gamma) -\crd(\gamma).
\qedhere\end{equation*}
\end{proof}

\begin{prop}[cf. \cite{Sc07} Lemma 3.14]\label{G1-bound}
\begin{equation*}
    \sum_{Q\in\scG_1}\beta_\Sigma(Q)^2\diam(Q) \lesssim_{J,A} \sH^1(\Sigma) \text{ and } \sum_{Q\in\scG_1}\beta_\Gamma(Q)^2\diam(Q) \lesssim_{J,A} \ell(\Gamma) - \crd(\Gamma).
\end{equation*}
\end{prop}
\begin{proof}
Let us build an appropriate mapping $ Q\mapsto \tau_{ Q}$. Put $ Q = B(x_Q,A\lambda2^{-n})$. We define $\tau_{ Q}$ in one of two ways:
\begin{enumerate}[label=(\roman*)]
    \item If $\#(\gamma_{ Q}\cap X_n) \leq 3\lambda A$, then set $\tau_{ Q} \vcentcolon= \gamma_{ Q}$.
    \item Otherwise, let $\tau_{ Q}$ be a subarc of $\gamma_{ Q}$ containing $x_Q$ such that $\#(\tau_{ Q}\cap X_n) = 3\lambda A$.
\end{enumerate}
In order to apply Lemma \ref{l:build-filtrations}, we must check the geometric diameters and bounded overlap conditions. The geometric diameters condition follows in case (i) because $\gamma_{ Q}\in\Lambda(Q)$. It follows in case (ii) because the net point condition implies that $2^{-n} \leq \Diam(\tau_{ Q}) \leq 2\diam(Q) = 4\lambda A 2^{-n}$. In either case, bounded overlap follows from a similar argument to that in the proof of Proposition \ref{G3-bound}. Because each arc is centered on a unique net point in $X_n$ and each arc contains at most $3\lambda A$ net points, ordering the net points via the parameterization $\gamma$ shows that there can be at most $6\lambda A$ net points (inclusive) between intersecting arcs $\tau_{ Q}$ and $\tau_{R}$ for either arcs of type (i) or (ii) above. This proves $\#(\{ R\in\scG_1: \diam( R) = \diam( Q),\ \tau_{ R}\cap\tau_{ Q} \not=\varnothing\}) \leq 12\lambda A$.

Applying Lemma \ref{l:build-filtrations} to the collections of type (i) and (ii) arcs above gives a family of filtrations $\scF_{\scG_1}^j,\ j\in J_{\scG_1}$ such that for any $Q\in\scG_3^\lambda$, there exists a $\tau_Q$-augmentation $\tau_Q'\in\scF_{\scG_1}^j$ for some $j$. In order to finish, we only have to check that $\tb{\tau_Q'} \gtrsim_A \beta_\Sigma(Q)$. For arcs of type (i), $\tau_{Q} = \gamma_{ Q}\in\Lambda(Q)$ so that Lemma \ref{l:nice-extensions} gives the result. For arcs of type (ii), observe that $\#(\tau_{Q}\cap X_n) = 3\lambda A$ implies that $\beta_{\tau_Q'}(\tau_Q') \gtrsim \beta_{\tau_Q}(\tau_{Q}) \gtrsim_{A} 1 \gtrsim_A \beta_\Sigma(Q)$. Therefore, applying Lemma \ref{l:nonflat} to this collection of filtrations gives
\begin{equation*}
    \sum_{Q\in\scG_1}\beta_\Sigma(Q)^2\diam(Q) \lesssim_A \sum_{j\in J_{\scG_1}}\sum_{\tau\in\scF_{\scG_1}^j}\tb{\tau}^2\diam(\tau) \lesssim_{J,A} \ell(\gamma) - \crd(\gamma).
\qedhere\end{equation*}
\end{proof}

\begin{prop}[cf. \cite{Sc07} Lemma 3.24]\label{Delta2.2-bound}
\begin{equation*}
    \sum_{Q\in\Delta_{2.2}}\beta_\Sigma(Q)^2\diam(Q) \lesssim_{J,A} \sH^1(\Sigma) \text{ and } \sum_{Q\in\Delta_{2.2}}\beta_\Gamma(Q)^2\diam(Q) \lesssim_{J,A} \ell(\Gamma) - \crd(\Gamma).
\end{equation*}
\end{prop}
\begin{proof}
Let us build an appropriate mapping $Q\mapsto \tau_{ Q}$ as in the previous two propositions. Again, let $Q = B(x_Q,\lambda A2^{-n})$ By definition, there exists $\xi_Q\in\Lambda(Q)\setminus S_{Q}$ such that $\xi_{Q}\cap U_{Q}^x \not= \varnothing$. We define $\tau_Q$ in one of two ways
\begin{enumerate}[label=(\roman*)]
    \item If $\#(\{R\in\Delta_{2.2} : \diam(R) = \diam(Q),\ \xi_Q\cap U_{R}^x \not=\varnothing\}) \leq 9\lambda A$, then set $\tau_Q = \xi_Q$.
    \item Otherwise, let $\tau_Q$ be a subarc of $\xi_Q$ such that $\tau_Q\cap U_Q^x\not=\varnothing$ and $\#(\{R\in\Delta_{2.2} : \diam(R) = \diam(Q),\ \xi_Q\cap U_{R}^x \not=\varnothing\}) = 9\lambda A$.
\end{enumerate}
Type (i) arcs have geometric diameters since $\xi_Q\in\Lambda(Q)$. Type (ii) arcs have nonempty intersection with two distinct, disjoint cores $U_Q^x$ and $U_R^x$ so that $\Diam(\xi_Q) \geq \dist(\partial U_Q^x, \partial U_Q^{xx}) \gtrsim \diam(Q)$. To check bounded overlap, we argue almost exactly as in the corresponding part of the proof of Proposition \ref{G1-bound}. Indeed, $\tau_Q\cap U_Q^x\not=\varnothing$ so that we can order the arcs $\tau_Q$ via the parameterization $\gamma$ by the ordering of $x_Q\in U_Q^x$. There can be at most $18\lambda A$ net points separating $x_Q$ and $x_R$ for admissible $R$ so that $\#(\{R\in\Delta_{2.2} : \diam(R) = \diam(Q),\ \tau_R\cap\tau_Q\not=\varnothing\}) \leq 36\lambda A$.

Applying Lemma \ref{l:build-filtrations} gives a bounded number of filtrations $\scF_{\Delta}^j$ such that each $Q\in\Delta_{2.2}$ has an associated $\tau_Q$-augmentation $\tau_Q'$. We only need to show that $\tb{\tau_Q'} \gtrsim \beta_\Sigma(Q)$. This follows for type (i) arcs by Lemma \ref{l:nice-extensions} and for type (ii) arcs by the fact that $\#(\{R\in\Delta_{2.2} : \diam(R) = \diam(Q),\ \xi_Q\cap U_{R}^x \not=\varnothing\}) = 9\lambda A$ implies $\tb{\tau_Q'} \gtrsim \beta_{\tau_Q}(\tau_Q) \gtrsim_A 1 \gtrsim_A \beta_\Sigma(Q)$ as in the proof of Proposition \ref{G1-bound}. The result follows by applying Lemma \ref{l:nonflat} to each filtration to get
\begin{equation*}
    \sum_{Q\in\Delta_{2.2}}\beta_\Sigma(Q)^2\diam(Q) \lesssim_A \sum_{j\in J_{\Delta_{2.2}}}\sum_{\tau\in\scF_{\Delta_{2.2}}^j}\tb{\tau}^2\diam(\tau) \lesssim_{J,A} \ell(\gamma) - \crd(\gamma).
\qedhere\end{equation*}
\end{proof}
\end{subsection}

\end{section}
\begin{section}{Almost flat arcs: \texorpdfstring{$\Delta_1,\Delta_{2.1}$}{}}\label{sec:almost-flat-arcs}
Our goal now is to prove the following proposition:
\begin{prop}\label{p:almost-flat}
Set $\scA\vcentcolon=\Delta_1\cup\Delta_{2.1}$. We have
\begin{align}
    \sum_{Q\in\scA}\beta_\Sigma(Q)^2\diam(Q) \lesssim_A \ell(\Sigma) \text{ and } \sum_{Q\in\scA}\beta_\Gamma(Q)^2\diam(Q) \lesssim_A \ell(\Gamma) - \crd(\Gamma).\label{e:almost-flat-bound}
\end{align}
\end{prop}

Recall that $\scA\subseteq\scG_2$ so that for any $Q\in\scA$, $\beta_\Sigma(Q) \leq \epsilon_1^{-1}\beta_{S(Q)}(Q)$. That is, the beta number of the union of images of almost flat arcs controls the total beta number for $Q$. For the purposes of estimating the beta-squared sum, we can essentially think of $\Sigma$ (or $\Gamma$) inside of $Q$ as consisting of a union of line segments (we have taken the parameter $\epsilon_2$ sufficiently small so that this heuristic holds at all scales we will perform estimates at). In Section \ref{subsec:almost-flat-arcs-sigma} we prove the first inequality in \eqref{e:almost-flat-bound}, finishing the proof of the Hilbert space necessary condition. In Section \ref{subsec:almost-flat-arcs-gamma} we prove the second, finishing the proof of Theorem \ref{t:thmA}.

We begin by giving some comments on the structure of almost flat arcs. Recall that an almost flat arc $\tau\in S(Q)$ satisfies the inequality
\begin{equation}\label{e:beta-tilde}
    \tb{\tau} \leq \epsilon_2\beta_\Sigma(Q).
\end{equation}
We interpret this as saying that $\tau$ is a \textit{very} small perturbation of $\Edge(\tau)$ relative to the overall flatness of $\Sigma$ inside $Q$. This means that $\tau$ is \textit{bilaterally} close to $\Edge(\tau)$, forcing $\tau$ to be ``diametrical'' and giving it the crossing property we prove in Lemma \ref{l:almost-flat-crossing} below. The condition \eqref{e:beta-tilde} is importantly stronger than the similar inequality
\begin{equation}\label{e:beta-tilde-replacement}
    \beta_{\tau}(\Image(\tau)) \leq \epsilon_2\beta_{\Sigma}(Q).
\end{equation}
This condition only forces the image of $\tau$ to be \textit{unilaterally} close to some line $L$ relative to the overall flatness of $\Sigma$ inside $Q$. This allows almost flat arcs which are ``radial'' rather than ``diametrical''. This is an important point at which the results here diverge from results of \cite{BM22b} in which analogous results are proven for this weaker notion of almost flat arcs in Banach spaces. 

We now record two lemmas needed in the following sections.
\begin{defn}[Cylinders]
Let $a,b\in \ell_2$, let $s\vcentcolon=[a,b]$ be a line segment, and let $r > 0$. We define the cylinder $C$ of radius $r$ around $s$ as
\begin{equation*}
    C(s,r) \vcentcolon= \{z\in\pi_{s}^{-1}(s) : \pi_{s}^\perp(z) \leq r\},
\end{equation*}
where $\pi_s$ is the orthogonal projection onto the line collinear with the line segment $s$ and $\pi_s^\perp:\ell_2\rightarrow\ell_2$ is the projection onto the corresponding affine orthogonal hyperplane. We also allow $s$ to be an affine line. For a segment $s$ as above, we define the faces
\begin{equation*}
    F_a(s,r) \vcentcolon= \{z\in C(s,r) : \pi_s(z) = a\} \text{ and } F_b(s,r) \vcentcolon= \{z\in C(s,r) : \pi_s(z) = b\}.
\end{equation*}
For any $\tau\in\Lambda(Q)$, $\tau\subseteq C(\Line(\tau),\tb{\tau}\Diam(\tau)) \subseteq C(\Line(\tau),\tb{\tau}\diam(2Q))$..
\end{defn}

\begin{lem}[Crossing Property]\label{l:almost-flat-crossing}
Let $Q\in\scG,\ \tau\in\Lambda(Q)$, and $\tau'\vcentcolon= [a_{\tau'},b_{\tau'}] \subseteq \Edge(\tau)\vcentcolon=[a_\tau,b_\tau]$. Let $\epsilon > 0$ such that $\tb{\tau} \leq \frac{\epsilon}{2}$. There exists an arc $\tau_0$ such that
\begin{enumerate}[label=(\roman*)]
    \item $\Domain(\tau_0) \subseteq \Domain(\tau)\cap\gamma^{-1}(C(\tau',\epsilon\diam(Q))),$ \text{ and} \label{i:tau0-in-cylinder}
    \item $\Diam(\tau_0) \geq \Diam(\tau')$. \label{i:tau0-diam}
\end{enumerate}
\end{lem}
\begin{proof}
Let $C\vcentcolon=C(\Line(\tau'),\epsilon\diam(Q))$. Because $\tb{\tau}\leq\frac{\epsilon}{2}$ and $\Diam(\tau)\leq 2\diam(Q)$, we know $\tau \subseteq C$. Because $\Image(\tau)$ is connected, $\pi_{\tau'}$ is continuous, and $a_\tau,b_\tau\in\Edge(\tau)\cap\tau$, there must exist $u\in\tau\cap F_{a_{\tau'}}$ and $v\in\tau\cap F_{b_{\tau'}}$. This implies that $\{t\in\Domain(\tau) : \tau(t) \in F_{b_{\tau'}}\}\not=\varnothing$ and $\{t\in\Domain(\tau) : \tau(t)\in F_{a_{\tau'}}\}\not=\varnothing$ so that we can further suppose without loss of generality that
\begin{equation}\label{e:tau-domain}
    0 \leq \inf\{t\in\Domain(\tau) : \tau(t)\in F_{a_{\tau'}}\} < \inf\{t\in\Domain(\tau) : \tau(t) \in F_{b_{\tau'}}\}.
\end{equation}
That is, $\tau$ enters $F_{a_{\tau'}}$ before $F_{b_{\tau'}}$. We define 
\begin{align*}
    t_2 &\vcentcolon= \inf\{t\in\Domain(\tau) :\tau(t)\in F_{b_{\tau'}}\},\\
    t_1 &\vcentcolon= \sup\{t\in\Domain(\tau) : t \leq t_2,\ \tau(t)\in F_{a_{\tau'}}\},\\
    \tau_0 &\vcentcolon= \tau|_{[t_1,t_2]}.
\end{align*}
Suppose without loss of generality that $\pi_{\tau'}(a_{\tau'}) \leq \pi_{\tau'}(b_{\tau'})$. We know $\tau(t_2)\in F_{b_{\tau'}}$ by the continuity of $\tau$. By the continuity of $\pi_{\tau'}$ and the definition of $t_2$, we also know that $\pi_{\tau'}(\tau(t)) \leq \pi_{\tau'}(b_{\tau'})$ for all $t \leq t_2$. On the other hand, the definition of $t_1$ implies that $\pi_{\tau'}(\tau(t)) \geq \pi_{\tau'}(a_{\tau'})$ for all $t_1 \leq t \leq t_2$ so that $\tau|_{[t_1,t_2]}\subseteq C(\tau',\epsilon\diam(Q))$. Item \ref{i:tau0-in-cylinder} follows. In fact, we can conclude $\tau(t_1)\in F_{a_{\tau'}}$ because the supremum in the definition of $t_1$ is over a non-empty set by \eqref{e:tau-domain}. Item \ref{i:tau0-diam} follows because $\Diam(\tau_0) \geq |b_{\tau'} - a_{\tau'}| = \Diam(\tau')$.
\end{proof}
For convenience, we also record an estimate for lower-bounding the diameter of chord segments of arcs which touch central balls inside of $Q$:
\begin{lem}\label{l:diam-low-bound}
Let $Q = B(x_Q,R)$ be a ball and let $0 < \alpha < 1$ be such that $\alpha^2 < 1/2$. Let $\varphi'\vcentcolon= [a_{\varphi'},b_{\varphi'}]$ be a line segment such that $a_{\varphi'},b_{\varphi'}\in \partial Q$ and $\varphi'\cap \alpha Q\not=\varnothing$. Then,
\begin{equation*}
    \sH^1([a_{\varphi'},b_{\varphi'}]) \geq \diam(Q)\left(1-2\alpha^2\right).
\end{equation*}
\end{lem}
\begin{proof}
We will first give a lower bound for the function $\sqrt{1-x}$, then apply this to a Pythagorean theorem estimate. Let $0 < x < \frac{1}{2}$ and observe that, by the generalized binomial theorem,
\begin{equation*}
    \sqrt{1-x} = \sum_{n=0}^\infty(-1)^n\frac{\frac{1}{2}(\frac{1}{2}-1)\cdots(\frac{1}{2}-n+1)}{n!}x^n \geq 1 - \sum_{n=1}^\infty x^n \geq 1 - x -\frac{x}{1-x} \geq 1-2x
\end{equation*}
using our assumption that $x \leq \frac{1}{2}$ in the last line. Now, let $y_{\varphi'}$ be the point in $\varphi'$ closest to $x_Q$. Then, using the Pythagorean theorem, we get $|y_{\varphi'} - a_{\varphi'}|^2 = |a_{\varphi'} - x_Q|^2 - |x_Q - y_{\varphi'}|^2$ from which we can estimate
\begin{align*}
    |y_{\varphi'} - a_{\varphi'}| \geq \sqrt{R^2 - (\alpha R)^2} \geq R\sqrt{1 - \alpha^2} \geq R\left(1-2\alpha^2\right).
\end{align*}
Applying the same argument to $|y_{\varphi'} - b_{\varphi'}|$, we get $\sH^1([a_{\varphi'},b_{\varphi'}]) = |y_{\varphi'} - b_{\varphi'}| + |y_{\varphi'} - a_{\varphi'}| \geq \diam(Q)(1-2\alpha^2)$.
\end{proof}
\begin{subsection}{Almost flat arcs for $\Sigma$}\label{subsec:almost-flat-arcs-sigma}
In this section, we complete our proof of \eqref{e:HS-nec}, the necessary condition in the Hilbert space traveling salesman theorem. We begin in Section \ref{subsec:martingale} by giving a general presentation of Schul's martingale construction. In Section \ref{subsec:sigma-delta1}, we give the first application of the martingale construction by repeating Schul's proof of the beta-squared sum bound for the family $\Delta_1$. Finally, in Section \ref{subsec:sigma-delta2.1} we give a new proof of the beta-squared sum bound for the family $\Delta_{2.1}$ using Schul's martingales again, filling in the final gap in proof of the Hilbert space necessary condition in \cite{Sc07}.
\begin{subsubsection}{Schul's Martingale Lemma}\label{subsec:martingale}
The martingale argument relies heavily on the structure of the cores for balls constructed in Proposition \ref{p:cores}, so we begin by giving some definitions and notation related to the families of cores. For the rest of this section, fix $0 < c < \frac{1}{4A}$ and $J\geq10$.
\begin{defn}[The tree structure of cores]
Fix a collection $\scL\subseteq\scG$. Proposition \ref{p:cores} gives a partition of $\scG$ into $J$ families $\{\scQ_j\}_{j=1}^J$ such that cores for balls inside $\scQ_j$ satisfy the inclusion and separation properties \ref{i:core-diam}, \ref{i:core-sep}, and \ref{i:core-nest} in Proposition \ref{p:cores}. Defining $\scL_j \vcentcolon= \scL\cap\scQ_j$, we see that for any $Q\in\scL_j$, either
\begin{enumerate}[label=(\alph*)]
    \item For all $Q'\in\scL_j$ such that $U^c_{Q'}\cap U_Q\not=\varnothing$, $U^c_{Q'}\subseteq U^c_Q$, \text{or}
    \item There exists $Q'\in\scL_j$ such that $U^c_Q\subsetneq U^c_{Q'}$.
\end{enumerate}
These set inclusion properties induce a partial order on $\scL$, giving it the structure of a forest in which the balls satisfying the first condition above are the roots of trees in the forest while the balls satisfying the second condition are descendants of some root. We denote the forest of trees (whose partial order depends on the constants $c$ and $J$) by $\mathcal{T}^{c,J}_{\scL} = \mathcal{T}^{c}_{\scL} = \mathcal{T}_{\scL}$ where we often suppress the constants when understood (in practice, we suppress $J$ more often than $c$ in the construction because $J$ will be fixed once and for all while $c$ will vary). We refer to the \textit{root} of $T\in\mathcal{T}^{c,J}_{\scL}$ as $Q(T)$. For each $Q\in T,\ Q\not=Q(T)$, there exists a unique minimal ball $P(Q)$ respect to the ordering of $T$ such that $U^c_{Q}\subsetneq U^c_{P(Q)}$. We call $P(Q)$ the \textit{parent} of $Q$. Similarly, for any $Q\in T$ we define the collection of \textit{children} of $Q$ in $T$ by
\begin{equation*}
    C(Q) \vcentcolon= \{Q'\in T : \text{$Q'$ is maximal in $T$ such that } U^c_{Q'}\subsetneq U^c_Q\}.
\end{equation*}
We also think of $C^1(Q) \vcentcolon= C(Q)$ as the first generation descendants of $Q$. Given the set $C^n(Q)$ for some $n\geq 1$, we define the $n+1$-th generation descendants of $Q$ as
\begin{equation*}
    C^{n+1}(Q) \vcentcolon= \{Q''\in C(Q') : Q'\in C^n(Q)\}.
\end{equation*}
Because each ball is either a root or its core is contained in the core of some root ball, we have
\begin{equation*}
    U_{\scL}^c\vcentcolon=\bigcup_{Q\in\mathcal{T}^c_\scL}U^c_{Q(T)} = \bigcup_{Q\in\scL}U^c_Q.
\end{equation*}
If $\scL\subseteq\scQ_j$ for some $j,\ 1\leq j \leq J$, then the union above over trees is disjoint. Otherwise it is a union in which each point is contained in at most $J$ constituent sets.
\end{defn}
We now give the definition of a martingale and relevant notions from probability theory.
\begin{defn}[Martingales]
Given a probability space $(\Omega,\mathcal{A},\mathbb{P})$, we define a \textit{filtration} of $\mathcal{A}$ to be an increasing sequence $(\mathcal{F}_n)_{n\geq0}$ of sub-$\sigma$-algebras of $\mathcal{A}$. We say that a collection of real-valued random variables $(X_n)_{n\geq0}$ on $\Omega$ is a \textit{martingale} with respect to $(\mathcal{F}_n)_{n\geq0}$ if for all $n\geq 0$,
\begin{enumerate}[label=(\roman*)]
    \item $X_n$ is $\mathcal{F}_n$-measurable,
    \item $\mathbb{E}(X_n) < \infty$,
    \item $\mathbb{E}(X_{n+1}|\mathcal{F}_n) = X_n$.
\end{enumerate}
where $\mathbb{E}(X_{n+1}|\mathcal{F}_n)$ denotes the conditional expectation of $X_{n+1}$ with respect to $\mathcal{F}_n$. Importantly, it is well-known that positive martingales converge pointwise almost surely. That is, if $X_n \geq 0$ for all $n\geq0$, then there exists a positive random variable $X$ such that $X(\omega) = \lim_{n\rightarrow\infty}X_n(\omega)$ for $\mathbb{P}$ almost all $\omega\in\Omega$. We will only consider positive martingales.
\end{defn}
\begin{remark}[Schul's martingales]
Let $\scL\subseteq \scG$ and form the forest $\mathcal{T}^{c}_{\scL}$ which gives $\scL$ a partial order, hence a child-parent structure as defined above. For each $Q\in\scL$, Schul constructs a martingale $(w^n_Q)_{n\geq0}$ supported inside $U^c_Q\cap\Sigma$. We define the remainder
\begin{equation*}
    R_Q \vcentcolon= U_Q\cap\Sigma \setminus \left( \bigcup_{Q'\in C(Q)} U^c_{Q'}\cap \Sigma \right)
\end{equation*}
so that
\begin{equation*}
    U^c_Q\cap \Sigma = \left(\bigcup_{Q'\in C(Q)}U^c_{Q'}\cap\Sigma\right) \cup R_Q
\end{equation*}
where the collection $\{U^c_{Q'}\cap\Sigma\}_{Q'\in C(Q)}\cup \{R_Q\}$ is pairwise disjoint. Applying this partitioning scheme iteratively to each $U^c_{Q'}\cap\Sigma$ in the union above, we see that for any $n\geq0$, we can write $U^c_Q\cap\Sigma$ as the partition.
\begin{equation}\label{e:filt-decomp}
    U^c_Q\cap\Sigma = R_Q \cup \left(\bigcup_{Q'\in C^1(Q)}R_{Q'}\right) \cup \ldots \cup \left( \bigcup_{Q'\in C^n(Q)}R_{Q'} \right) \cup \left(\bigcup_{Q'\in C^{n+1}(Q)} U^c_{Q'}\cap \Sigma\right)
\end{equation}
This gives a decomposition of $U^c_Q\cap\Sigma$ into ``atoms'' at the $(n+1)$-th level, from which we will define a filtration by setting $\mathcal{F}_n$ to be the sigma algebra generated by $\cup_{k\leq n}\{U^c_{Q'}\cap\Sigma\}_{Q'\in C^k(Q)}$. We will form the martingale $(w_Q^n)_{n\geq0}$ by setting $w_Q^0$ to be constant on $U^c_Q\cap\Sigma$ and defining $w_Q^{n+1}$ by distributing the mass that $w_Q^n$ assigns to $U^c_{Q'}\cap\Sigma$ for any $Q'\in C^n(Q)$ onto its constituent pieces $R_{Q'}\cup\bigcup_{Q''\in C(Q')}U^c_{Q''}\cap\Sigma = U^c_{Q'}\cap\Sigma$ with weighting factors depending on the size and number of children in $C(Q')$ and the length of the remainder $R_{Q'}$.
\end{remark}

\begin{lem}[Martingale construction]\label{l:martingale}
Fix a constant $D > 0,\ c <\frac{1}{4A}$ and let $\scL\subseteq \scG\cap\scQ_j$. Suppose that there exists a constant $q < 1$ such that for any $Q\in\scL$
\begin{equation}\label{e:martingale-decay}
    \frac{\diam(U^c_Q)}{\sum_{Q'\in C(Q)}\diam(U^c_{Q'}) + D\ell(R_Q)} \leq q.
\end{equation}
Then, there exists a collection of positive real-valued functions $\{w_Q\}_{Q\in\scL}$ satisfying
\begin{enumerate}[label=(\roman*)]
    \item $\int_\Sigma w_Qd\ell = \diam(U^c_Q),$\label{i:weight-diam}
    \item $\sum_{Q\in\scL}w_Q(x) \leq \frac{D}{1-q}\chi_{U^c_\scL}(x)\ \  \text{ for almost every $x\in\Sigma$}$, \label{i:weight-sum}
    \item $\supp(w_Q) \subseteq U^c_Q\cap\Sigma$.\label{i:weight-support}
\end{enumerate}
\end{lem}
\begin{proof}
We will suppress the superscript of the cores and write $U_Q = U_Q^c$. Fix $Q\in \scL$. For any set $E$ and function $w:\Sigma\rightarrow\R$, we let $w(E) = \int_E wd\ell$. Let $\mathcal{F}_n$ be the $\sigma$-algebra generated by $\cup_{k\leq n}\{U_{Q'}\cap\Sigma\}_{Q'\in C^k(Q)}$. We construct the function $w_Q$ as the pointwise limit of a martingale $(w^n_Q)_{n\geq 0}$ adapted to the filtration $(\mathcal{F}_n)_{n\geq0}$ with underlying finite measure $\ell|_{U_Q\cap\Sigma}$.
 We begin by defining the function $w^0_Q$:
\begin{equation*}
    w^0_Q(x) \vcentcolon= \frac{\diam(U_Q)}{\ell(U_Q\cap \Sigma)}\ \text{ for any $x\in U_Q\cap\Sigma$}.
\end{equation*}
The martingale sequence will have fixed total mass $w^0_Q(U_Q) = \diam(U_Q)$. We next define
\begin{equation*}
    s_Q \vcentcolon= \sum_{Q'\in C(Q)}\diam(U_{Q'}) + D\ell(R_Q) < \infty.
\end{equation*}
Given the function $w^n_Q$, we define $w^{n+1}_Q$ by readjusting the distribution of mass inside the cores $\{U_{Q'}\}_{Q'\in C^{n+1}(Q)}$ and leaving the remainders $R_{Q''}$ constant for any ancestor balls $Q''\in C^j(Q)$ for $j < n$. Let $Q'\in C^n(Q),\ Q''\in C(Q') \subseteq C^{n+1}(Q)$, and $Q_0 \in C^j(Q)$ for some $j < n$. We define $w_Q^{n+1}(x)$ by declaring $w_Q^{n+1}$ to be constant on each of $R_{Q_0},R_{Q'},$ and $U_{Q''}\cap\Sigma$ and imposing
\begin{align}\label{e:weight-past}
    w^{n+1}_Q(R_{Q_0}) &= w^n_Q(R_{Q_0}),\\\label{e:weight-remainder}
    w^{n+1}_Q(R_{Q'}) &= w^n_Q(U_{Q'})\cdot\frac{D\ell(R_{Q'})}{s_{Q'}},\\\label{e:weight-children}
    w^{n+1}_Q(U_{Q''}) &= w^n_Q(U_{Q'})\cdot\frac{\diam(U_{Q''})}{s_{Q'}}.
\end{align}
We could find the pointwise value for $w^{n+1}$ on each set by dividing the three above equations by $\ell(R_{Q_0}),\ell(R_{Q'}),$ and $\ell(U_{Q''}\cap\Sigma)$ respectively. It follows from the definition that $w_Q^n(U_Q)$ is $\mathcal{F}_n$ measurable. In order to show that $(w^n_Q)$ is a martingale adapted to the filtration $(\mathcal{F}_n)$, we must prove
\begin{equation*}
    \mathbb{E}(w_Q^{n+1}|\mathcal{F}_n) = w_Q^n.
\end{equation*}
It suffices to show that $w^{n+1}_Q(U_{Q'}) = w_Q^n(U_{Q'})$ for any $Q'\in \cup_{k\leq n}C^k(Q')$. First, suppose $Q'\in C^n(Q)$. Then, using \eqref{e:weight-children} and \eqref{e:weight-remainder},
\begin{align*}
    w_Q^{n+1}(U_{Q'}) &= w_Q^{n+1}(R_{Q'}) + \sum_{Q''\in C(Q')}w_Q^{n+1}(U_{Q'})\\ 
    &= w^n_Q(U_{Q'})\cdot\frac{D\ell(R_{Q'})}{s_{Q'}} + \sum_{Q''\in C(Q')}w^n_Q(U_{Q'})\cdot\frac{\diam(U_{Q''})}{s_{Q'}} = w_Q^n(U_{Q'}).
\end{align*}
On the other hand, if $Q'\in C^k(Q)$ for some $k < n$, we can apply \eqref{e:filt-decomp} to $Q'$ to write $U_{Q'}$ in terms of the remainders down to level $n-1$ and the cores at level $n$ inside of $U_{Q'}$:
\begin{align*}
    w^{n+1}_Q(U_{Q'}) &= \sum_{j = 0}^{n-k-1}\sum_{Q''\in C^{j}(Q')}w^{n+1}_Q(R_{Q'}) + \sum_{Q''\in C^{n-k}(Q')}w_Q^{n+1}(U_{Q'})\\ 
    &= \sum_{j = 0}^{n-k-1}\sum_{Q''\in C^k(Q')}w^{n}_Q(R_{Q'}) + \sum_{Q''\in C^{n-k}(Q')}w_Q^{n}(U_{Q'}) = w^n_Q(U_{Q'})
\end{align*}
using the previous two cases. This also shows that $w^{n+1}_Q(U_Q) = w_Q^n(U_Q) =\ldots = \diam(U_Q)$, verifying the finite expectation condition. Hence, $(w_Q^n)_{n\geq0}$ is a positive martingale so that it converges pointwise $\ell$ almost everywhere to a function
\begin{equation*}
    w_Q(x)=\lim_{n\rightarrow\infty}w_Q^n(x).
\end{equation*}
By definition, $\supp(w_Q) \subseteq \supp(w_Q^0) = U_Q\cap \Sigma$ and $\int_Q w_Q d\ell = w_Q(U_Q) = w_Q^0(U_Q) = \diam(U_Q)$ verifying properties \ref{i:weight-diam} and \ref{i:weight-support} above. We now prove \ref{i:weight-sum}. Fix $x\in U_Q\cap\Sigma$ for which $\lim_{n\rightarrow\infty} w_Q^n(x)$ exists and suppose that $x\in R_{Q_N}\cap U_{Q_N} \subset U_{Q_{N-1}} \subset\cdots\subset U_{Q_0} = U_Q$. Then \eqref{e:weight-children} and \eqref{e:martingale-decay} imply
\begin{align*}
    \frac{w_Q(U_{Q_N})}{\diam(U_{Q_N})} = \frac{w_Q(U_{Q_{N-1}})}{s_{Q_{N-1}}} = \frac{w_Q(U_{Q_{N-1}})}{\diam(U_{Q_{N-1}})}\frac{\diam(U_{Q_{N-1}})}{s_{Q_{N-1}}} < q\frac{w_Q(U_{Q_{N-1}})}{\diam(U_{Q_{N-1}})}.
\end{align*}
Applying this $N$ times, we get
\begin{equation*}
    \frac{w_Q(U_{Q_N})}{\diam(U_{Q_N})} < q^N\frac{w_Q(U_Q)}{\diam(U_Q)} = q^N.
\end{equation*}
Therefore, using \eqref{e:weight-remainder}, we conclude
\begin{equation*}
    w_Q(x) = \frac{w_Q(R_{Q_N})}{\ell(R_{Q_N})} \leq \frac{w_Q(U_{Q_N})}{\ell(R_{Q_N})}\frac{D\ell(R_{Q_N})}{s_{Q_N}} \leq D\frac{w_Q(U_{Q_N})}{s_{Q_N}} < Dq^N.
\end{equation*}
In particular, if $x$ is contained in an infinite sequence of nested cores, then $w_Q(x) = 0$ for all $Q$. Applying the above calculation for each $Q_k,\ 0\leq k \leq N$, we see that $w_{Q_k}(x) \leq Dq^{N-k}$. . Because $\supp(w_Q) = U_Q\cap\Sigma$, we also know that $\bigcup_{Q\in\scL}\supp(w_Q) \subseteq U_{\scL}$ and we can compute
\begin{equation*}
    \sum_{Q\in\scL}w_Q(x) = \sum_{Q\in\scL} w_{Q}(x)\chi_{U_Q}(x) \leq \left(\sum_{Q\in\scL}w_Q(x)\right)\chi_{U_{\scL}}(x) \leq \left(\sum_{n=0}^\infty Dq^n\right)\chi_{U_\scL}(x) \leq  \frac{D}{1-q}\chi_{U_\scL}(x).
\end{equation*}
This concludes the proof of \ref{i:weight-sum}.
\end{proof}
\end{subsubsection}
\begin{subsubsection}{\texorpdfstring{Bound on the $\Delta_1$ sum for $\Sigma$}{}}\label{subsec:sigma-delta1}
The ideas of this section are all present in \cite{Sc07}. We present them here in greater detail out of a desire for completeness. For $M\in\N$, define
\begin{equation*}
    \Delta(M) = \{Q\in\Delta_1:2^{-M}\leq\beta_{S(Q)}(U_Q^x) < 2^{-M+1}\}.
\end{equation*}
Fix $K\in\N$ such that $1\leq K \leq MJ$ and define
\begin{equation*}
    \Delta' \vcentcolon= \Delta'(M,K) \vcentcolon= \{Q\in\Delta(M):\rad(Q) = A2^{K + MJn},\ n\in\Z\}.
\end{equation*}
Intuitively, $\Delta^\prime$ is obtained by starting at an offset $K$ and skipping all elements in $\Delta(M)$ on the nearest $MJ$ scales so that the difference in scale between adjacent levels within $\Delta^\prime$ is large. We want to apply the martingale lemma, Lemma \ref{l:martingale}, with $\scL = \Delta'$, so we need to prove the following lemma:

\begin{lem}[cf. \cite{Sc07} Lemma 3.25]\label{l:delta1-cores}
For any $Q\in\Delta'$,
\begin{equation*}
    \frac{\diam(U_Q^{xx})}{\sum_{Q'\in C(Q)}\diam(U_{Q'}^{xx}) + \ell(R_Q)} \leq \frac{1}{1+\frac{1}{10}}.
\end{equation*}
\end{lem}
\begin{proof}
(See Figure \ref{fig:delta1-martingale} for a picture of this proof). Recall that the definition of $\Delta_1$ implies $\beta_{S(Q)}(U_Q^x) \geq C_U^{-1}\beta_{S(Q)}(Q) \geq C_U^{-1}\epsilon_1\beta_\Sigma(Q)$. For any $\eta\in S(Q)$, we get the bound 
\begin{equation}\label{e:delta1-betatilde}
    \tb{\eta}\leq\epsilon_2\beta_\Sigma(Q) < \epsilon_1(10^5AC_U)^{-1}\epsilon_1\beta_\Sigma(Q) \leq 10^{-5}A^{-1}\epsilon_1\beta_{S(Q)}(U_Q^x) < 10^{-5}A^{-1}\epsilon_12^{-M+1}.
\end{equation}
Therefore, because $\gamma_Q\in S(Q)$ we conclude $\beta_{\gamma_Q}(U_Q^{x}) \leq 16A \beta_{\gamma_Q}(Q) \leq 32A\tb{\gamma_Q} \leq 10^{-3}\epsilon_1\beta_{S(Q)}(U_Q^x)$. This implies that there exists $\xi_Q\in S(Q),\ \xi_Q\not=\gamma_Q$ such that we have $y\in\xi_Q\cap U_Q^x$ with $\dist(y,\Edge(\gamma_Q)) \geq \beta_{S(Q)}(U_Q^x)\diam(U_Q^x) \geq 2^{-M}\diam(U_Q^x)$ by using the line collinear with $\Edge(\gamma_Q)$ as an approximating line for $\beta_{S(Q)}(U_Q^x)$. Define 
\begin{align*}
    \gamma' \vcentcolon=\Edge(\gamma_Q),\quad &B_{\gamma'} \vcentcolon= B(\gamma',\epsilon_12^{-M}\diam(U_Q^x)),\\
    \xi'\vcentcolon=\Edge(\xi_Q)\cap15c_0Q,\quad        &B_{\xi'} \vcentcolon= B(\xi',\epsilon_12^{-M}\diam(U_Q^x)).
\end{align*}
By Lemma \ref{l:almost-flat-crossing}, there exists $y_{\xi'}\in\xi'$ such that $|y-y_{\xi'}| \leq \tb{\xi_Q}\diam(2Q) \leq \epsilon_12^{-M-1}\diam(U_Q^{x})$ so that $\dist(y_{\xi'},\gamma') \geq 2^{-M-1}\diam(U_Q^x)$, and hence $y_{\xi'}\in 9c_0Q$.
Write $\xi'$ as the union of two subsegments $\xi' = [a_{\xi'},y_{\xi'}]\cup [y_{\xi'},b_{\xi'}]$ where $a_{\xi'}$ and $b_{\xi'}$ are the endpoints of $\xi'$. Because the line segments $[a_{\xi'},y_{\xi'}]$ and $[b_{\xi'},y_{\xi'}]$ extend in opposite directions away from $y_{\xi'}$, one of them, suppose it is $[a_{\xi'},y_{\xi'}]$, satisfies $\dist([a_{\xi'},y_{\xi'}], \gamma') \geq \dist(y_{\xi'},\gamma') \geq 2^{-M-5}\diam(U_Q^{xx})$ also using the fact that $\gamma'$ is a line segment. In addition, $[a_{\xi'},y_{\xi'}]$ has nonempty intersection with both $9c_0Q$ and $15c_0Q\subseteq U_Q^{xx}$ so that we can assume both of the following hold:
\begin{align}
    \dist([a_{\xi'}, y_{\xi'}],\gamma') &\geq 2^{-M-5}\diam(U_Q^{xx}),\label{e:xi-seg-far}\\
    \diam([a_{\xi'},y_{\xi'}]) &\geq 6c_0\rad(Q) = 3c_0\diam(Q)\label{e:xi-seg-large}.
\end{align}
Therefore, we can apply Lemma $\ref{l:almost-flat-crossing}$ to the segment $[a_{\xi'},y_{\xi'}]$ to get an arc $\xi_0\subseteq \xi_Q$ such that $\xi_0\subseteq B_{\xi'} \subseteq 16c_0Q\subseteq U_Q^{xx}$ and $\Diam(\xi_0) \geq 3c_0\diam(Q)$. Now, $Q\in\Delta'$ implies, for all $Q'\in C(Q)$,
\begin{equation}\label{e:M-small-child}
    \diam(U_{Q'}^{xx})\leq \diam(Q') \leq 2^{-MJ}\diam(Q) \leq 2^{-M}\cdot2^{-J+2}\diam(Q) \leq 2^{-M}\epsilon_1\diam(U_Q^{xx})
\end{equation}
Using \eqref{e:xi-seg-far} and the fact that $\gamma_Q\subseteq B_{\gamma'}$, this implies that $\xi_0$ satisfies the following:
\begin{align*}
    U_{Q'}^{xx}\cap\gamma_Q &= \varnothing\text{ for all } Q'\in C(Q) \text{ such that } U_{Q'}^{xx}\cap\xi_0\not=\varnothing.
\end{align*}
Hence, we can estimate
\begin{align}\nonumber
    \sum_{Q'\in C(Q)}\diam(U_{Q'}^{xx}) + \ell(R_Q) &\geq \sum_{\substack{Q'\in C(Q) \\\nonumber U_{Q'}^{xx}\cap\gamma_Q\not=\varnothing}}\diam(U_{Q'}^{xx}) + \ell(R_Q\cap\gamma_Q) + \sum_{\substack{Q'\in C(Q) \\\nonumber U_{Q'}^{xx}\cap\xi_0\not=\varnothing}}\diam(U_{Q'}^{xx}) + \ell(R_Q\cap\xi_0)\\\nonumber
    &\geq \diam(\gamma_Q\cap U_Q^{xx}) + \Diam(\xi_0) \geq 15c_0\diam(Q) + 3c_0\diam(Q)\\
    &\geq \left(1+\frac{1}{10}\right)\diam(U_Q^{xx}).\label{e:delta1-diam-lb}
\qedhere\end{align}
\end{proof}
\begin{figure}[h]
    \centering
    \includegraphics[scale=0.9]{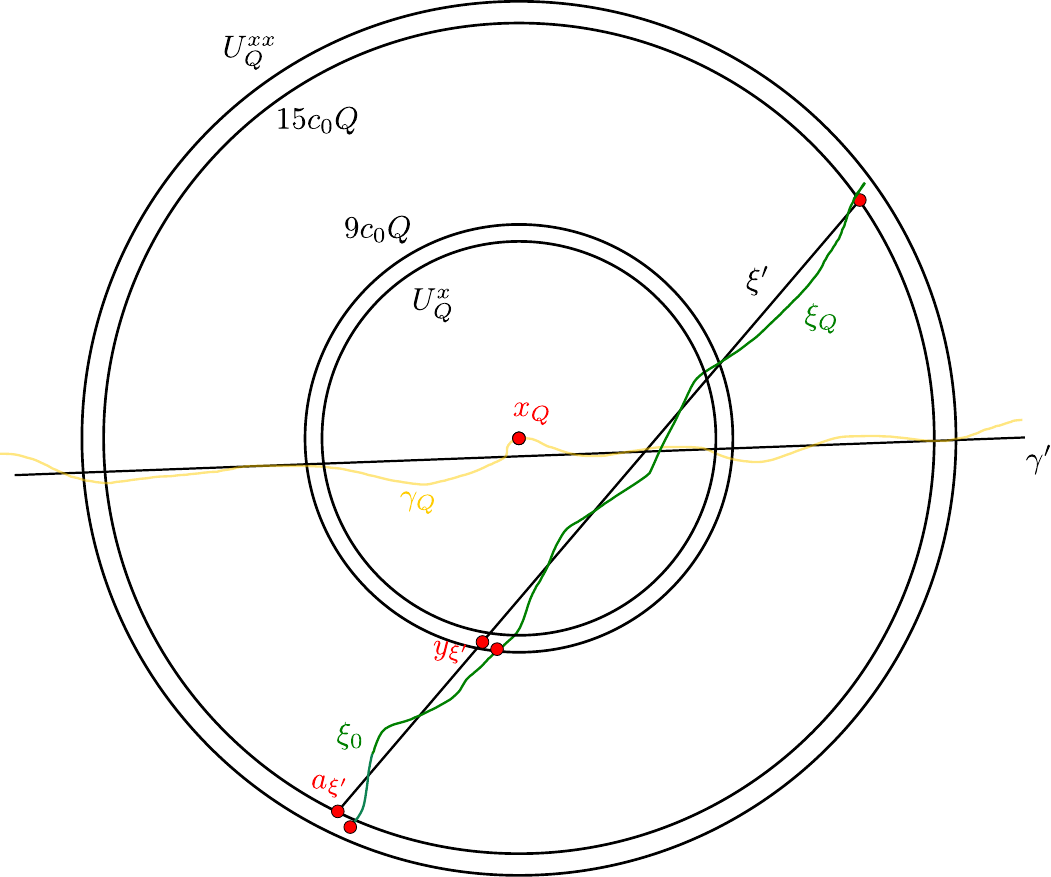}
    \caption{A picture of the proof of Lemma \ref{l:delta1-cores}.}
    \label{fig:delta1-martingale}
\end{figure}
\begin{prop}[cf. \cite{Sc07} Lemma 3.25]\label{p:delta1-sigma-bound}
\begin{align*}
    \sum_{Q\in\Delta_1}\beta_\Sigma(Q)^2\diam(Q) &\lesssim_{A,J} \sH^1(\Sigma).
\end{align*}
\end{prop}
\begin{proof}
Fix $\Delta^\prime(M,K)\subseteq\Delta_1$ as defined above and order it via the forest $\mathcal{T}_{\Delta'}^{16c_0}$. By Lemma \ref{l:delta1-cores}, we can apply Lemma \ref{l:martingale} with $\scL = \Delta',\ D = 1,\ q = \frac{1}{1+\frac{1}{10}}$ to get a collection of positive real-valued functions $\{w_Q\}_{Q\in\Delta'}$ such that
\begin{enumerate}[label=(\roman*)]
    \item $\int_Q w_Q\ d\ell = \diam(U_Q^{xx}),$ \text{and}
    \item $\sum_{Q\in\Delta'}w_Q(x) \lesssim \chi_{U^{16c_0}_{\Delta'}}(x)\text{ for almost every $x\in\Sigma$}$,
\end{enumerate}
Therefore, we have
\begin{align*}
    \sum_{Q\in\Delta^\prime}\beta_\Sigma(Q)\diam(Q) &\lesssim_A 2^{-M}\sum_{Q\in\Delta^\prime}\int_Q w_Q d\ell \leq 2^{-M}\int_\Gamma \sum_{Q\in\Delta^\prime}w_Q d\ell \\
    &\lesssim2^{-M}\int_{U^{16c_0}_{\Delta'}\cap\Sigma}d\ell \leq 2^{-M}\sum_{T\in\mathcal{T}_{\Delta'}^{16c_0}}\ell(U^{xx}_{Q(T)}\cap\Sigma) \lesssim 2^{-M}\ell(\Sigma)
\end{align*}
where the final inequality follows because the collection $\{U^{xx}_{Q(T)}\}_{T\in\mathcal{T}_{\Delta'}^{16c_0}}$ is pairwise disjoint. Summing this over $M\geq 0$ and $1 \leq K \leq MJ$, we get
\begin{equation*}
    \sum_{Q\in\Delta_1}\beta_\Sigma(Q)^2\diam(Q) \leq \sum_{M=0}^\infty \sum_{K=1}^{MJ}\sum_{Q\in\Delta^\prime(M,K)}\beta_\Sigma(Q)\diam(Q) \lesssim_{A,J} \sum_{M=0}^\infty M2^{-M}\ell(\Sigma) \lesssim \sH^1(\Sigma).\qedhere
\end{equation*}
\end{proof}
\end{subsubsection}
\begin{subsubsection}{\texorpdfstring{Bound on the $\Delta_{2.1}$ sum for $\Sigma$}{}}\label{subsec:sigma-delta2.1}
Our proof of the beta sum bound for $\Delta_{2.1}$ relies heavily on the construction of weights analogous to those of Proposition \ref{p:delta1-sigma-bound} adapted to $\Delta_{2.1}$ balls rather than $\Delta_1$ balls. Unfortunately, the proof of the existence of these weights in \cite{Sc07} Lemma 3.28 contains technical errors which leave gaps in the proof (see \cite{BM22b} Appendix C. for further explanation of the issues). In this section, we provide a new proof.

We begin with a general lemma which gives a nice approximating line segment for almost flat $\gamma_Q$ inside the core of a general ball $Q\in\scG$ for a range of core sizes. We will use this line segment $\gamma'$ as an accounting tool for proving the analogue of Lemma \ref{l:delta1-cores} for $\Delta_{2.1}$ balls.

\begin{lem}\label{l:gamma'}
Let $Q\in\scG$ such that $\gamma_Q\in S(Q)$ and fix $\frac{2\epsilon_2}{\epsilon_1} < c < \frac{1}{4A}$. There exists a line segment $\gamma' = [a_{\gamma'},b_{\gamma'}] \subseteq \Edge(\gamma_Q)$ such that
\begin{enumerate}[label=(\roman*)]
    \item $\sH^1(\gamma') \geq (1-30\epsilon_1)\diam(U^c_Q)$, \label{i:gamma'-diam}
    \item $B(\gamma',5\epsilon_1\diam(U^c_Q))\subseteq U^c_Q$, \label{i:gamma'-nbhd-containment}
    \item $\gamma' \subseteq \pi_{\gamma'}(\gamma_Q\cap cQ)$, \text{and} \label{i:gamma'-proj}
    \item If $\beta_{\Sigma}(U_Q^c) < \epsilon_1$, then for any $Q'\in C(Q),\ \pi_{\gamma'}(U^c_{Q'}) \cap \gamma' \not=\varnothing \implies 2Q'\subseteq U^c_Q$. \label{i:gamma'-covered-nearby}
\end{enumerate}
\end{lem}
\begin{proof}
Let $\gamma'' \vcentcolon= [a_{\gamma''},b_{\gamma''}] \vcentcolon= \Edge(\gamma_Q)\cap cQ$. Define $\gamma'$ to be the line segment gotten by chopping off the segments of length $10\epsilon_1\diam(U^c_Q)$ from either end of $\gamma''$:
\begin{equation*}
    \gamma' \vcentcolon= \left[ a_{\gamma''} - 10\epsilon_1\diam(U^c_Q)\frac{b_{\gamma''} - a_{\gamma''}}{|b_{\gamma''} - a_{\gamma''}|}, b_{\gamma''} - 10\epsilon_1\diam(U^c_Q)\frac{a_{\gamma''} - b_{\gamma''}}{|b_{\gamma''} - a_{\gamma''}|} \right]=\vcentcolon [a_{\gamma'},b_{\gamma'}].
\end{equation*}
Because $x_Q\in\gamma_Q\in S(Q)$ implies $\dist(x_Q,\gamma'') < 2\epsilon_2\diam(Q) \leq \epsilon_1c\rad(Q)$, and $a_{\gamma''},b_{\gamma''} \in \partial(cQ)$, Lemma \ref{l:diam-low-bound} implies
\begin{align*}
    \sH^1(\gamma'') &= |a_{\gamma''} - b_{\gamma''}| \geq c\diam(Q)\left(1 - 2\epsilon_1^2\right) \geq (1-2\epsilon_1)\diam(U^c_Q)
\end{align*}
where we used the fact that $c\diam(Q) \geq \frac{1}{1+2^{-J+2}}\diam(U^c_Q) \geq (1-\epsilon_1)\diam(U^c_Q)$. This gives
\begin{align*}
    \sH^1(\gamma') &\geq \sH^1(\gamma'') - 20\epsilon_1\diam(U^c_Q) \geq (1-30\epsilon_1)\diam(U^c_Q).
\end{align*}
This proves \ref{i:gamma'-diam}. In a similar vein, we can use the Pythagorean theorem to estimate
\begin{align*}
    |a_{\gamma'}-x_Q| &\leq \sqrt{(1-10\epsilon_1)^2\diam(U^c_Q)^2 + \epsilon_1^2\diam(U^c_Q)^2} \leq \diam(U^c_Q)\sqrt{1-20\epsilon_1 + 101\epsilon_1^2}\\
    &\leq \diam(U^c_Q)\sqrt{1-19\epsilon_1} \leq (1-9\epsilon_1)\diam(U^c_Q)
\end{align*}
using the Taylor expansion estimate $\sqrt{1-x} = 1 - \frac{x}{2} - \frac{x^2}{8} - \ldots \leq 1-\frac{x}{2}$. A similar inequality holds for $b_{\gamma'}$, implying
\begin{equation}\label{e:xiQ-nbhd}
    B(\gamma',5\epsilon_1\diam(U^c_Q))\subseteq cQ \subseteq U^c_Q
\end{equation}
by the triangle inequality, the convexity of balls in $\ell_2$, and the fact that $\gamma'$ is a line segment. This proves \ref{i:gamma'-nbhd-containment}. To prove \ref{i:gamma'-proj}, we observe that $\tb{\gamma_Q}\diam(2Q) \leq 2\epsilon_2\diam(Q) < c\epsilon_1\diam(Q) \leq \epsilon_1\diam(U^c_Q)$ and apply Lemma \ref{l:almost-flat-crossing} to the segment $\gamma'$ to get a subarc $\gamma_0\subseteq\gamma_Q$ such that $\Domain(\gamma_0)\subseteq\Domain(\gamma_Q)\cap \gamma^{-1}(C(\gamma',\epsilon_1\diam(U^c_Q)))$ such that $\pi_{\gamma'}(\gamma_0) = \gamma'$. We now prove \ref{i:gamma'-covered-nearby}. We compute
\begin{equation}\label{e:child-diam}
    \dist(\pi_{\gamma'}(x_{Q'}),\gamma') \leq \diam(2Q') \leq 2^{-J+1}\diam(Q) < 2(10A)^{-2}\epsilon_1\diam(Q) \leq \epsilon_1\diam(U^c_Q).
\end{equation}
On the other hand, $\beta_{\Sigma}(U_Q^c) < \epsilon_1$ implies $|\pi_{\gamma'}^\perp(x_{Q'})| \leq 2\epsilon_1\diam(U_Q^c)$ which combined with \eqref{e:child-diam} gives $2Q'\subseteq B(\gamma',5\epsilon_1\diam(U_Q^c)) \subseteq U_Q^c$ by \ref{i:gamma'-nbhd-containment}.
\end{proof}
For the rest of this section, we consider $\Delta_{2.1}$ with the ordering given by the forest $\mathcal{T}_{\Delta_{2.1}}^{c_0}$. We now identify a good family of balls from which we will extract excess length in order to prove the existence of $q < 1$ such that $\diam(U_Q)\leq qs_Q$ for $Q\in\Delta_{2.1}$. 
\begin{defn}[Dominant balls]
Fix $Q\in\Delta_{2.1}$ and let $\gamma'$ be as in Lemma \ref{l:gamma'}. Define the ``interior'' and ``exterior'' children as 
\begin{align*}
    C_I(Q) &\vcentcolon= \{Q'\in C(Q) : \pi_{\gamma'}(U_{Q'})\cap\gamma'\not=\varnothing\},\\
    C_E(Q) &\vcentcolon= \{Q'\in C(Q) : \pi_{\gamma'}(U_{Q'})\cap\gamma'=\varnothing\},\\
    &= C(Q) \setminus C_I(Q).
\end{align*}
We have $\bigcup_{Q'\in C_I(Q)}2Q'\subseteq U_Q$ by Lemma \ref{l:gamma'} \ref{i:gamma'-covered-nearby}. For any $Q'\in C_I(Q)$, one of the following two properties holds:
\begin{enumerate}[label=(\roman*)]
    \item For all $Q''\in C(Q)$ such that $U_{Q''}\cap 2Q'\not= \varnothing$, $\diam(Q'')\leq\diam(Q')$, \text{ or}\label{i:dom}
    \item There exists $Q''\in C(Q)$ such that $U_{Q''}\cap 2Q'\not= \varnothing$ and $\diam(Q'') > \diam(Q')$ \label{i:minor}
\end{enumerate}
Define the ``dominant'' and ``minor'' balls as
\begin{align*}
    C_D(Q) &\vcentcolon= \{Q'\in C_I(Q) : \text{$Q'$ satisfies \ref{i:dom}}\},\\
    C_M(Q) &\vcentcolon= \{Q'\in C_I(Q) : \text{$Q'$ satisfies \ref{i:minor}}\}\\
    &= C_I(Q) \setminus C_D(Q).
\end{align*}
\end{defn}
The balls in $C_D(Q)$ have dominant projections on $\gamma'$ in the sense of the following lemma:
\begin{lem}\label{l:big-proj}
Let $Q\in\Delta_{2.1}$ be such that $\ell(R_Q) < \epsilon_1\diam(U_Q)$. Then
\begin{equation*}
    \sH^1\left(\bigcup_{Q'\in C_D(Q)}\pi_{\gamma'}(U_{Q'}) \right)  \geq (1-50\epsilon_1)\diam(U_Q).
\end{equation*}
\end{lem}
\begin{proof}
The collection $\{\pi_{\gamma'}(U_{Q'})\}_{Q'\in C_I(Q)}\cup\{\pi_{\gamma'}(R_Q\cap\gamma_Q)\}$ is a covering of $\gamma'$ by Lemma \ref{l:gamma'} \ref{i:gamma'-proj} and \ref{i:gamma'-covered-nearby}. Hence,
\begin{align*}
\sH^1\left( \bigcup_{Q'\in C_I(Q)}\pi_{\gamma'}(U_{Q'})\right) + \ell(R_Q\cap\gamma_Q)&\geq \sH^1\left( \bigcup_{Q'\in C_I(Q)}\pi_{\gamma'}(U_{Q'}) \right) + \sH^1(\pi_{\gamma'}(R_Q\cap\gamma_Q))\\
&\geq \sH^1(\gamma') \geq (1-30\epsilon_1)\diam(U_Q)
\end{align*}
by Lemma \ref{l:gamma'} \ref{i:gamma'-diam}. Using the fact that $\ell(R_Q) < \epsilon_1\diam(U_Q)$, we get $\sH^1\left( \bigcup_{Q'\in C_I(Q)}\pi_{\gamma'}(U_{Q'}) \right) \geq (1-31\epsilon_1)\diam(U_Q)$. The rest of the proof amounts to showing that the projections of cores of balls in the subfamily $C_D(Q)\subseteq C_I(Q)$ cover most of the projected cores of balls in $C_I(Q)$.

We begin by defining a many-to-one mapping $\psi:C_M(Q)\rightarrow C(Q)$. Fix $Q_0\in C_M(Q)$. By definition, there exists some $Q_1\in C(Q)$ such that $U_{Q_1}\cap 2Q_0\not=\varnothing$ and $\diam(Q_1) > \diam(Q_0)$. If $Q_1\in C_D(Q)\cup C_E(Q)$, then define $\psi(Q_0) = Q_1$. Otherwise, $Q_1\in C_M(Q)$ and, applying the same logic to $Q_1$ as we did to $Q_0$, we get the existence of $Q_2\in C(Q)$  satisfying condition (ii) for $Q_1$. Repeating this argument recursively, we get a finite chain of balls $Q_0,Q_1,Q_2,\ldots, Q_N$ with strictly increasing diameter such that $Q_N\in C_D(Q)\cup C_E(Q)$ and $Q_0,\ldots,Q_{N-1}\in C_M(Q)$ with $2Q_i\cap U_{Q_{i+1}}\not=\varnothing$ (the chain must be finite because there is an absolute upper bound on the diameter for balls in $C(Q)$). Set $\psi(Q_0) = \psi(Q_1) = \cdots = \psi(Q_{N-1}) = Q_N$.

Now, let $x\in Q_0\in C_M(Q)$ with the above described chain $Q_0,\ldots,Q_{N-1},\psi(Q_0)$. By the triangle inequality, we get
\begin{align*}
    \dist(x,U_{\psi(Q_0)}) &\leq \sum_{i=0}^{N-1}\diam(2Q_i) \leq 2^{-J+1}\sum_{i=0}^{N-1}(2^{-J})^i\diam(\psi(Q_0)) \leq \frac{2^{-J+1}}{1-2^{-J}}\diam(\psi(Q_0))\\
    &< \epsilon_1\diam(U_{\psi(Q_0)}).
\end{align*}
This implies that for any $Q'\in C_D(Q)\cup C_E(Q)$, the set of balls $Q''\in C_M(Q)$ such that $\psi(Q'') = Q'$ are contained in a neighborhood of radius $\epsilon_1\diam(U_{Q'})$ around $U_{Q'}$. Therefore, the projection $\pi_{\gamma'}(U_{Q'})$ is an interval of length at least $c_0\diam(Q') \geq (1-\epsilon_1)\diam(U_{Q'})$ while the set $\left(\bigcup_{\psi(Q'') = Q'}\pi_{\gamma'}(U_{Q''})\right) \setminus \pi_{\gamma'}(U_{Q'})$ is contained in the union of two intervals of length $\epsilon_1\diam(U_{Q'})$ adjoined to either end of $\pi_{\gamma'}(U_{Q'})$. This means that if $Q'\in C_M(Q)$ is such that $\psi(Q')\in C_E(Q)$, then $\pi_{\gamma'}(U_{Q''})\cap\gamma'$ is contained in an interval of width less than $\epsilon_1\diam(U_{Q'}) \leq \epsilon_1\diam(U_Q)$ containing one of the two endpoints of $\gamma'$. Hence,
\begin{equation*}
    \sH^1\left( \bigcup_{\substack{Q'\in C_M(Q) \\ \psi(Q')\in C_E(Q)}}\pi_{\gamma'}(U_{Q'}) \right) \leq 2\epsilon_1\diam(U_{Q}).
\end{equation*}
This implies
\begin{align*}
    \sH^1&\left(\bigcup_{Q'\in C_M(Q)}\pi_{\gamma'}(U_{Q'}) \setminus \bigcup_{Q'\in C_D(Q)}\pi_{\gamma'}(U_{Q'}) \right)\\
    &\quad\quad\leq\sH^1\left(\bigcup_{\substack{Q'\in C_M(Q) \\ \psi(Q')\in C_D(Q)}}\pi_{\gamma'}(U_{Q'}) \setminus \bigcup_{Q'\in C_D(Q)}\pi_{\gamma'}(U_{Q'}) \right) + \sH^1\left( \bigcup_{\substack{Q'\in C_M(Q) \\ \psi(Q')\in C_E(Q)}}\pi_{\gamma'}(U_{Q''}) \right)  \\
    &\quad\quad\leq 4\epsilon_1 \sH^1\left(\bigcup_{Q'\in C_D(Q)}\pi_{\gamma'}(U_{Q'}) \right) + 2\epsilon_1\diam(U_Q).
\end{align*}
Therefore, we conclude
\begin{align*}
    (1-31\epsilon_1)\diam(U_Q) &\leq \sH^1\left(\bigcup_{Q'\in C_I(Q)}\pi_{\gamma'}(U_{Q'})\right) \\ 
    &\leq\sH^1\left( \bigcup_{Q'\in C_D(Q)}\pi_{\gamma'}(U_{Q'})  \right) + \sH^1\left( \bigcup_{Q'\in C_M(Q)}\pi_{\gamma'}(U_{Q'}) \setminus \bigcup_{Q'\in C_D(Q)}\pi_{\gamma'}(U_{Q'}) \right)\\
    &\leq (1+4\epsilon_1)\sH^1\left(\bigcup_{Q'\in C_D(Q)}\pi_{\gamma'}(U_{Q'})\right) + 2\epsilon_1\diam(U_Q)
\end{align*}
from which we get the result.
\end{proof}
We now want to show that each ball $Q'\in C(Q)$ has double $2Q'$ which contains a significant amount of excess length which contributes to the value of $s_Q$. We begin by isolating an almost flat arc $\tau_{Q'}$ of large diameter which does not overlap with $\gamma_Q$ too badly.
\begin{remark}[Existence of $\tau_{Q'}$]\label{rem:tau-q}
Fix $Q'\in\Delta_{2.1}$. Because $\beta_{S_{Q'}}(U_{Q'}^x) < C_U^{-1}\beta_{S_{Q'}}(Q')$, there must exist an arc $\tau_{Q'}\in S_{Q'}$ such that both
\begin{enumerate}[label=(\roman*)]
    \item $\tau_{Q'}\cap U_{Q'}^x = \varnothing$, \text{and}\label{i:tau-far}
    \item $\beta_{\gamma_{Q'}\cup\tau_{Q'}}(2{Q'}) \gtrsim_A 1$. \label{i:tau-beta}
\end{enumerate}
Intuitively, we think of $\tau_{Q'}$ as an ``additional'' arc alongside $\gamma_{Q'}$ which makes a significant contribution to $\sum_{Q''\in C(Q)}\diam(U_{Q''}) + \ell(R_Q)$ inside $2Q'$ because it carries a large number of child cores disjoint from those on $\gamma_Q$.
\end{remark}
In order to estimate the core diameter sum, we will use line segment approximations to $\tau_{Q'}$ and $\gamma_{Q'}$ with the idea of first isolating appropriate subsegments which are far apart, then applying Lemma \ref{l:almost-flat-crossing} to get associated arcs which are far apart. We define $\tau'\vcentcolon=\Edge(\tau_{Q'})\cap (1-3c_0)2Q'$. Then $\tau'$ is a line segment with endpoints in the boundary of $(1-3c_0)2Q'$ such that $\tau'\cap (1+c_0)Q'\not=\varnothing$ because $\tb{\tau_{Q'}} \leq \epsilon_2$ and $\tau\cap Q\not=\varnothing$ because $\tau\in\Lambda(Q)$. Similarly, we define $\eta'\vcentcolon=\Edge(\gamma_{Q'})\cap(1-3c_0)2Q'$. Because $x_{Q'}\in\gamma_{Q'}$ and $\gamma_{Q'}\in S_{Q'}$, we have $\dist(\eta',x_{Q'}) \leq \tb{\gamma_Q}\diam(2Q') \leq \epsilon_2\diam(2Q') = \frac{\epsilon_2}{(1-3c_0)}\cdot(1-3c_0)\diam(2Q')$ so that we can apply Lemma \ref{l:diam-low-bound} and receive 
\begin{align}\label{e:gamma-diam}\nonumber
    \diam(\eta') &\geq \left(1 - 2\left(\frac{\epsilon_2}{1-3c_0}\right)^2\right)(1-3c_0)\diam(2Q') \geq (1-8\epsilon_2^2)(1-3c_0)\diam(2Q') \geq \frac{999}{1000}\diam(2Q')
\end{align}
because $(1-8\epsilon_2^2)(1-3c_0) \geq (1-4c_0) \geq (1 - 4\cdot10^{-4}) \geq \frac{999}{1000}$. We will use $\eta'$ and $\tau'$ in the following lemma:
\begin{lem}\label{l:delta2.1-excess}
Let $Q\in\Delta_{2.1}$. For any $Q'\in C_D(Q)$,
\begin{equation*}
    \sum_{\substack{Q''\in C(Q) \\ U_{Q''}\subseteq 2Q' \\ U_{Q''}\cap(\gamma_{Q'}\cup\tau_{Q'})\not=\varnothing}}\diam(U_{Q''}) + \ell(R_Q\cap 2Q') \geq \left(1+\frac{1}{10}\right)\diam(2Q').
\end{equation*}
\end{lem}
\begin{proof}
Our plan is to apply Lemma \ref{l:almost-flat-crossing} to $\eta'$ and a large diameter subsegment $\tau''\subseteq\tau'$ which is far from $\eta'$ to get arcs $\gamma_0\subseteq\gamma_{Q'}$ and $\tau_0\subseteq\tau_{Q'}$ such that no child core of $Q$ touches both (See Figure \ref{fig:delta2.1-martingale-global} for a picture of the proof). Because $\tb{\gamma_{Q'}}\diam(2Q') \leq \epsilon_2\diam(2Q') \leq \epsilon_1\diam(U_{Q'})$, we can apply Lemma \ref{l:almost-flat-crossing} to the segment $\eta'$ to get an arc $\gamma_0\subseteq\gamma_{Q'}$ such that $\Domain(\gamma_0)\subseteq \Domain(\gamma_{Q'})\cap\gamma^{-1}(C(\eta',\epsilon_1\diam(U_{Q'})))$ with $\Diam(\gamma_0) \geq \diam(\eta') \geq \frac{999}{1000}\diam(Q)$. We claim that for any $Q''\in C(Q),$ 
\begin{equation}\label{e:core-contained-2Q}
    U_{Q''}\cap\gamma_0\not=\varnothing \implies U_{Q''}\subseteq 2Q'.
\end{equation}
For proof, first note that because $Q'\in C_D(Q)$, $U_{Q''}\cap 2Q'\not=\varnothing$ implies $\diam(Q'')\leq\diam(Q')$ so that
\begin{equation}\label{e:dominant-diam}
    \diam(U_{Q''}) \leq (1+2^{-J+2})c_0\diam(Q') \leq (1+\epsilon_1)c_0\diam(Q') = (1+\epsilon_1)c_0\rad(2Q').
\end{equation}
Hence, if $Q''\in C(Q)$ such that $U_{Q''}\cap\gamma_0\not=\varnothing$, then $U_{Q''}\cap(1-\frac{3}{2}c_0)2Q'\not=\varnothing$ so that any $x\in U_{Q''}$ satisfies
\begin{align*}
    \dist(x,x_Q) &\leq \left(1-\frac{3}{2}c_0\right)\rad(2Q') + \diam(U_Q'') \leq \left(1-\frac{3}{2}c_0\right)\rad(2Q') + (1+\epsilon_1)c_0\rad(2Q') \leq \rad(2Q')
\end{align*}
so that $U_{Q''}\subseteq 2Q'$. Because $Q'\in C_D(Q)$, we also know that $\gamma_0\subseteq \gamma_{Q'}\subseteq 2Q'\subseteq U_Q$ so that the family $\{U_{Q''} : Q''\in C(Q),\ U_{Q''}\cap\gamma_0\not=\varnothing\}\cup\{R_Q\cap\gamma_0\}$ is a covering of $\gamma_0$. We can then estimate
\begin{align*}
    \sum_{\substack{Q''\in C(Q) \\ U_{Q''}\cap\gamma_{Q'}\not=\varnothing}}\diam(U_{Q''}) + \ell(R_Q\cap \gamma_{Q'}) &\geq \sum_{\substack{Q''\in C(Q) \\ U_{Q''}\subseteq 2Q' \\ U_{Q''}\cap\gamma_0\not=\varnothing}}\diam(U_{Q''}) + \ell(R_Q\cap \gamma_0)\\
    &\geq\Diam(\gamma_0) \geq \frac{999}{1000}\diam(2Q').
\end{align*}
We want to apply a similar argument to $\tau'$, this time finding an arc $\tau_0$ which lies close to a subsegment $\tau''$ of $\tau'$ which is far from $\eta'$, hence from $\gamma_Q$. Indeed, suppose first that there exists a point $y_{\tau'}\in\tau'\cap(1+c_0)Q$ such that $\dist(y_{\tau'},\eta') \geq 7c_0\rad(Q')$. Write $\tau'$ as the union of two subsegments $\tau' = [a_{\tau'},y_{\tau'}]\cup[y_{\tau'},b_{\tau'}]$ where $a_{\tau'},b_{\tau'}$ are the endpoints of $\tau'$. Because the line segments $[a_{\tau'},y_{\tau'}]$ and $[y_{\tau'},b_{\tau'}]$ extend in opposite directions away from $y_{\tau'}$, we know that one of them, suppose it is $[a_{\tau'},y_{\tau'}]$, satisfies $\dist([a_{\tau'},y_{\tau'}],\eta') \geq \dist(y_{\tau'},\eta') \geq 7c_0\rad(Q)$ using the fact that $\eta'$ is a line segment. Set $\tau''\vcentcolon= [a_{\tau'},y_{\tau'}]$. This completes the definition of $\tau''$ in the first case.

If instead there is no such point $y_{\tau'}\in (1+c_0)Q'\cap\tau'$, then $(1+c_0)Q'\cap\tau'\subseteq B(\eta',7c_0\rad(Q'))$. We claim that $\tau'$ is nearly perpendicular to $\eta'$. Indeed, consider $E \vcentcolon= \partial((1+c_0)Q')\cap\tau'$ and let $C_1,C_2$ be the two connected components of the set $B(\eta',7c_0\rad(Q'))\cap\partial((1+c_0)Q')$. First, we claim there cannot exist distinct points $e_1,e_2\in E$ such that $e_1\in C_1$ and $e_2\in C_2$. If there did exist such points, then because $\tau'$ is a line segment and $B(\eta',7c_0\rad(Q'))$ is convex, there would exist $e'\in \tau'\cap B(\eta',7c_0\rad(Q'))$ with $\pi_{\eta'}(e) = \pi_{\eta'}(x_{Q'})$ so that $e'\in \frac{15}{2}c_0Q'$. Hence, we would have $\tau_{Q'}\cap U_{Q'}^x\not=\varnothing$, contradicting the definition of $\tau_{Q'}$. Therefore, without loss of generality we can assume that $E\subseteq C_1$.

Let $P\subseteq H$ be the affine plane containing the line segments $\eta'$ and $\tau'$ (this is at most $3$-dimensional). By translating and rotating, we can assume without loss of generality that $x_{Q'} = 0$ and $\eta'$ is collinear with the $x_1$-axis so that 
\begin{equation*}
    E\subseteq S \vcentcolon= \left\{x\in P : x_1 >0,\ |x| = (1+c_0)\rad(Q'),\ |x^\perp| \leq \frac{1}{1000}\rad(Q')\right\} 
\end{equation*}
where $|x^\perp|^2 = |x|^2 - |x_1|^2$, and we have used the fact that $8c_0 < \frac{1}{1000}$. The set $S$ is a small spherical cap of the (at most $2$-dimensional) sphere $\{x\in P : |x| = (1+c_0)\rad(Q')\}$ around the point $((1+c_0)\rad(Q'),0,0,\ldots)$. Fix $e_{\tau'}\in E$. We can write $[a_{\tau'},e_{\tau'}] = \{e_{\tau'} + tv : 0\leq t \leq |a_{\tau'}-e_{\tau'}|\}$ where $|v| = 1$ and we claim $v$ is parallel to a tangent vector to $S$. Indeed, if $\#(E) = 1$, then $\tau'$ is tangent to $S$ while if $\#(E) = 2$, then $\tau'\cap (1+c_0)Q'$ is a line segment with two endpoints in $S$ and the claim follows from considering $S$ as a graph over the plane $\{x_1 = 0\}$ and applying the mean value theorem (geometrically, one can imagine translating the line segment to be tangent to $S$). 

One can compute by implicit differentiation in $S\subseteq P \subseteq \R^3$ that $\frac{|v^\perp|}{|v|} \geq \frac{1}{2}.$ We can also assume $\dist(e_{\tau'} + tv,\eta')$ is increasing in $t$ by exchanging $[a_{\tau'},e_{\tau'}]$ with $[e_{\tau'},b_{\tau'}]$ or $v$ with $-v$ if necessary. Therefore, $\dist(e_{\tau'} + (20c_0\rad(Q'))v, \eta') \geq 7c_0\rad(Q')$. Define $\tau'' \vcentcolon= [a_{\tau'},e_{\tau'} + (20c_0\rad(Q'))v]$.

With $\tau''$ defined as in either of the two cases above, we get the following two lower bounds:
\begin{equation}\label{e:tau''-far-eta'}
    \dist(\tau'',\eta') \geq 7c_0\rad(Q'),
\end{equation}
\begin{align*}
    \diam(\tau'') &\geq 2(1-3c_0)\rad(Q') - (1+c_0)\rad(Q') - 20c_0\rad(Q') \\
    &\geq \frac{1}{4}\left(1 - 27c_0\right)\diam(2Q') \geq \frac{1}{5}\diam(2Q').
\end{align*}
Applying Lemma \ref{l:almost-flat-crossing} to the segment $\tau''$, we get an arc $\tau_0\subseteq C(\tau'',\epsilon_1c_0\rad(Q'))$ with $\Diam(\tau_0) \geq \frac{1}{5}\diam(2Q')$. Therefore, we conclude from \eqref{e:tau''-far-eta'} and \eqref{e:dominant-diam} that $U_{Q''}\cap\tau_0\not=\varnothing$ implies $U_{Q''}\cap\gamma_0=\varnothing$ and $U_{Q''}\subseteq 2Q'$ as in \eqref{e:core-contained-2Q} so that we can estimate
\begin{align*}
    \sum_{\substack{Q''\in C(Q) \\ U_{Q''}\cap\gamma_{Q'} = \varnothing \\ U_{Q''}\cap\tau_{Q'}\not=\varnothing}}\diam(U_{Q''}) + \ell(R_Q\cap \tau_{Q'}) &\geq \sum_{\substack{Q''\in C(Q) \\ U_{Q''}\subseteq 2Q' \\ U_{Q''}\cap\tau_0\not=\varnothing}}\diam(U_{Q''}) + \ell(R_Q\cap \tau_0)\\
    &\geq \Diam(\tau_0) \geq \frac{1}{5}\diam(2Q').
\end{align*}
By summing the estimates for $\gamma_0$ and $\tau_0$, we conclude
\begin{equation*}
    \sum_{\substack{Q''\in C(Q) \\ U_{Q''}\subseteq 2Q' \\ U_{Q''}\cap(\gamma_{Q'}\cup\tau_{Q'})\not=\varnothing}}\diam(U_{Q''}) + \ell(R_Q\cap 2Q') \geq \left(\frac{999}{1000} + \frac{1}{5}\right)\diam(2Q') \geq \left(1+\frac{1}{10}\right)\diam(2Q').
\qedhere\end{equation*}
\begin{figure}
    \centering
    \includegraphics[scale=0.7]{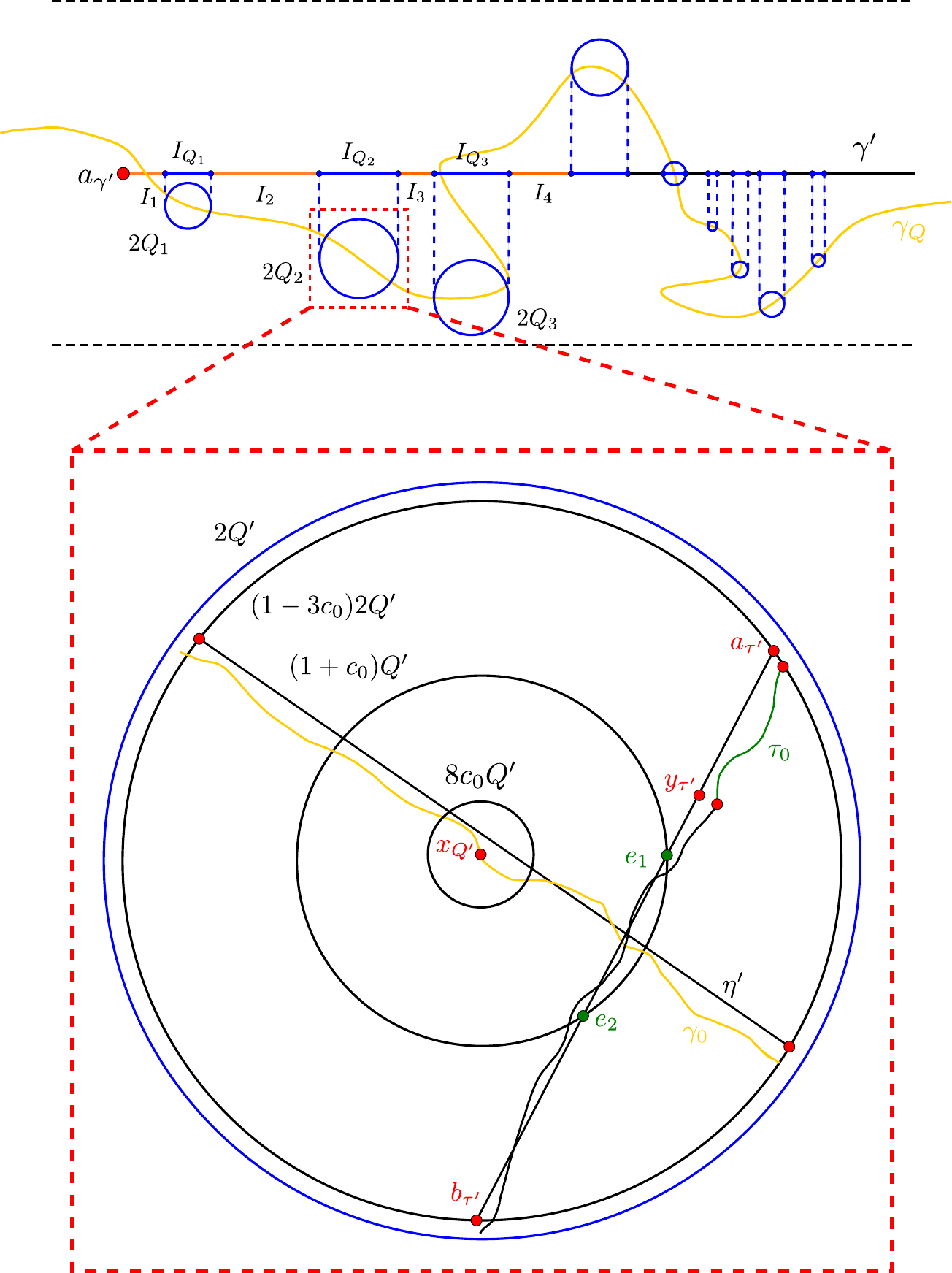}
    \caption{A picture of the proofs of Lemmas \ref{l:delta2.1-cores} and \ref{l:delta2.1-excess}. At the top is a picture of $\gamma_Q\cap U_Q$ in Lemma \ref{l:delta2.1-cores} on which lies members of the large family of disjoint dominant balls. In the image below, we have zoomed-in on one of these balls and have labeled pieces present in the proof of Lemma \ref{l:delta2.1-excess}.}
    \label{fig:delta2.1-martingale-global}
\end{figure}

\end{proof}
We can now combine this lemma with Lemma \ref{l:big-proj} on dominant projections to prove that the martingale construction can be applied to $\Delta_{2.1}$.
\begin{lem}[cf. \cite{Sc07} Lemma 3.28]\label{l:delta2.1-cores}
For any $Q\in\Delta_{2.1}$,
\begin{equation*}
    \frac{\diam(U_Q)}{\sum_{Q'\in C(Q)}\diam(U_{Q'}) + 2\epsilon_1^{-1}\ell(R_Q)} \leq \frac{1}{1+\frac{1}{50}}.
\end{equation*}
\end{lem}
\begin{proof}
First, observe that if $\ell(R_Q) > \epsilon_1\diam(U_Q)$, then
\begin{equation*}
    \frac{\diam(U_Q)}{s_Q} = \frac{\diam(U_Q)}{\sum_{Q'\in C(Q)}\diam(U_{Q'}) + 2\epsilon_1^{-1}\ell(R_Q)} < \frac{\diam(U_Q)}{2\diam(U_Q)} = \frac{1}{2} < 1.
\end{equation*}
Therefore, we can assume without loss of generality that $\ell(R_Q) \leq \epsilon_1\diam(U_Q)$. For $Q'\in C_D(Q)$, define $I_{Q'} :=\pi_{\gamma'}(2Q')$. Identifying $\gamma'$ with $\R$, we can apply a covering lemma for the real line (see \cite{Al91} Lemma 2.1, for example) to the collection $\{I_{Q'}\}_{Q'\in C_D(Q)}$ to get a collection $\scQ\subseteq C_D(Q)$ of balls with pairwise disjoint doubles so that
\begin{align}\label{e:covering-big-proj}
    \sH^1\left(\bigcup_{Q'\in\scQ}I_{Q'}\right) \geq \frac{1}{3}\sH^1\left( \bigcup_{Q'\in C_D(Q)}I_{Q'} \right) \geq \frac{1}{3}(1 - 50\epsilon_1)\diam(U_Q) \geq \frac{1}{4}\diam(U_Q)
\end{align}
where we used Lemma \ref{l:big-proj} in the penultimate inequality. We can then enumerate the components of $\gamma'\setminus\bigcup_{Q'\in\scQ}I_{Q'}$ as
\begin{equation*}
    \gamma'\setminus\bigcup_{Q'\in\scQ}I_{Q'} =: \bigcup_{j\in J_{\scQ}}I_j.
\end{equation*}
Define $\scQ_c \vcentcolon= \{Q'\in C(Q) : U_{Q'}\setminus \bigcup_{Q''\in\scQ}2Q''\not=\varnothing\}$. We have
\begin{equation*}
    \bigcup_{j\in J_\scQ}I_j \subseteq \pi_{\gamma'}\left(R_Q\setminus\bigcup_{Q'\in\scQ}2Q'\right)\cup\bigcup_{Q'\in\scQ_c}\pi_{\gamma'}(U_{Q'})
\end{equation*}
using Lemma \ref{l:gamma'} \ref{i:gamma'-proj}. Therefore, combining this fact with Lemma \ref{l:delta2.1-excess},
\begin{align*}
    \sum_{Q'\in C(Q)}&\diam(U_{Q'}) + 2\epsilon_1^{-1}\ell(R_Q) \\
    &\geq  \sum_{Q'\in\scQ}\sum_{\substack{Q''\in C(Q) \\ U_{Q''}\subseteq 2Q'}}\diam(U_{Q''}) + \ell(R_Q\cap 2Q') + \sum_{Q'\in\scQ_c}\diam(U_{Q'}) + \ell\left(R_Q\setminus \bigcup_{Q'\in\scQ}2Q'\right)\\
    &\geq \sum_{Q'\in\scQ}\left(1 + \frac{1}{10}\right)\diam(2Q') + \sum_{Q'\in\scQ_c}\sH^1(\pi_{\gamma'}(U_{Q'})) + \sH^1\left(\pi_{\gamma'}\left(R_Q\setminus \bigcup_{Q'\in\scQ}2Q'\right)\right)\\
    &\geq \sum_{Q'\in\scQ}\left(1 + \frac{1}{10}\right)\sH^1(I_{Q'}) + \sum_{j\in J_{\scQ}} \sH^1(I_j)\\
    &\geq \frac{1}{10}\sum_{Q'\in\scQ}\sH^1(I_{Q'}) + \sH^1(\gamma')\\
    &\stackrel{\eqref{e:covering-big-proj}}{\geq} \frac{1}{40}\diam(U_Q) + (1-30\epsilon_1)\diam(U_Q) > \left(1 + \frac{1}{50}\right)\diam(U_Q).
\qedhere\end{align*}
\end{proof}

\begin{prop}[cf. \cite{Sc07} Lemma 3.28]\label{p:delta2.1-sigma-bound}
\begin{equation*}
    \sum_{Q\in\Delta_{2.1}}\beta_{\Sigma}(Q)^2\diam(Q) \lesssim_A \sum_{Q\in\Delta_{2.1}}\diam U_Q \lesssim_J \ell(\Sigma)
\end{equation*}
\end{prop}
\begin{proof}
Order $\Delta_{2.1}$ via the forest $\mathcal{T}_{\Delta_{2.1}}^{c_0}$. Using Lemma \ref{l:delta2.1-cores}, we apply Lemma \ref{l:martingale} with $\scL = \Delta_{2.1},\ D = 2\epsilon_1^{-1},\ q = \frac{1}{1+\frac{1}{50}}$ to get the existence of a collection of positive real-valued functions $\{w_Q\}_{Q\in\Delta_{2.1}}$ satisfying
\begin{enumerate}[label=(\roman*)]
    \item $\int_Q w_Qd\ell = \diam(U_Q),$ \text{and}
    \item $\sum_{Q\in\Delta'}w_Q(x) \lesssim \epsilon_1^{-1}\chi_{U^{c_0}_{\Delta_{2.1}}}(x)\text{ for almost every $x\in\Sigma$}$.
\end{enumerate}
Using these properties, we can finish the proof of the lemma as follows:
\begin{align*}
    \sum_{Q\in\Delta_{2.1}}\beta(Q)^2\diam(Q) &\lesssim_A \sum_{Q\in\Delta_{2.1}}\diam(U_Q) = \sum_{Q\in\Delta_{2.1}}\int_Qw_Qd\ell = \int_\Gamma\sum_{Q\in\Delta_{2.1}}w_Qd\ell \\
    &\lesssim_{\epsilon_1} \int_{U^{c_0}_{\Delta_{2.1}}}d\ell = \sum_{T\in\mathcal{T}^{c_0}_{\Delta_{2.1}}}\ell(U_{Q(T)})\lesssim_J \sH^1(\Sigma).
\qedhere\end{align*}
\end{proof}
\end{subsubsection}

\end{subsection}
\begin{subsection}{Almost flat arcs for $\Gamma$}\label{subsec:almost-flat-arcs-gamma}\
The goal of this section is to finish the proof of Proposition \ref{p:almost-flat} by proving the second inequality in \eqref{e:almost-flat-bound}. In Section \ref{subsec:def-tools}, we give preliminary definitions and lemmas needed to refine the results of the previous section. In Section \ref{subsec:martingale-refinement} we use these tools to strengthen the previously given martingale arguments for the family $\Delta_1$ and the newly defined family $\Delta_{2.1.1}\subseteq\Delta_2$. Finally, in Section \ref{subsec:delta2.1.2} we analyze the leftover family $\Delta_{2.1.2}$ and finish the proof of Proposition \ref{p:almost-flat}, and hence the proof of Theorem \ref{t:thmA}.

\begin{subsubsection}{New Definitions and Tools}\label{subsec:def-tools}
Recall that $\Gamma\subseteq\ell_2$ is a Jordan arc with an injective arc length parameterization $\gamma:I\rightarrow\Gamma$ where we fix $I\vcentcolon=[0,\ell(\Gamma)]$. We assume without loss of generality that the chord line of $\Gamma$ is the $e_1$-axis. Let $\pi:\ell_2\rightarrow\R$ is the orthogonal projection onto the $e_1$-axis and let $\pi^\perp:\ell_2\rightarrow\ell_2$ be the orthogonal projection onto the orthogonal hyperplane to the $e_1$-axis. For every $i\in\N$ the function $\gamma_i(t)\vcentcolon= \langle \gamma(t), e_i\rangle$ is $1$-Lipschitz, hence differentiable almost everywhere. We let $\gamma_i'(t)$ denote the derivative and write 
\begin{equation*}
    \gamma(t) = \sum_{i=1}^\infty \gamma_i(t)e_i \text{ and } \gamma'(t)\vcentcolon= \sum_{i=1}^\infty \gamma_i'(t)e_i.
\end{equation*}
The fact that $\gamma$ is an arc length parameterization means that $|\gamma'(t)| = 1$ almost everywhere. In particular, $\gamma'(t)$ gives an almost everywhere well-defined notion of tangent vector to $\Gamma$ at $\gamma(t)$. For $x\in\Gamma$, we let $t(x)\in I$ be the unique number such that $\gamma(t(x)) = x$. 

We begin by defining a new measure $\mu \ll \ell$ which quantifies how much subsets of $\Gamma$ contribute to the value of $\ell(\Gamma) - \crd(\Gamma)$.
\begin{defn}[$\mu$ measure]
Let $\rho:I\rightarrow[0,2]$ be given by
\begin{equation*}
    \rho(t) \vcentcolon= \begin{cases} 1 - \gamma^\prime_1(t),\ &\gamma_1^\prime(t)\ \text{exists} \\
    1,  &\text{otherwise}.
    \end{cases}
\end{equation*}
Define the finite Borel measure $\mu$ supported on $\Gamma$ as
\begin{equation*}
    d\mu \vcentcolon= \gamma_*(\rho\ dt).
\end{equation*}
where $\mu(A) = \gamma_*(\rho\ dt)(A) \vcentcolon= \int_{\gamma^{-1}(A)}\rho(t)dt$ is the pushforward of $\rho(t)dt$ by $\gamma$.
\end{defn}
The definition of $\mu$ is motivated by the fundamental theorem of calculus in the following way:
\begin{lem}\label{l:intmu}
Let $a,b\in I$ with $a \leq b$. Then 
\begin{equation*}
    \mu(\gamma|_{[a,b]}) = \ell(\gamma|_{[a,b]}) - (\pi(\gamma(b)) - \pi(\gamma(a))).
\end{equation*}
In particular,
\begin{equation*}
    \mu(\Gamma) = \ell(\Gamma) - \crd(\Gamma).
\end{equation*}
\end{lem}
\begin{proof}
We compute
\begin{align*}
    \mu(\gamma|_{[a,b]}) &\vcentcolon= \mu(\Image(\gamma_{[a,b]})) = \int_{\gamma([a,b])}\gamma_*(\rho\ dt) = \int_{a}^{b}\rho(t)\ dt = \int_{a}^{b}1 - \gamma_1'(t)\ dt\\ 
    &= (b-a) - (\gamma_1(b)) - \gamma_1(a)) = \ell(\gamma|_{[a,b]}) - (\pi(\gamma(b)) - \pi(\gamma(a))).
\end{align*}
Setting $a = 0$ and $b = \ell(\Gamma)$ gives $\mu(\Gamma) = \ell(\Gamma) - \crd(\Gamma)$.
\end{proof}
\begin{remark}[Null sets and examples]
Fix $x,y\in\Gamma$ with $t(x) < t(y)$ and suppose $\xi$ is a subarc of $\Gamma$ such that $\Start(\xi) = x$ and $\End(\xi) = y$. If $\mu(\xi) = 0$, then
\begin{equation*}
    \ell(\xi) = \pi(y) - \pi(x) = y_1 - x_1.
\end{equation*}
This forces $y_1 > x_1$ and forces $\xi$ to be a parameterization of the line segment $[x,y] = [x,x+(y_1-x_1)e_1$ which is parallel to the chord line of $\Gamma$. Now, suppose $\tau(t)\vcentcolon= x + t\frac{y-x}{|y-x|}$ for $t(x) \leq t \leq t(x) + |y-x|$. That is, $\tau$ parameterizes $[x,y]\subseteq\Gamma$. In general, we have the formula
\begin{equation}\label{e:mu-of-segment}
    \mu(\tau) = \ell(\tau) - (\pi(y) - \pi(x)) = |y-x|-(y_1-x_1).
\end{equation}
If $y_1 < x_1$, then $\mu$ is larger than $\ell$ on $\tau$; this measure assigns ``bonus'' length to arcs which ``backtrack'' along the direction of the chord line of $\Gamma$. If $y_1 > x_1$, the right side of \eqref{e:mu-of-segment} bears resemblance to triangle inequality excess estimates. Indeed, let $x,y,z\in\Gamma$ and suppose there exists a subarc $\eta$ such that
\begin{equation*}
    \eta(t)\vcentcolon= 
    \begin{cases}
        x + t\frac{y-x}{|y-x|},\ & t(x) \leq t \leq t(x) + |y-x|\\
        y + t\frac{z-y}{|z-y|},\ & t(x) + |y-x| t \leq t(x) + |y-x| + |z-y|.
    \end{cases}
\end{equation*}
The arc $\eta$ injectively parameterizes the line segments $[x,y]$ and $[y,z]$. We compute
\begin{equation*}
    \mu(\eta) = \ell(\eta) - (\pi(x) - \pi(z)) = |x-y| + |y-z| - (z_1-x_1).
\end{equation*}
When $x_1 < y_1 < z_1$, this is something like a triangle inequality excess estimate (see Remark \ref{rem:pythagorean-theorem}) where instead of subtracting the length of the triangle base $[x,z]$, we subtract the length of the projection of $[x,z]$ along the chord line of $\Gamma$ .
\end{remark}

Our goal for proving Theorem \ref{t:thmA} is to bound the beta sums above by $\mu(\Gamma)$ rather than $\ell(\Gamma)$. Intuitively, this is plausible since $\mu$ assigns small measure only to those regions of $\Gamma$ which are nearly parallel to the chord and are directed via the parameterization $\gamma$ towards the terminal endpoint of $\Gamma$, i.e., have $\gamma^\prime_1 > 1 - \delta$ for $\delta > 0$ small. One would expect $\beta_\Gamma(Q)$ to be small on average for $Q$ centered in such a region. There is a problem with this definition of $\mu$, however. We would like to have a bound of the form
\begin{equation}\label{e:desire-mulb}
    \mu(U_Q) \gtrsim \ell(U_Q)
\end{equation}
for individual cores in some family because this would allow a translation of the preceding martingale arguments to this setting. However, such a result cannot hold, as $\mu(U_Q) = 0$ may hold even for $U_Q$ with $\beta_\Gamma(U_Q) \approx 1$ (see Figure \ref{fig:mu0beta1}). However, in order for a situation like Figure \ref{fig:mu0beta1} to occur, there must be some ``backtracking'' arc (given in the figure by the bottom-most horizontal piece of $\Gamma$ outside of $U_Q$). On this arc, $\gamma_1' < 0$ so that $d\mu \geq d\ell$. To recover inequalities like \eqref{e:desire-mulb}, we will construct a new, larger measure $\tilde{\mu}$ that fills in the $\mu$ measure gaps in Figure \ref{fig:mu0beta1} by ``borrowing'' mass from backtracking arcs. We begin by isolating these regions of change as maximal disjoint arcs where $\Gamma$ ``bends" back on itself along the $e_1$ axis in the sense that the projection map $\pi$ is non-injective. This is made more precise with the following definition.
\begin{figure}[h]
    \centering
    \includegraphics[scale=0.5]{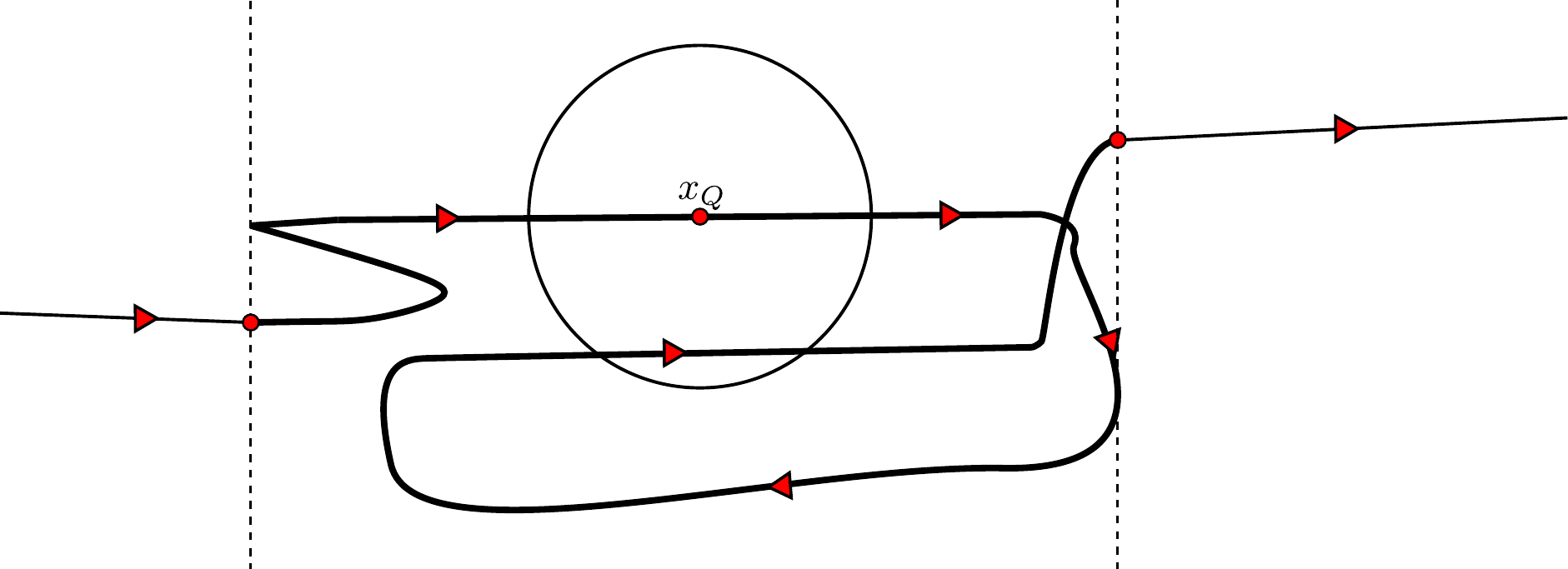}
    \caption{A core with $\beta_\Gamma(U_Q) \approx 1$ but $\mu(U_Q) = 0$. The red arrows indicate the direction of the parameterization such that $\rho \equiv 0$ on the two horizontal lines passing through $U_Q$. The thickened piece of $\Gamma$ in between the vertical dotted lines is a bend (assume that $\Gamma$'s chord line is horizontal).}
    \label{fig:mu0beta1}
\end{figure}

\begin{defn}[Multiplicity]
For $t\in I$, define
\begin{equation*}
    \prp(t) \vcentcolon= \gamma^{-1}(\pi^{-1}(\pi(\gamma(t))) \cap \Gamma).
\end{equation*}
This is the set of points in $I$ which map to points in $\Gamma$ that have the same first coordinate as $\gamma(t)$. Define $M:I\rightarrow\N\cup\{\infty\}$, the multiplicity function, by
\begin{equation*}
    M(t) \vcentcolon= \#\prp(t).
\end{equation*}
Additionally, we let $E \vcentcolon= \{t\in I : \pi(\gamma(t)) = \min(\pi(\Gamma)) \text{ or } \pi(\gamma(t)) = \max(\pi(\Gamma))\}$ and set
\begin{align*}
    S_\pi &\vcentcolon= \{t\in I : M(t) = 1\},\\
    M_\pi &\vcentcolon= \{t\in I : M(t) \geq 2\} \cup E. 
\end{align*}
$S_\pi$ is the set where $\prp(t) = \{t\}$ is a singleton while $M_\pi$ is the set where either $\prp(t)$ has multiple elements or $\gamma(t)$ is an extremal point of $\Gamma$ along the $e_1$-axis. If one of these latter points is also a member of $S_\pi$, then it is isolated by connectedness.  
\end{defn}
We will add mass to $\mu$ by raising the value of $\rho$ in a carefully chosen neighborhood of $M_\pi$. In order to define this neighborhood, we first provide a decomposition of $M_\pi$ into the maximal arcs promised above. The following structure lemma for $M_\pi$ will aid us:
\begin{lem}\label{l:bendconnected}
Suppose $a,b\in I$ with $\pi(\gamma(a)) = \pi(\gamma(b))$ and $a \leq b$. Then $[a,b]\subseteq M_\pi$.
\end{lem}
\begin{proof}
Let $r\in(a,b)$. If $\pi(\gamma(r)) = \pi(\gamma(a))$, then $r\in M_\pi$. Otherwise, $\pi(\gamma(r))\not=\pi(\gamma(a))$ and
\begin{equation*}
    \inf_{t\in[a,b]}\pi(\gamma(t)) \leq \pi(\gamma(r)) \leq \sup_{t\in[a,b]}\pi(\gamma(t)).
\end{equation*}
Since $\pi$ is continuous, there exist points $s,u\in[a,b]$ on which $\pi\circ\gamma$ achieves the infimum and supremum above respectively. Suppose first that $\pi(\gamma(r)) = \pi(\gamma(s))$. If there exists $s'\in I,\ s'\not=r$ such that $\pi(\gamma(s')) = \pi(\gamma(r))$, then $r\in M_\pi$ by definition. Otherwise, $\pi(x) > \pi(\gamma(r))$ for all $x\in\Gamma\setminus \gamma(r)$ so that $\gamma(r) = \min(\pi(\Gamma))\in E\subseteq M_\pi$. Therefore, it suffices to consider the case when $\pi(\gamma(r)) > \pi(\gamma(s))$. By a similar argument, we can also assume $\pi(\gamma(r)) < \pi(\gamma(u))$. Hence, the function  $f:[a,b]\rightarrow\R$ given by $f(t) = \pi(\gamma(t)) - \pi(\gamma(a))$ is continuous and satisfies
\begin{enumerate}[label=(\roman*)]
    \item $f(r)\not=0$,
    \item $f(s) < f(r) < f(u)$, \text{and}
    \item $f(a) = f(b) = 0$.
\end{enumerate}
We claim that the intermediate value theorem implies the existence of a point $r'\in[a,b],\ r\not=r'$ with $f(r') = f(r)$ so that $\gamma(r)\in M_\pi$. Indeed, suppose without loss of generality that $f(r) > 0$. If $u < r$, then $f(a) < f(r) < f(u)$ so that there exists such $r'\in[a,u]$. Otherwise, $u > r$ and $f(b) < f(r) < f(u)$ so that there exists such $r'\in[u,b]$.
\end{proof}
\begin{defn}[Bends]\label{def:bends}
It follows from Lemma \ref{l:bendconnected} that for any $t\in M_{\pi}$, the family
\begin{equation*}
    \sC_{t} \vcentcolon= \{[a,b]\subseteq I:t\in [a,b]\subseteq M_\pi\}
\end{equation*}
contains a non-degenerate interval so that the union $I_{t}\vcentcolon= \bigcup \sC_{t}$ is also a non-degenerate interval. The union $\bigcup_{t\in I}I_{t} = M_\pi$ has countably many disjoint connected components, each of which is a non-degenerate subinterval of $I$ (which is possibly open). Thus, the closure $\overline{M}_\pi$ has connected components that are closed, non-degenerate intervals which we enumerate as $\overline{M}_\pi = \bigcup_{k\in K} [s_k,u_k]$. We define 
\begin{equation*}
    \Phi \vcentcolon= \{\gamma|_{[s_k,u_k]} : k\in K\}.
\end{equation*}
We refer the elements of $\Phi$ as \textit{bends}.
\end{defn}

An important fact is that these regions contribute a proportionally large amount of measure to $\mu$.
\begin{lem}\label{l:mulb}
Let $\phi\in\Phi$. Then 
\begin{equation*}
    \mu(\phi) \geq \frac{1}{2}\ell(\phi).
\end{equation*}
\end{lem}
\begin{proof}
Let $\phi\in\Phi$. Lemma \ref{l:intmu} implies
\begin{equation*}
    \mu(\phi) = \ell(\phi) - \left[\pi(\End(\phi)) - \pi(\Start(\phi))\right].
\end{equation*}
But $\pi$ is at least two-to-one almost everywhere on $\Image(\phi)\subseteq\overline{M}_\pi$, implying $\ell(\phi) \geq 2|\pi(\End(\phi)) - \pi(\Start(\phi))|$.
\end{proof}
Lemma \ref{l:mulb} says that the bends are arcs on which $\mu$ measure is globally comparable to length measure. This allows us to promote $\mu$ to a bigger measure $\tilde{\mu}$ which is pointwise comparable to length inside bends at the cost of increasing the total measure by a bounded factor independent of $\Gamma$. In fact, at the cost of further increasing $\mu$'s mass by a bounded factor, we can take our proposed larger measure $\tilde{\mu}$ to be comparable to length on regions of $\Gamma$ which extend a distance comparable to $\ell(\phi)$ out from $\phi$ in the $e_1$ direction.
\begin{defn}[$\tilde{\mu}$ measure]\label{def:mu-tilde}
For any $\phi\in\Phi$, define
\begin{align*}
    N_\phi &\vcentcolon= \{t\in I : \dist(\pi(\gamma(t)),\pi(\Image(\phi))) \leq 100\ell(\phi)\},
\end{align*}
and set
\begin{equation*}
    N(\Phi) \vcentcolon= \bigcup_{\phi\in\Phi}N_\phi.
\end{equation*}
We define a new weight $\tilde{\rho}:I\rightarrow[0,2]$ as
\begin{equation*}
    \tilde{\rho}(t)\vcentcolon= 
    \begin{cases} 
        2,\ & t\in\Domain(\phi)\text{ for some $\phi\in\Phi$}\\
        1,\ & t\in N(\Phi)\setminus \bigcup_{\phi\in\Phi}\Domain(\phi)\\
        \rho(t), & t\in I \setminus N(\Phi).
    \end{cases}
\end{equation*}
We use the value $2$ in the first case so that $\tilde{\rho}(t) \geq \rho(t)$ for all $t\in I$. Put
\begin{equation*}
    d\tilde{\mu}\vcentcolon=\gamma_*(\tilde{\rho}\ dt).
\end{equation*}
In words, $\tilde{\mu}$ is equal to twice length measure on bends, equal to length measure just outside of bends, and equal to $\mu$ measure far away from bends. 
\end{defn}
Looking back at Figure \ref{fig:mu0beta1}, we can see that $\tilde{\mu}(U_Q) = 2\ell(U_Q)$ because $\Gamma\cap U_Q$ is contained in a single bend $\phi$. We will use $\tilde{\mu}$ as our primary accounting tool for bounding the beta-squared sum from above. We begin by verifying that the total mass of $\tilde{\mu}$ is controlled by the total mass of $\mu$.
\begin{lem}\label{l:mutilde}
\begin{equation*}
    \tilde{\mu}(\Gamma) \lesssim \mu(\Gamma).
\end{equation*}
\end{lem}
\begin{proof}
Fix $\phi\in\Phi$ and consider the region 
\begin{equation*}
    U_\phi \vcentcolon= \gamma\left(N_\phi\setminus\bigcup_{\substack{\phi'\in\Phi \\ \phi'\not=\phi}}\Image(\phi')\right)
\end{equation*}
Observe that we can bound the mass added in $U_\phi$ as follows:
\begin{align*}
    \tilde{\mu}(U_\phi) - \mu(U_\phi) &= \tilde{\mu}(U_\phi\setminus\Image(\phi)) + \tilde{\mu}(\phi) - \left[\mu(U_\phi\setminus \Image(\phi)) + \mu(\phi)\right] \\
    &= \ell(U_\phi\setminus\Image(\phi)) - \mu(U_\phi\setminus \Image(\phi)) + 2\ell(\phi) - \mu(\phi) \\
    &\leq 2\ell(\phi) + \int_{U_\phi\setminus\Image(\phi)}\gamma_1'(t)\ dt \\
    &\leq \diam(\pi(U_\phi\setminus\Image(\phi))) + 2\ell(\phi) \leq 200\ell(\phi) + 2\ell(\phi) \leq 404\mu(\phi)
\end{align*}
where the final inequality follows from Lemma \ref{l:mulb}. The fact that $\tilde{\rho}(t) \geq \rho(t)$ for all $t\in I$ implies $\tilde{\mu} - \mu$ is a positive measure so that, because $\tilde{\rho}(t) = \rho(t)$ for all $t\in I \setminus N(\Phi)$ and $N(\Phi) = \bigcup_{\phi\in\Phi}\gamma^{-1}(U_\phi)$,
\begin{align*}
    \tilde{\mu}(\Gamma) - \mu(\Gamma) &\leq \tilde{\mu}\left(\bigcup_{\phi\in\Phi}U_\phi\right)  - \mu\left(\bigcup_{\phi\in\Phi}U_\phi\right) = (\tilde{\mu}-\mu)\left(\bigcup_{\phi\in\Phi}U_\phi\right)\\
    &\leq \sum_{\phi\in\Phi}(\tilde{\mu}-\mu)(U_\phi) = \sum_{\phi\in\Phi}\tilde{\mu}(U_\phi)-\mu(U_\phi) \leq 404\sum_{\phi\in\Phi}\mu(\phi) \leq 404\mu(\Gamma).
\qedhere\end{align*}
\end{proof}
The final lemma we will prove in this section gives sufficient conditions for an inequality like $\tilde{\mu}(B)\gtrsim\ell(B)$ to hold for any Borel set $B$. The fact that $\tilde{\mu}$ is comparable to $\ell$ on $\gamma(M_\pi)$ means that in order for $\tilde{\mu}(B) \ll \ell(B)$ to hold, most of $\Gamma\cap B$ must be contained in $\gamma(S_\pi)$. Even then, it must be true that $\tilde{\rho} \ll 1$ on most of $\Gamma\cap B$ so that $\gamma_1'\approx 1$ inside $B$, constraining the total amount of length that $B$ is allowed to contain. Lemma \ref{l:mulength} proves a sort of contrapositive of this observation, showing that a lower bound on $\ell(B)$ translates into a lower bound on $\tilde{\mu}(B)$ in terms of $\ell(B)$. First, we will need a version of the area formula:
\begin{lem}{(Area formula)}\label{l:area}
Let $f:\R\rightarrow\R$ be a Lipschitz map and let $g:\R\rightarrow\R$ be an integrable function. Then the map
\begin{equation*}
    z\mapsto\sum_{x\in f^{-1}(\{z\})}g(x),
\end{equation*}
is measurable, and
\begin{equation*}
    \int_\R g(y)|f^\prime(y)| dy = \int_\R \sum_{x\in f^{-1}(\{z\})}g(x) dz.
\end{equation*}
\end{lem}
For proof of this result, see \cite{Fe69} Theorem 3.2.5.
\begin{lem}\label{l:mulength}
Fix $\delta < \frac{1}{10}$ and let $B\subseteq \ell_2$ be Borel. If either
\begin{enumerate}[label=(\roman*)]
    \item $\ell(B) \geq (1+\delta)\diam(B),$ \text{or} \label{i:ell-diam}
    \item $\tilde{\mu}(B) \geq\delta\diam(B)$, \label{i:mu-diam}
\end{enumerate}
then,
\begin{equation*}
    \tilde{\mu}(B) \geq \frac{\delta^3}{2}\ell(B).
\end{equation*}
\end{lem}
\begin{proof}
We first prove that \ref{i:ell-diam} implies the conclusion. Consider the set $E_{\delta^2} = \{t\in I:\tilde{\rho}(t) < \delta^2\}\cap \gamma^{-1}(B)$. Observe that $\tilde{\rho}(t) < \delta^2$ implies $\gamma_1^\prime(t) > 1-\delta^2$ so that $E_{\delta^2}\subseteq S_\pi$. The former inequality directly implies $t\in S_\pi$ so that $E_{\delta^2} \subseteq S_\pi$. Applying the area formula (Lemma \ref{l:area}) with the Lipschitz function $\gamma_1:I\rightarrow\R$ and the integrable function $\frac{1}{\gamma_1^\prime}\chi_{E_{\delta^2}}$ gives
\begin{equation*}
    \int_{E_{\delta^2}}\frac{1}{\gamma_1^\prime}|\gamma_1^\prime|\ dt = \int_{\R}\sum_{s\in \gamma_1^{-1}(u)}\frac{1}{\gamma_1^\prime(s)}\chi_{E_{\delta^2}}(s)\ du = \int_{\R}\sum_{s\in\gamma^{-1}(\pi^{-1}(u))\cap E_{\delta^2}}\frac{1}{\gamma_1'(s)}\ du.
\end{equation*}
Because $E_{\delta^2}\subseteq S_\pi$ and $\gamma$ is injective, the set $\gamma^{-1}(\pi^{-1}(u))\cap E_{\delta^2}$ contains at most one element, and is nonempty only if $u\in \pi(\gamma(E_{\delta^2}))\subseteq B$. Therefore, we get
\begin{align*}
    \sum_{s\in\gamma^{-1}(\pi^{-1}(u))\cap E_{\delta^2}}\frac{1}{\gamma_1'(s)} = \sum_{s\in\gamma^{-1}(\pi^{-1}(u))\cap E_{\delta^2}}\frac{1}{\gamma_1'(s)}\chi_{\gamma_1(E_{\delta^2})}(u) \leq \frac{1}{1-\delta^2}\chi_{\gamma_1(E_{\delta^2})}(u).
\end{align*}
Using these statements, the area formula simplifies to
\begin{align*}
    \ell(E_{\delta^2}) = \int_{\R}\sum_{s\in\gamma^{-1}(\pi^{-1}(u))\cap E_{\delta^2}}\frac{1}{\gamma_1'(s)}\ du \leq \int_{\gamma_1(E_{\delta^2})}\frac{1}{1-\delta^2}\ du \leq \frac{\diam(B)}{1-\delta^2}.
\end{align*}
Now, define $C_{\delta^2}\vcentcolon= B\cap\Gamma\setminus \gamma(E_{\delta^2})$. We have
\begin{align*}
    \tilde{\mu}(B) \geq \tilde{\mu}(C_{\delta^2}) \geq \delta^2\ell(C_{\delta^2}) = \delta^2(\ell(B) - \ell(E_{\delta^2})).
\end{align*}
Adding in the lemma's hypothesis, we have both
\begin{align*}
    \ell(B) &\geq (1+\delta)\diam(B), \text{ and}\\
    \ell(E_{\delta^2}) &\leq \frac{\diam(B)}{1-\delta^2}.
\end{align*}
This means
\begin{align*}
    \frac{\ell(E_{\delta^2})}{\ell(B)} &\leq \frac{\diam(B)}{1-\delta^2}\cdot\frac{1}{(1+\delta)\diam(B)} = 1 + \left(\frac{1}{(1+\delta)(1-\delta^2)} - 1\right) \\
    &= 1 - \frac{\delta - \delta^2 - \delta^3}{(1+\delta)(1-\delta^2)} \leq 1 - \frac{\delta - \frac{\delta}{10} - \frac{\delta}{100}}{1+\frac{1}{10}} \leq 1-\frac{\delta}{2}.
\end{align*}
using the fact that $\delta < \frac{1}{10}$. Rearranging this inequality gives $\ell(B) - \ell(E_{\delta^2}) \geq \frac{\delta}{2}\ell(B)$. Therefore
\begin{equation*}
    \tilde{\mu}(B) \geq \delta^2(\ell(B) - \ell(E_{\delta^2})) \geq \frac{\delta^3}{2}\ell(B).
\end{equation*}
This concludes the proof that \ref{i:ell-diam} implies the conclusion. We now show that \ref{i:mu-diam} implies the conclusion. From \ref{i:ell-diam}, it suffices to assume $\ell(B) < (1+\delta)\diam(B)$. Then,
\begin{equation*}
    \tilde{\mu}(B) \geq \delta\diam(B) \geq \delta\frac{1+\delta}{2}\diam(B) \geq \frac{\delta}{2}\ell(B).
\qedhere\end{equation*}
\end{proof}

\end{subsubsection}

\begin{subsubsection}{\texorpdfstring{Martingale refinement: Bounds on the $\Delta_1$ and $\Delta_{2.1.1}$ sums for $\Gamma$}{}}\label{subsec:martingale-refinement}
In this section, we provide the refinements of Proposition \ref{p:delta1-sigma-bound} and (part of) Proposition \ref{p:delta1-sigma-bound} for a rectifiable Jordan arc $\Gamma$. 
\begin{lem}\label{l:delta1-mutilde}
For any $Q\in\Delta^\prime(M,K)$
\begin{equation*}
    \tilde{\mu}(U^{xx}_Q) \gtrsim \ell(U^{xx}_Q).
\end{equation*}
\end{lem}
\begin{proof}
In the proof of Lemma \ref{l:delta1-cores}, we gave the existence of an arc $\xi_0\subseteq U_Q^{xx}$ such that
\begin{equation*}
    \ell(U_Q^{xx}) \geq \ell(\gamma_Q\cap U_Q^{xx}) + \ell(\xi_0) \geq \left(1 + \frac{1}{10}\right)\diam(U_Q^{xx})
\end{equation*}
as in \eqref{e:delta1-diam-lb}. Applying Lemma \ref{l:mulength} gives the result.
\end{proof}

\begin{prop}\label{p:delta1-gamma}
\begin{equation*}
    \sum_{Q\in\Delta_1}\beta_{\Gamma}(Q)^2\diam(Q) \lesssim_A \sum_{Q\in\Delta_1}\diam(U_Q) \lesssim_{A,J} \tilde{\mu}(\Gamma).
\end{equation*}
\end{prop}
\begin{proof}
Fix $\Delta' = \Delta'(M,K)$ and follow the proof of Proposition \ref{p:delta1-sigma-bound} to get
\begin{equation*}
    \sum_{Q\in\Delta'}\beta_\Gamma(Q)\diam(Q) \lesssim_A 2^{-M}\sum_{T\in\mathcal{T}_{\Delta'}^{16c_0}}\ell(U_{Q(T)}^{xx}) \lesssim 2^{-M}\tilde{\mu}(U_{Q(T)}^{xx}) \leq 2^{-M}\tilde{\mu}(\Gamma).
\end{equation*}
The result follows by summing over $M$ and $K$.
\end{proof}

We wish to argue for similarly for $\Delta_{2.1}$, but an inequality like that of Lemma \ref{l:delta1-mutilde} does not hold for $\Delta_{2.1}$ balls. We proceed by splitting $\Delta_{2.1}$ into a subfamily where $\tilde{\mu}(U_Q)\gtrsim\ell(U_Q)$ on which we can run the martingale argument and a leftover subfamily on which $\tilde{\mu}(U_Q)\ll\ell(U_Q)$. We define
\begin{align*}
    \Delta_{2.1.1} &= \{Q\in\Delta_{2.1}: \tilde{\mu}(U_{Q}) \geq \epsilon_3^2\ell(U_{Q})\},\\
    \Delta_{2.1.2} &= \{Q\in\Delta_{2.1}: \tilde{\mu}(U_Q) < \epsilon_3^2 \ell(U_Q)\}\\
    &=\Delta_{2.1}\setminus \Delta_{2.1.1}.
\end{align*}
The collection $\Delta_{2.1.1}$ can be handled with the addition of one inequality to the proof of Proposition \ref{p:delta2.1-sigma-bound}.
\begin{prop}
\begin{equation*}
    \sum_{Q\in\Delta_{2.1.1}}\beta_\Gamma(Q)^2\diam(Q) \lesssim_A \sum_{Q\in\Delta_{2.1.1}}\diam(U_Q) \lesssim_{J,\epsilon_1,\epsilon_3} \mu(\Gamma).
\end{equation*}
\end{prop}
\begin{proof}
We apply Lemma \ref{l:martingale} with  $D = 2\epsilon_1^{-1}, q = \frac{1}{1+\frac{1}{50}}$ to the family $\scL = \Delta_{2.1.1}$ with ordered by the forest structure $\mathcal{T}^{c_0}_{\Delta_{2.1.1}}$ and use the produced collection $\{w_Q\}_{Q\in\Delta_{2.1.1}}$ to calculate, as in the proof of Proposition \ref{p:delta2.1-sigma-bound},
\begin{align*}
     \sum_{Q\in\Delta_{2.1.1}}\diam(U_Q)
    &\lesssim_{\epsilon_1}\sum_{T\in\mathcal{T}_{\Delta_{2.1.1}}^{c_0}}\ell(U_{Q(T)}) \lesssim_{\epsilon_3}\sum_{T\in\mathcal{T}_{\Delta_{2.1.1}}^{c_0}}\tilde{\mu}(U_{Q(T)}) \lesssim_J \tilde{\mu}(\Gamma).
\qedhere\end{align*}
\end{proof}
\end{subsubsection}

\begin{subsubsection}{Bound on the \texorpdfstring{$\Delta_{2.1.2}$}{} sum}\label{subsec:delta2.1.2}
We now handle the family $\Delta_{2.1.2}$, beginning with a general summary of the argument. We show that $Q\in\Delta_{2.1.2}$ implies $U_Q\cap\Gamma$ essentially consists of a small perturbation (in length) of a line segment through the center of $Q$ which is parallel to the chord line of $\Gamma$ (see Lemma \ref{l:straightcenter}). Because of the definition of the bends and $\tilde{\mu}$, the nearly-segment pieces inside disjoint cores in this family project to line segments on the chord line of $\Gamma$ which have controlled overlap (see Lemma \ref{l:projection-packing}). By decomposing $\Delta_{2.1.2}$ into a sequence of ``levels'', each of which consists of a disjoint subfamily of $\Delta_{2.1}$, we can exploit this packing lemma by controlling the number of balls which have overlapping $\tau_Q$ arcs (see Remark \ref{rem:tau-q}), controlling the core diameter sum on each level in terms of a disjoint collection of subarcs of $\tau_Q$'s (see Lemma \ref{l:maximal-tau-bound} and \ref{l:zeta-disjoint}). This all works for the subcollection of cores which lie on the ``inner'' region of parent cores. The proof is completed by showing that the ``outer'' family of cores is controlled by the inner family (see Lemma \ref{l:deltaO-bound}).

Let us begin the proof. Define
\begin{equation*}
    \eta_Q \vcentcolon= \{x_Q + t e_1 : t\in\R\}.
\end{equation*}
This is the line parallel to the chord of $\Gamma$ which passes through the center of $Q$. The next lemma states that any $Q\in \Delta_{2.1.2}$ has $\gamma_Q$ close to $\eta_Q$ with constant dependent on $\epsilon_3$.
\begin{lem}\label{l:straightcenter}
Let $Q\in\Delta_{2.1.2}$ and $\xi\subseteq\Gamma$ such that $\xi\cap\gamma_Q = \varnothing$. Then,
\begin{enumerate}[label=(\roman*)]
    \item $\gamma_Q\subseteq B(\eta_Q,100\epsilon_3\diam(Q))$, \text{and}\label{i:gamma-sub}
    \item $\xi \cap \pi^{-1}((1-10\epsilon_1)c_0Q) = \varnothing.$\label{i:xi-empty}
\end{enumerate}
\end{lem}
\begin{proof}
We begin by proving \ref{i:gamma-sub}. We will assume that $\gamma_Q\not\subseteq B(\eta_Q, 100\epsilon_3\diam(Q))$ and show that $\tilde{\mu}(U_Q) > \epsilon_3^2\ell(U_Q)$. Let $\gamma'' \vcentcolon= \Edge(\gamma_Q),\ \gamma' \vcentcolon= \gamma''\cap c_0Q =\vcentcolon [x,y]$, and let $\varphi$ be a connected component of $\gamma_Q\cap c_0Q$ of largest diameter. We will show that $\gamma'$ makes angle of order $\epsilon_3$ with $\eta_Q$, derive a lower bound for the ``excess'' length of $\gamma'$, and then use that to bound $\frac{\tilde{\mu}(\varphi)}{\ell(\varphi)}$ from below. First, observe that $\tb{\gamma_Q}\diam(2Q) \leq 2\epsilon_2\diam(Q)$ implies, by Lemma \ref{l:almost-flat-crossing},
\begin{equation}\label{e:gamma''-haus-dist}
    \gamma_Q \subseteq B(\gamma'', 2\epsilon_2\diam(Q)) \text{ and } \gamma''\subseteq B(\gamma_Q, 2\epsilon_2\diam(Q)).
\end{equation}
Let $z\in\gamma_Q$ be such that $\dist(z,\eta_Q) \geq 100\epsilon_3\diam(Q)$. Then, there exists $z''\in\gamma''$ such that $|z-z''| \leq 2\epsilon_2\diam(Q) \leq \epsilon_3\diam(Q)$ so that $\dist(z'', \eta_Q) \geq \dist(z,\eta_Q) - |z-z''| \geq 99\epsilon_3\diam(Q)$.
We can define the angle $\theta\vcentcolon= \angle(\gamma',\eta_Q) = \angle(\gamma'',\eta_Q)$ by translating the segment $\gamma''$ so that one of its endpoints lies in $\eta_Q$ and measuring the angle in the (at most $2$-dimensional) plane containing $\eta_Q$ and this translated segment. The previous estimates then imply $\tan(\theta) \geq \frac{99\epsilon_3\diam(Q)}{\diam(2Q)} \geq 45\epsilon_3$. Using the Pythagorean theorem, we get $|x-y|^2 = |\pi(x) - \pi(y)|^2 + |\pi^{\perp}(x) - \pi^{\perp}(y)|^2$, and the lower bound on $\tan(\theta)$ implies $|\pi^\perp(x) - \pi^\perp(y)| \geq 45\epsilon_3|\pi(x)-\pi(y)|$. Using the difference of squares formula with the Pythagorean theorem estimate, we compute 
\begin{align}\nonumber
    \frac{|x-y| - |\pi(x) - \pi(y)|}{|x-y|} &= \frac{|\pi^{\perp}(x) - \pi^{\perp}(y)|^2}{|x-y|(|x-y| + |\pi(x)-\pi(y)|)} \geq 45^2\epsilon_3^2\frac{|\pi(x)-\pi(y)|^2}{|x-y|(|x-y| + |\pi(x)-\pi(y)|)}\\ 
    &\geq 45^2\epsilon_3^2\frac{\frac{1}{4}\diam(U_Q)^2}{\diam(U_Q)\cdot2\diam(U_Q)} \geq 100\epsilon_3^2 \label{e:seg-mu}
\end{align}
Now, let $x',y'$ be the endpoints of $\varphi$ on $\partial(c_0Q)$. By \eqref{e:gamma''-haus-dist} and the fact that $x_Q\in\gamma_Q$, we have
\begin{align*}
    |x-x'| \leq 8\epsilon_2\diam(Q) \text{ and } |y-y'| \leq 8\epsilon_2\diam(Q).
\end{align*}
We estimate
\begin{align}\nonumber
    \frac{\tilde{\mu}(\varphi)}{\ell(\varphi)} &\geq \frac{\mu(\varphi)}{\ell(\varphi)} \geq 1 - \frac{|\pi(x') - \pi(y')|}{|x'-y'|} = \frac{|x'-y'| - |\pi(x') - \pi(y')|}{|x'-y'|}\\\nonumber
    &\geq \frac{|x-y| - |\pi(x) - \pi(y)| - 32\epsilon_2\diam(Q)}{|x-y| + 16\epsilon_2\diam(Q)}\\
    &\geq \frac{1}{2}\frac{|x-y| - |\pi(x) - \pi(y)|}{|x-y|} - \frac{32\epsilon_2\diam(Q)}{|x-y| + 16\epsilon_2\diam(Q)} \geq 50\epsilon_3^2 - 64c_0^{-1}\epsilon_2 \geq 40\epsilon_3^2\label{e:phi-bound}
\end{align}
where we used \eqref{e:seg-mu} and $|x-y| \geq \frac{1}{2}c_0\diam(Q)$ in the penultimate inequality. We would like to show that $\tilde{\mu}(\gamma_Q\cap U_Q) \geq \epsilon_3^2\ell(\gamma_Q\cap U_Q)$. Given the preceding inequality, the only possible obstruction is the existence of components of $\gamma_Q\cap U_Q$ with long length and small $\tilde{\mu}$ measure. It suffices to consider the case where $\tilde{\mu}(\gamma_Q\cap U_Q\setminus\varphi) \leq 40\epsilon_3^2\ell(\gamma_Q\cap U_Q\setminus\varphi)$. Unpacking this inequality, we see
\begin{align*}
    40\epsilon_3^2\ell(\gamma_Q\cap U_Q\setminus\varphi) \geq \tilde{\mu}(\gamma_Q\cap U_Q\setminus \varphi) \geq \mu(\gamma_Q\cap U_Q\setminus\varphi) \geq \ell(\gamma_Q\cap U_Q\setminus\varphi) - \diam(\pi(\gamma_Q\cap U_Q\setminus\varphi)).
\end{align*}
Rearranging gives
\begin{equation}{\label{e:gamma-minus-phi}}
    \ell(\gamma_Q\cap U_Q\setminus\varphi) \leq \frac{\diam(\pi(\gamma_Q\cap U_Q\setminus\varphi))}{1-\epsilon_3^2} \leq \frac{\diam(U_Q)}{1-40\epsilon_3^2} \leq 2\ell(\varphi)
\end{equation}
where the final inequality follows since $x_Q\in\gamma_Q\in S(Q)$. Using \eqref{e:phi-bound} and \eqref{e:gamma-minus-phi},
\begin{align*}
    \tilde{\mu}(\gamma_Q\cap U_Q) &\geq \tilde{\mu}(\varphi) \geq 40\epsilon_3^2\ell(\varphi)
    \geq \frac{1}{2}(40\epsilon_3^2\ell(\varphi) + 10\epsilon_3^2\ell(\gamma_Q\cap U_Q\setminus \varphi)) \geq 5\epsilon_3^2\ell(\gamma_Q\cap U_Q).
\end{align*}
With this intermediate inequality, we can now prove the lemma. Arguing as in the proof of $\tilde{\mu}(\gamma_Q\cap U_Q) \geq 5\epsilon_3^2\ell(\gamma_Q\cap U_Q)$ above, it suffices to assume that $\tilde{\mu}(U_Q\setminus\gamma_Q) \leq 5\epsilon_3^2 \ell(U_Q\setminus\gamma_Q)$. We get
\begin{equation*}
    \ell(U_Q\setminus\gamma_Q) \leq \frac{\diam(\pi(U_Q\setminus\gamma_Q))}{1-5\epsilon_3^2} \leq 2\ell(\gamma_Q\cap U_Q).
\end{equation*}
Multiplying this inequality on both sides by $\epsilon_3^2$, we use this to estimate
\begin{align*}
    \tilde{\mu}(U_Q) &\geq \tilde{\mu}(U_Q\cap\gamma_Q) \geq 5\epsilon_3^2\ell(U_Q\cap \gamma_Q)
    \geq 3\epsilon_3^2\ell(U_Q\cap\gamma_Q) + \epsilon_3^2\ell(U_Q\setminus\gamma_Q) > \epsilon_3^2\ell(U_Q).
\end{align*}
This concludes the proof of \ref{i:gamma-sub}. We now prove \ref{i:xi-empty}. Suppose $Q\in\Delta_{2.1.2}$ is such that $\xi\cap\pi^{-1}(1-10\epsilon_1)c_0Q\not=\varnothing$ and let $x\in\xi\cap\pi^{-1}((1-10\epsilon_1)c_0Q)$. We will show that $\tilde{\mu}(U_Q) > \epsilon_3^2\ell(U_Q)$. By Lemma \ref{l:straightcenter}, we have either $\gamma_Q\cap U_Q \cap \{z : \pi(z) \geq \pi(x)\} \subseteq \gamma(\overline{M}_\pi)$ or $\gamma_Q\cap U_Q \cap \{z : \pi(z) \leq \pi(x)\} \subseteq \gamma(\overline{M}_\pi)$ because the extension of $\gamma_Q$ to an arc containing $\xi$ must cross from the boundary of $2Q$ to $x$ with an arc disjoint from $\gamma_Q$. Therefore, we conclude that $\gamma_Q\cap U_Q$ contains an arc $\zeta\subseteq\gamma_Q\cap c_0Q\cap(1-10\epsilon_1)c_0Q$ with $\tilde{\mu}(\zeta) = 2\ell(\zeta) \geq 2(10\epsilon_1c_0)\rad(Q) \geq 10\epsilon_1\diam(U_Q)$. This means $\tilde{\mu}(U_Q) \geq 2\epsilon_1\diam(U_Q)$ so that, by Lemma \ref{l:mulength},
\begin{equation*}
    \tilde{\mu}(U_Q) \geq \frac{(2\epsilon_1)^3}{2}\ell(U_Q) \geq \epsilon_1^3\ell(U_Q) > \epsilon_3^2\ell(U_Q).
\qedhere\end{equation*}
\end{proof}
This lemma places strong restrictions of the geometry of $\Gamma\cap 2Q$. The fact that $\gamma_Q$ is restricted to be nearly parallel to $\eta_Q$ on the scale of $\diam(Q)$ allows us to derive packing estimates for disjoint families of balls in $\Delta_{2.1.2}$ along the direction of the chord of $\Gamma$ as in the following lemma.
\begin{lem}\label{l:projection-packing}
For any $Q_1,Q_2\in\Delta_{2.1.2}$ such that $U_{Q_1}\cap U_{Q_2} = \varnothing$,
\begin{equation*}
    \pi\left(\frac{1}{4}c_0Q_1\right) \cap \pi\left(\frac{1}{4}c_0Q_2\right) = \varnothing.
\end{equation*}
\end{lem}

\begin{proof}
Suppose by way of contradiction that $\pi\left(\frac{1}{4}c_0Q_1\right) \cap \pi\left(\frac{1}{4}c_0Q_2\right) \not= \varnothing$ and assume $\diam(Q_2)\leq\diam(Q_1)$. The fact that $c_0Q_1\cap c_0Q_2 = \varnothing$ and $\pi\left(\frac{1}{4}c_0Q_1\right) \cap \pi\left(\frac{1}{4}c_0Q_2\right) \not= \varnothing$ imply, respectively,
\begin{align*}
    |x_{Q_1} - x_{Q_2}| &\geq c_0(\rad(Q_1) + \rad(Q_2)),\\
    |\pi(x_{Q_1}) - \pi(x_{Q_2})| &\leq \frac{c_0}{4}(\rad(Q_1) + \rad(Q_2)).
\end{align*}
From these, we estimate
\begin{align*}
    |\pi^{\perp}(x_{Q_1}) - \pi^{\perp}(x_{Q_2})| &\geq |x_{Q_1} - x_{Q_2}| - |\pi(x_{Q_1}) - \pi(x_{Q_2})|\\ 
    &\geq \frac{3c_0}{4}(\rad(Q_1) + \rad(Q_2)) \geq \frac{c_0}{2}\rad(Q_1) + c_0\rad(Q_2).
\end{align*}
Therefore, $B(\eta_{Q_1},100\epsilon_3\diam(Q_1)) \cap U_{Q_2} =\varnothing$ because $100\epsilon_3\diam(Q_1) < \frac{c_0}{4}\rad(Q_1)$ so that Lemma \ref{l:straightcenter} implies $\gamma_{Q_1}\cap U_{Q_2} = \varnothing$ and $\pi(U_{Q_2})\subseteq\pi(\gamma_{Q_1})$. Because $\pi(U_{Q_2})\subseteq \pi(\gamma_{Q_2})$, we get $U_{Q_2}\subseteq M_\pi$, implying  $\tilde{\mu}(U_{Q_2}) = 2\ell(U_{Q_2})$, contradicting the fact that $\tilde{\mu}(U_{Q_2}) < \epsilon_3^2\ell(U_{Q_2})$ because $Q_2\in\Delta_{2.1.2}$.
\end{proof}
\begin{defn}[Levels and inner/outer cores]
Because we would prefer to work with pairwise disjoint subfamilies of $\Delta_{2.1.2}$ in view of Lemma \ref{l:projection-packing}, we will divide $\Delta_{2.1.2}$ into pairwise disjoint ``levels'' using its tree structure. Indeed, fix $j,\ 1\leq j \leq J$ and recall $\scQ_j$ is one of the $J$ families of balls ordered by inclusion of cores constructed in Proposition \ref{p:cores}. Consider the family $\Delta_{2.1.2}^j \vcentcolon= \Delta_{2.1.2}\cap\scQ_j$. We define the $k$-th level of $\Delta_{2.1.2}^j$ for $k \geq 0$ as
\begin{equation*}
    \scL_k \vcentcolon= \{Q\in\Delta_{2.1.2}^j : Q\in C^k(Q(T)),\ T\in\mathcal{T}_{\Delta_{2.1.2}^j}^{c_0}\}
\end{equation*}
where we set $C^0(Q(T)) = \{Q(T)\}$. The family $\scL_k$ is pairwise disjoint for any $k \geq 0$ and $\Delta_{2.1.2}^j = \bigcup_{k\geq0}\scL_k$. We additionally want to single out balls which live away from from the boundary of the core of their parent in the tree structure. We define the inner and outer balls:
\begin{align*}
    \Delta_I &\vcentcolon= \scL_0\cup\{Q\in\Delta_{2.1.2}^j\setminus\scL_0 : 2Q\subseteq (1-5\epsilon_1)c_0P(Q)\},\\
    \Delta_O &\vcentcolon= \Delta_{2.1.2}^j \setminus \Delta_I.
\end{align*}
\end{defn}
We can show that the diameters of the outer cores are controlled by the diameters of the inner cores using Lemma \ref{l:projection-packing} and some algebra.
\begin{lem}\label{l:deltaO-bound}
\begin{equation*}
    \sum_{Q\in\Delta_O}\diam(U_Q) \lesssim \sum_{Q\in\Delta_I}\diam(U_Q).
\end{equation*}
\end{lem}
\begin{proof}
Fix $Q\in\Delta^j_{2.1.2}$ and let $Q'\in C(Q)$. Recall that $\diam(2Q') \leq \epsilon_1\diam(U_{Q})$ as in \eqref{e:child-diam}. Hence, if $Q'\in\Delta_O$, then Lemma \ref{l:straightcenter} implies
\begin{equation}\label{e:inner-core-inclusion}
    (\Gamma\setminus\gamma_Q)\cap2Q' \cap \pi^{-1}\left((1-10\epsilon_1)c_0Q\right) = \varnothing.
\end{equation}
Because the cores of balls in $C(Q)$ are pairwise disjoint, the projection lemma \ref{l:projection-packing} implies
\begin{equation*}
    \sum_{Q'\in C(Q)\cap\Delta_O}\diam(U_{Q'}) \leq 5\sum_{Q'\in C(Q)\cap\Delta_O} \sH^1\left(\pi\left(\frac{c_0}{4}Q'\right)\right) \leq 200\epsilon_1c_0\diam(Q) \leq 200\epsilon_1\diam(U_Q).
\end{equation*}
Summing this inequality over $Q\in\scL_k$, we get
\begin{align}\nonumber
    \sum_{Q'\in\scL_{k+1}\cap\Delta_O}\diam(U_{Q'}) &= \sum_{Q\in\scL_k}\sum_{Q'\in C(Q)\cap\Delta_O}\diam(U_{Q'}) \leq 200\epsilon_1\sum_{Q\in\scL_k}\diam(U_Q)\\\label{e:outer-bound}
    &\leq 200\epsilon_1\sum_{Q'\in\scL_{k}\cap\Delta_O}\diam(U_{Q'}) + 200\epsilon_1\sum_{Q'\in\scL_{k}\cap\Delta_I}\diam(U_{Q'}).
\end{align}
In order to simplify the notation, define $S_k^O := \sum_{Q'\in\scL_{k}\cap\Delta_O}\diam(U_{Q'})$ and $S_k^I := \sum_{Q'\in\scL_{k}\cap\Delta_I}\diam(U_{Q'})$. The lemma will follow from some algebraic manipulations of \eqref{e:outer-bound}. We can restate \eqref{e:outer-bound} in this notation as
\begin{equation*}
    S_{k+1}^O \leq 200\epsilon_1S_k^O + 200\epsilon_1S_k^I.
\end{equation*}
Iterating this inequality over $k$, we get
\begin{equation*}
    S_{k}^O \leq \sum_{n=0}^k(200\epsilon_1)^nS_{k-n}^I
\end{equation*}
which gives
\begin{align*}
    \sum_{Q\in\Delta_O}\diam(U_Q) &= \sum_{k=0}^\infty S_k^O \leq \sum_{k=0}^\infty \sum_{n=0}^k (200\epsilon_1)^n S_{k-n}^I = \sum_{\substack{(k,n)\in\N\times\N \\ n \leq k}}(200\epsilon_1)^n S_{k-n}^I\\
    &= \sum_{m=0}^\infty \sum_{n=0}^\infty (200\epsilon_1)^nS_m^I \lesssim \sum_{m=0}^\infty S_m^I = \sum_{Q\in\Delta_I}\diam(U_Q).
\qedhere\end{align*}

\end{proof}
With this lemma, we now concentrate on proving $\sum_{Q\in\Delta_I}\diam(U_Q) \lesssim \tilde{\mu}(\Gamma)$. Because $\Delta_I\subseteq\Delta_{2.1}$, any $Q\in\Delta_I$ has the existence of an arc $\tau_Q$ as described in the section following Lemma \ref{l:big-proj}. It follows from the neighborhood containment of $\gamma_Q$ in Lemma \ref{l:straightcenter} that there exists a subarc $\zeta_Q\subseteq\tau_Q\cap (1-c_0)2Q \subseteq \gamma(\overline{M}_\pi)$ such that 
\begin{align*}
    \Diam(\zeta_Q) &\geq \frac{1}{10}\diam(2Q).
\end{align*}
These properties imply
\begin{equation}\label{e:zeta-big-mu}
    \tilde{\mu}(\zeta_Q) \geq \frac{1}{5}\diam(2Q).
\end{equation}
Fix a level $\scL_k^I := \scL_k\cap\Delta_I$ and define an equivalence relation on $\scL_k^I$ by putting $Q\sim Q'$ if and only if there exists a collection $\{Q_n\}_{n\geq1}\subseteq \scL_k^I$ such that $\zeta_{Q_i}\cap \zeta_{Q_{i+1}} \not=\varnothing$ while both $\zeta_{Q}\cap \bigcup_{n\geq1}\zeta_{Q_n}\not=\varnothing$ and $\zeta_{Q'}\cap \bigcup_{n\geq1}\zeta_{Q_n}\not=\varnothing$. That is, $Q\sim Q'$ if and only if $\zeta_Q$ and $\zeta_{Q'}$ can be connected by a connected path of $\zeta$ arcs from $\scL_k^I$. This partitions $\scL_k^I$ into equivalence classes $\scL_k^I = \bigcup_{i\in I_k}\scC_{k,i}$. In each equivalence class $\scC_{k,i}$, there exists a ball $Q^M_{k,i}$ of maximal diameter. The arc $\zeta_{Q^M_{k,i}}$ dominates the sum of diameters of balls in this equivalence class in the sense of the following lemma:
\begin{lem}\label{l:maximal-tau-bound}
\begin{equation*}
    \sum_{Q\in\scC_{k,i}} \diam(U_{Q}) \lesssim \tilde{\mu}(\zeta_{Q^M_{k,i}}).
\end{equation*}
\end{lem}
\begin{proof}
Define $\zeta_{k,i} \vcentcolon= \bigcup_{Q\in\scC_{k,i}}\zeta_Q$. We claim that for any $Q\in\scC_{k,i}$,
\begin{align*}
    0 < \dist(\pi(x_Q),\pi(\Image(\zeta_{k,i}))) &\leq \diam(2Q) \leq \diam(2Q^M_{k,i}).
\end{align*}
Indeed, to prove the left inequality, suppose that $\dist(\pi(x_Q),\pi(\Image(\zeta_{k,i}))) = 0$. Because $\zeta_Q \subseteq \gamma(\overline{M}_\pi)$ for any $Q\in\scL_k$, we know that $\zeta_{k,i}\subseteq \gamma(\overline{M}_\pi)$. Therefore, there exists $\phi\in\Phi$ such that $\zeta_{k,i}\subseteq \phi$ and $\ell(\phi) \geq \ell(\zeta_{Q^M_{k,i}}) \geq \frac{1}{10}\diam(2Q^M_{k,i}) \geq \diam(U_Q)$. Recalling the definition of $N(\Phi)$ (see Definition \ref{def:mu-tilde}), it follows from the fact that $\sup_{x\in U_Q}\dist(\pi(x),\pi(\Image(\phi))) \leq \diam(U_Q)$ that $\gamma^{-1}(U_Q)\subseteq N(\Phi)$, implying $\tilde{\mu}(U_Q) \geq \ell(U_Q)$ which contradicts the fact that $Q\in\Delta_{2.1.2}$. For the right inequality above, notice that $\zeta_Q\subseteq 2Q$ so that $\zeta_{k,i}\cap 2Q\not=\varnothing$.
Therefore, by Lemma \ref{l:projection-packing}, $\{\pi\left(\frac{c_0}{4}Q\right)\}_{Q\in\scC_{k,i}}$ is a collection of pairwise disjoint intervals of total length less than $10\diam(Q^M_{k,i})$. This means
\begin{align*}
    \sum_{Q\in\scC_{k,i}}\diam(U_Q) \leq 5\sum_{Q\in\scC_{k,i}}\sH^1\left(\pi\left(\frac{c_0}{4}Q\right)\right) \leq 50\diam(Q^M_{k,i}) \leq 250\tilde{\mu}(\zeta_{Q^M_{k,i}})
\end{align*}
using \eqref{e:zeta-big-mu} in the final inequality.
\end{proof}
The following lemma gives the reason for restricting this argument to the inner cores.
\begin{lem}\label{l:zeta-disjoint}
The arcs in the collection $\{\zeta_{Q^M_{k,i}}\}_{k\geq0,i\in I_k}$ have pairwise disjoint images. As a result,
\begin{equation*}
    \sum_{Q\in\Delta_I}\diam(U_Q) \lesssim \tilde{\mu}(\Gamma).
\end{equation*}
\end{lem}
\begin{proof}
Fix $k,k' \geq 0$ and $i\in I_k,\ i'\in I_{k'}$. Because the cores of balls in $\scL_k$ are pairwise disjoint for any $k\geq0$, we can assume $k > k' \geq 0$. For ease of notation, let $Q_1 := Q^M_{k,i},\ Q_2 := Q^M_{k',i'}$, and $Q'=P(Q_1)$ which exists because $k > 0$. Suppose by way of contradiction that $\Image(\zeta_{Q_1})\cap\Image(\zeta_{Q_2})\not=\varnothing$. We first claim that $U_{Q'}\cap U_{Q_2} = \varnothing$. Indeed, further suppose by way of contradiction that $U_{Q_2}\cap U_{Q'} \not=\varnothing$. Then $k < k'$ implies $Q'\in C^m(Q_2)$ for some $m \geq 0$. Because $Q_1\in \Delta_I$ and because $\diam(Q_2) \geq \diam(Q') > \diam(Q_1)$, we have
\begin{align*}
    2Q_1 \subseteq (1-5\epsilon_1)c_0Q', \text{ and }
\end{align*}
\begin{equation*}
    \diam(2Q_1) \leq \epsilon_1\diam(U_{Q'}).
\end{equation*}
We conclude $2Q_1\subseteq U_{Q'} \subseteq U_{Q_2}$, which contradicts $\Image(\zeta_{Q_1})\cap\Image(\zeta_{Q_2})\not=\varnothing$ because $\zeta_{Q_2}\cap U_{Q_2} = \varnothing$ by definition as a subarc of $\tau_{Q_2}$.

From this claim, we see that $\dist(2Q_1, x_{Q_2}) \geq 4\epsilon_1\diam(U_{Q'})$ so that we can again conclude $\diam(Q_2) \geq \diam(Q')$, for otherwise we would have $2Q_1\cap2Q_2 = \varnothing$ which is in contradiction to our starting assumption that $\Image(\zeta_{Q_1})\cap\Image(\zeta_{Q_2})\not=\varnothing$. Now, because $\zeta_{Q_2},\zeta_{Q_1}\subseteq \gamma(M_\pi)$ and $\Image(\zeta_{Q_2})\cap\Image(\zeta_{Q_1})\not=\varnothing$, we can conclude that there exists a bend $\phi\in\Phi$ such that $\Image(\zeta_{Q_1})\cup\Image(\zeta_{Q_2})\subseteq \Image(\phi)$, hence $\ell(\phi) \geq \ell(\zeta_{Q_2}) \geq \frac{1}{10}\diam(2Q_2)$. Because $2Q_1\cap 2Q_2 \not=\varnothing$, this implies $\gamma^{-1}(Q_1)\subseteq N(\Phi)$ so that $\tilde{\mu}(U_{Q_1}) \geq \ell(U_{Q_1})$, contradicting the fact that $Q\in\Delta_{2.1.2}$ and implying our assumption that $\Image(\zeta_{Q_1})\cap\Image(\zeta_{Q_2})\not=\varnothing$ must be false. This proves the first claim of the lemma. Using Lemma \ref{l:maximal-tau-bound}, we get
\begin{equation*}
    \sum_{Q\in\Delta_I}\diam(U_Q) = \sum_{k \geq 0}\sum_{i\in\scC_{k,i}}\sum_{Q\in\scC_{k,i}}\diam(U_Q) \lesssim \sum_{k \geq 0}\sum_{i\in\scC_{k,i}}\tilde{\mu}(\zeta_{Q^M_{k,i}}) \leq \tilde{\mu}(\Gamma).\qedhere
\end{equation*}
\end{proof}
Now that we have controlled the inner cores, we can finish the proof of the bound for $\Delta_{2.1.2}$ and of the proof of Theorem \ref{t:thmA}.

\begin{prop}
\begin{equation*}
    \sum_{Q\in\Delta_{2.1.2}}\beta_\Gamma(Q)^2\diam(Q) \lesssim_A \sum_{Q\in\Delta_{2.1.2}}\diam(U_Q) \lesssim_J \tilde{\mu}(\Gamma).
\end{equation*}
\end{prop}
\begin{proof}
Using Lemmas \ref{l:deltaO-bound} and \ref{l:zeta-disjoint}, we get
\begin{equation*}
    \sum_{Q\in\Delta_{2.1.2}}\diam(U_Q) = \sum_{j=1}^J\sum_{Q\in\Delta_{2.1.2}^j}\diam(U_Q) \lesssim_J \sum_{Q\in\Delta_I}\diam(U_Q) + \sum_{Q\in\Delta_O}\diam(U_Q) \lesssim \tilde{\mu}(\Gamma).
\qedhere\end{equation*}
\end{proof}
This completes the proof of Theorem \ref{t:thmA}.
\begin{remark}[General rectifiable arcs]\label{rem:rect-almost-flat}
Given a rectifiable arc $\Gamma_r$ with arc length parameterization $\gamma_r$, one can proceed as for Jordan arcs and define the measure $\mu$ by $d\mu\vcentcolon=(\gamma_r)_*(\rho\ d\ell)$. Lemma \ref{l:bendconnected} likely holds, and one can define bends and likely carry out a similar program to that of Section \ref{subsec:almost-flat-arcs-gamma} to show that $\sum_{Q\in\scA}\beta_{\Gamma_r}(Q)^2\diam(Q)\lesssim_A\ell(\Gamma_r) - \crd(\Gamma_r)$. This is importantly weaker than the more desirable inequality
\begin{equation}\label{e:rect-almost-flat-beta}
    \sum_{Q\in\scA}\beta_{\Gamma_r}(Q)^2\diam(Q)\lesssim_A\sH^1(\Gamma_r) - \crd(\Gamma_r).
\end{equation}
If one would like to achieve \ref{e:rect-almost-flat-beta} via methods similar to those used here, one likely needs a stronger definition of $\mu$. For any $t\in I$, set $m(t)\vcentcolon=\inf\{\gamma_1'(s) : \gamma(s) = \gamma(t)\}$. A more prudent choice of $\mu$ might be something like
\begin{align*}
    d\mu_r\vcentcolon= \sigma_rd\sH^1
\end{align*}
where
\begin{equation*}
    \sigma_r(x)\vcentcolon= \begin{cases}
        1-\gamma_1'(t), & \gamma_1'(t) \text{ exists and $\gamma_1'(t) = \inf\{\gamma_1'(s) : s\in\gamma^{-1}(x)\}$}\\
        0 &\text{otherwise}.
    \end{cases}
\end{equation*}
That is, we assign to $x$ the maximal $\rho$ value achieved on $\gamma^{-1}(x)$. We do not investigate this approach further here.
\end{remark}
\end{subsubsection}

\end{subsection}
\end{section}

\begin{section}{Theorem \ref{t:thmB}}\label{sec:thm-b}
In this section, we show how making minor modifications to the proof of the $\gtrsim$ direction of Theorem 1.3 in $\protect\cite{Bi20}$ gives a proof of Theorem \ref{t:thmB}. First, we need a slightly weaker version of Theorem \ref{t:thmB}.
\begin{thm}[\protect\cite{Bi20} Theorem 1.1 in $\R^n$]\label{t:lengthminusdiam}
Let $\Gamma\subseteq \ell_2$ be a rectifiable Jordan arc. For any multiresolution family $\scH$ associated to $\Gamma$ with inflation factor $A > 30$, we have
\begin{equation*}
    \ell(\Gamma) - \diam(\Gamma)  \lesssim \sum_{Q\in\scH}\beta_\Gamma(Q)^2\diam(Q).
\end{equation*}
\end{thm}
\begin{proof}
The proof is similar in form to Bishop's proof in $\R^n$, but we construct coverings of the curve by pieces of $\Gamma$ inside Voronoi cells centered at net points rather than dividing convex hulls of pieces of the curve along diameter segments. Assume $\diam(\Gamma) = 1$. Define
\begin{equation*}
    \sV_n \vcentcolon= \{V_n(x):x\in X_{n}\}
\end{equation*}
where, given any $x\in X_{n}$,
\begin{equation*}
    V_n(x) \vcentcolon= \{y\in\Gamma: \forall z\in X_{n},\ |x-y| \leq |z-y|\}.
\end{equation*}
Since $X_{n}$ is a $2^{-n}$-net centered on $\Gamma$, it is clear that
\begin{equation}\label{e:vordiam}
    2^{-n-1} \leq \diam(V_n(x)) < 2^{-n+1}.
\end{equation}
By definition, for any $n\geq 0$ we also have $\Gamma = \bigcup_{x\in X_{n}} V_n(x)$ so that
\begin{equation*}
    \ell(\Gamma) = \sH^1(\Gamma) \leq \limsup_{n\rightarrow\infty}\sum_{x\in X_{n}}\diam(V_n(x)).
\end{equation*}
This means is suffices to prove
\begin{equation}\label{e:lengthminusdiamfinal}
    \sum_{x\in X_{n}}\diam(V_n(x)) \leq \diam(\Gamma) + C\sum_{Q\in\scH}\beta_\Gamma(Q)^2\diam(Q)
\end{equation}
for any $n\geq0$ and some $C>0$. We will show
\begin{equation}\label{e:diamrecurse}
    \sum_{x\in X_{n}} \diam(V_n(x)) \leq \sum_{y\in X_{n-1}}\diam(V_{n-1}(y)) + C^\prime\sum_{Q\in\scG_n}\beta_\Gamma(Q)^2\diam(Q)
\end{equation}
where each ball $Q\in\scH$ will only appear in $\scG_n$ for a bounded number of values of $n$. We can prove Theorem \ref{t:lengthminusdiam} by repeatedly applying inequality (\ref{e:diamrecurse}). Indeed,
\begin{align*}
    \sum_{x\in X_{n}} \diam(V_n(x)) &\leq \sum_{y\in X_{n-1}}\diam(V_{n-1}(y)) + C^\prime\sum_{Q\in\scG_n}\beta_\Gamma(Q)^2\diam(Q)\\
    &\leq \sum_{y\in X_{n-2}}\diam(V_{n-2}(y)) +  C^\prime\sum_{Q\in\scG_{n-1}}\beta_\Gamma(Q)^2\diam(Q) + C^\prime\sum_{Q\in\scG_n}\beta_\Gamma(Q)^2\diam(Q)\\
    &\ \vdots\\
    &\leq \diam(V_0(x_0)) + C^\prime\sum_{k=1}^n\sum_{Q\in\scG_k}\beta_\Gamma(Q)^2\diam(Q)\\
    &\leq \diam(\Gamma) + C\sum_{Q\in\scH}\beta_\Gamma(Q)^2\diam(Q).
\end{align*}

We now turn to proving \eqref{e:diamrecurse}. Fix $\epsilon < 2^{-10}$. For any net point $x\in X_{n}$, call $x$ \textit{flat} if $\beta_\Gamma(B(x,10\cdot2^{-n})) < \epsilon$ and call $x$ \textit{non-flat} if $\beta_\Gamma(B(x,10\cdot2^{-n})) \geq \epsilon$. We will construct a function $P:\cup_nX_{n}\rightarrow \cup_n X_{n}$ which assigns each $y\in X_{n+1}$ to a parent $P(y)\in X_{n}$. Since $\diam(\Gamma) = 1$, $\sV_0 = \{V_0(x_0)\}$ so for any $y\in X_1$, define $P(y) = x_0$. Fix $n > 0$ and a point $y\in X_{n+1}$.
If there exists $x'\in X_{n}$ such that $x'$ is non-flat and $y\in V_n(x')$, then define $P(y) = x'$. Otherwise, every $x\in X_{n}$ such that $y\in V_n(x)$ is flat. Choose one such $x$ and let $x',y'\in V_n(x)$ such that $\diam(V_n(x)) = |x'-y'|$. We define $d_x\vcentcolon=[x',y']$ and call $d_x$ a diameter segment of $V_n(x)$. Let $\pi_{d_x}:\ell_2\rightarrow \R$ be the orthogonal projection onto the line containing $d_x$. Since $\epsilon$ is small, we can write
\begin{align*}
    X_{n}\cap B(x,10\cdot2^{-n}) &= \{v_1,\ldots, v_N\}\\
    X_{n+1}\cap V_n(x) &= \{u_1,\ldots, u_M\} 
\end{align*}
where 
\begin{align*}
    \pi_{d_x}(v_1) &< \ldots < \pi_{d_x}(v_N),\\
    \pi_{d_x}(u_1) &< \ldots < \pi_{d_x}(u_M)
\end{align*}
We define $E_x \vcentcolon= \{u_1, u_M\}$. If $y\not\in E_x$, then define $P(y) = x$. If $y\in E_x$ we define $P(y)$ dependent on the behavior of points adjacent to $x$ in $X_{n}$. Assuming $x = v_j$ with $0 < j < N$, suppose first that $y = u_1$. Then if $v_{j-1}$ is non-flat, define $P(y) = v_{j-1}$. Similarly, if $y = u_M$ and $v_{j+1}$ is non-flat, then put $P(y) = v_{j+1}$. Otherwise, put $P(y) = x$.

With the function $P$ defined, we write
\begin{align*}
    \sum_{y\in X_{n+1}} \diam(V_{n+1}(y)) = \sum_{\substack{y\in X_{n+1} \\ P(y)\ \text{flat}}} \diam(V_{n+1}(y)) + \sum_{\substack{z\in X_{n+1} \\ P(z)\  \text{non-flat}}} \diam(V_{n+1}(z)).
\end{align*}
If $P(z)$ is non-flat, then $|z - P(z)| < 4\cdot 2^{-n}$ so that $B(P(z), 10\cdot 2^{-n}) \subseteq B(z,A2^{-n-1})$ because $A > 30$. Hence, $\beta_\Gamma(B(z,A2^{-n-1}))\gtrsim_A \epsilon$. Using (\ref{e:vordiam}), this means
\begin{align}
    \sum_{\substack{z\in X_{n+1} \\ P(z)\  \text{non-flat}}}& \diam(V_{n+1}(z)) \lesssim_{\epsilon, A} \sum_{\substack{z\in X_{n+1} \\ P(z)\  \text{non-flat}}} \beta_\Gamma(B(z,A2^{-n-1}))^2\diam(B(z,A2^{-n-1})) \nonumber\\
    &\leq \sum_{\substack{x\in X_{n}\\x\ \text{non-flat}}}\left[\diam(V_n(x)) + C^\prime \sum_{ P(z)=x}\beta_\Gamma(B(z,A2^{-n-1}))^2\diam(B(z,A2^{-n-1}))\right]\nonumber\\
    &\leq \sum_{\substack{x\in X_{n}\\x\ \text{non-flat}}}\diam(V_n(x)) + C^\prime\sum_{\rad(Q)=A2^{-n-1}} \beta_\Gamma(Q)^2 \diam(Q).
\end{align}

Now, let $x\in X_{n}$ be flat. We will construct a set $\seg(x)$ of subsets of diameter segments of the sets $\{V_{n+1}(y)\}$ for $y$ with $P(y)$ flat. For $y\in X_{n+1}$ with $P(y)$ flat, let $d_y$ be a diameter segment. If $y\in X_{n+1}\cap V_n(x)$ but $y\not\in E_x$, then put $d_y$ into $\seg(x)$. If $y = u_1\in E_x$, then we have the decomposition
\begin{equation*}
    V_{n+1}(y) = (V_{n+1}(y) \cap V_n(x)) \cup (V_{n+1}(y) \cap V_n(v_{j-1}))
\end{equation*}
because $\epsilon$ is small. Put the segment $d_y\cap V_n(v_{j-1})$ in $\seg(v_{j-1})$ and the segment $d_y \cap V_n(x)$ in $\seg(x)$. We similarly handle the case when $y = u_N$. In this case, put the segment $d_y \cap V_n(x)$ in $\seg(x)$ and the segment $d_y \cap V_n(v_{j+1})$ in $\seg(v_{j+1})$. With these sets constructed, we can now write
\begin{equation}\label{e:vordiamdecomp}
    \sum_{\substack{y\in X_{n+1} \\ P(y)\ \text{flat}} }\diam(V_{n+1}(y)) = \sum_{\substack{y\in X_{n+1} \\ P(y)\ \text{flat}}}d_y = \sum_{\substack{x\in X_{n} \\ x\ \text{flat}}}\sum_{s\in\seg(x)} \sH^1(s).
\end{equation}
With (\ref{e:vordiamdecomp}) in place, we only need to give an appropriate bound for $\sum_{s\in\seg(x)}\sH^1(s)$. Define $Q_x = B(x,A2^{-n})$. We claim
\begin{equation}\label{e:vorsegest1}
    \sum_{s\in\seg(x)}\sH^1(s) \leq \diam(V_n(x)) + C\beta_\Gamma(Q_x)^2 \diam(Q_x)
\end{equation}
for some large $C>0$. In order to prove this statement, we first state a lemma given in \protect\cite{BS17}:

\begin{lem}(\protect\cite{BS17} Lemma 8.3)
Suppose that $V\subseteq\R^n$ is a $1$-separated set with $\#V\geq 2$ and there exist lines $\ell_1$ and $\ell_2$ and a number $0\leq \alpha\leq 1/16$ such that $$\dist(v,\ell_i)\leq \alpha\quad\text{for all $v\in V$ and $i=1,2$}.$$ Let $\pi_i$ denote the orthogonal projection onto $\ell_i$. There exist compatible identifications of $\ell_1$ and $\ell_2$ with $\R$ such that $\pi_1(v')\leq \pi_1(v'')$ if and only if $\pi_2(v')\leq \pi_2(v'')$ for all $v',v''\in V$. If $v_1$ and $v_2$ are consecutive points in $V$ relative to the ordering of $\pi_1(V)$, then 
\begin{equation} 
\sH^1([u_1,u_2]) < (1+ 3\alpha^2)\cdot \sH^1([\pi_1(u_1),\pi_1(u_2)]) \quad\text{for all }[u_1,u_2]\subseteq[v_1,v_2].
\end{equation}
Moreover, \begin{equation} 
\sH^1([y_1,y_2]) < (1+12\alpha^2)\cdot\sH^1([\pi_1(y_1),\pi_1(y_2)])\quad\text{for all }[y_1,y_2]\subseteq\ell_2.
\end{equation}
\end{lem}
Applying this lemma with $\ell_1 = d_x$ and $\ell_2 = d_y$, for any segment $[s_1,s_2]\subseteq d_y$, we have
\begin{equation*}
    \sH^1([s_1,s_2]) < (1 + C\beta_\Gamma(Q_x)^2\diam(Q_x))\sH^1([\pi_{d_x}(s_1),\pi_{d_x}(s_2)]).
\end{equation*}
Enumerate $\seg(x) = \{s_1,\ldots, s_N\}$. With this, we can write
\begin{align}
    \sum_{s\in\seg(x)}\sH^1(s) &< (1 + C\beta_\Gamma(Q_x)^2\diam(Q_x))\sum_{s\in\seg(x)}\sH^1(\pi_{d_x}(s)) \nonumber\\
    &\leq (1 + C\beta_\Gamma(Q_x)^2\diam(Q_x))\left( \sH^1(d_x) + 2\sum_{i=1}^{N-1}\sH^1(\pi_{d_x}(s_i)\cap \pi_{d_x}(s_{i+1})) \right) \nonumber\\
    &= (1 + C\beta_\Gamma(Q_x)^2\diam(Q_x))\left( \diam(V_n(x)) + 2\sum_{i=1}^{N-1}\sH^1(\pi_{d_x}(s_i)\cap \pi_{d_x}(s_{i+1}))\right).\label{e:vorsegest2}
\end{align}
Because $x$ is flat, $\#\seg(x)$ is bounded above by a universal constant, so we only need to show that 
\begin{equation}\label{e:voroverlap}
   \sH^1(\pi_L(s_i)\cap \pi_L(s_{i+1})) \leq C^\prime\beta_\Gamma(Q_x)^2\diam(Q_x). 
\end{equation}
for some $C^\prime>0$. It suffices to bound the length of the overlap between the projections of consecutive Voronoi cells $V_{n+1}(u_i), V_{n+1}(u_{i+1})$ within the tube of radius $2\beta_\Gamma(Q_x)\diam(Q_x)$ around $d_x$. The boundary between the cells is the intersection of this tube with the hyperplane of points of equal distance from both $u_i$ and $u_{i+1}$. A simple geometric estimate of the type carried out in \cite{BS17} pages 41 and 42 gives \eqref{e:voroverlap} (also see Figure \ref{fig:VCEstimate}). Combining (\ref{e:vorsegest2}) and (\ref{e:voroverlap}) gives (\ref{e:vorsegest1}). Applying (\ref{e:vorsegest1}) to (\ref{e:vordiamdecomp}), we can finally write
\begin{align*}
    \sum_{y\in X_{n+1}} \diam(V_{n+1}(y)) &= \sum_{\substack{y\in X_{n+1} \\ P(y)\ \text{flat}}} \diam(V_{n+1}(y)) + \sum_{\substack{z\in X_{n+1} \\ P(z)\  \text{non-flat}}} \diam(V_{n+1}(z))\\
    &\leq \sum_{\substack{x\in X_{n}\\x\ \text{non-flat}}}\diam(V_n(x)) + C^\prime\sum_{\rad(Q) = A2^{-n-1}} \beta_\Gamma(Q)^2 \diam(Q)\\
    &\quad\quad + \sum_{\substack{x\in X_{n} \\ x\ \text{flat}}}\sum_{s\in\seg(x)} \sH^1(s)\\
    &\leq \sum_{\substack{x\in X_{n}\\x\ \text{non-flat}}}\diam(V_n(x)) + C^\prime\sum_{\rad(Q) = A2^{-n-1}} \beta_\Gamma(Q)^2 \diam(Q)\\
    &\quad\quad + \sum_{\substack{x\in X_{n} \\ x\ \text{flat}}} \diam(V_n(x)) + C^{\prime\prime}\sum_{\rad(Q) = A2^{-n}} \beta_\Gamma(Q)^2\diam(Q)\\
    &\leq \sum_{x\in X_{n}}\diam(V_n(x)) + C\sum_{\rad(Q)\in \{A2^{-n-1}, A2^{-n}\}}\beta_\Gamma(Q)^2\diam(Q).
\end{align*}
This proves inequality (\ref{e:lengthminusdiamfinal}) and finishes the proof of Theorem \ref{t:lengthminusdiam}.
\end{proof}

\begin{figure}[h]
    \centering
    \includegraphics[scale=0.9]{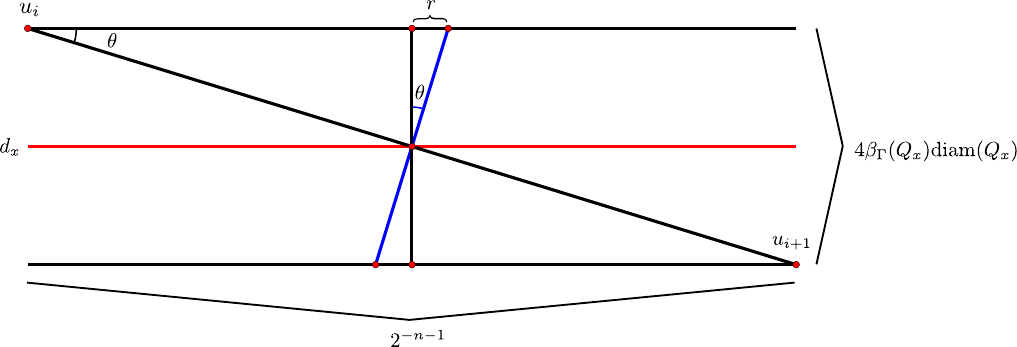}
    \caption{This figure depicts a worst case estimate for the overlap of $V_{n+1}(u_i)$ and $V_{n+1}(u_{i+1})$ along the diameter $d_x$ for $u_i,u_{i+1}\in X_{n+1}\cap V_n(x)$. The blue line depicts the hyperplane which forms the boundary between $V_{n+1}(u_i)$ and $V_{n+1}(u_{i+1})$ while the overlap is bounded above by the quantity $r$. We have $\tan \theta = \frac{2\beta_\Gamma(Q_x)\diam(Q_x)}{2^{-n-1-1}} = \frac{r}{2\beta_\Gamma(Q_x)\diam(Q_x)}$. Rearranging, we find $r = 8\beta_\Gamma(Q_x)^2\diam(Q_x)^2\cdot 2^{n+1}\lesssim_A\beta_\Gamma(Q_x)^2\diam(Q_x)$ as desired.}
    \label{fig:VCEstimate}
\end{figure}

The method for replacing $\diam(\Gamma)$ with $\crd(\Gamma)$ in Theorem \ref{t:lengthminusdiam} is nearly identical to that given in $\protect\cite{Bi20}$. Since most of the argument is dimension independent, we only need to replace certain collections of dyadic cubes with appropriate collections of balls in a multiresolution family, and switch out applications of the version of Theorem \ref{t:lengthminusdiam} proven there with Theorem \ref{t:lengthminusdiam} itself. We now give a summary of the needed modifications to Bishop's proof.

\begin{proof}(Theorem \ref{t:thmB})
We will make some amendments to Section 4 of \protect\cite{Bi20} beginning on page 15, but the vast majority of the proof is identical because Bishop's arguments are mostly dimension-independent from this point on. We reproduce most of the argument here for convenience.

Assume the $\sum_{Q\in\scH}\beta_\Gamma(Q)^2\diam(Q) < \infty$ and that $\diam(\Gamma) = 1$. Let $Q_0\in\scH$ with $2 <  \diam(Q_0) \leq 4$ so that $\Gamma \subseteq Q$. Suppose $\beta_0$ is a small positive number to be chosen below. If $\beta_\Gamma(Q_0) > \beta_0$, then we have
\begin{equation*}
    \crd(\Gamma) \leq \diam(\Gamma) = 1 \leq \frac{\beta_\Gamma(Q_0)^2}{\beta_0^2}\diam(Q_0) \lesssim \beta_\Gamma(Q_0)^2\diam(Q_0).
\end{equation*}
Therefore, by Theorem \ref{t:lengthminusdiam},
\begin{equation*}
    \ell(\Gamma) - \crd(\Gamma) \leq \diam(\Gamma) - \crd(\Gamma) + C\sum_{Q\in\scH}\beta_\Gamma(Q)^2\diam(Q) \lesssim \sum_{Q\in\scH}\beta_\Gamma(Q)^2\diam(Q).
\end{equation*}
Hence, we may assume that $\beta_\Gamma(Q_0) < \beta_0$. Let $x,y\in\Gamma$ be such that $|x-y| = \diam(\Gamma)$ and re-orient $\Gamma$ so that $x = 0$ and $y = (1,0,0,\ldots)$. Assuming $t(x) < t(y)$, let $\gamma_1 = \gamma(0,t(x))$ and $\gamma_2 = \gamma(t(y),\ell(\Gamma))$. That is, $\gamma_1$ is the subarc of $\Gamma$ from the beginning of $\Gamma$ to $x$ and $\gamma_2$ is the subarc from $y$ to the end of $\Gamma$. Observe that
\begin{equation*}
    \crd(\Gamma) \geq \diam(\Gamma) - \ell(\gamma_1) - \ell(\gamma_2)
\end{equation*}
so that applying Theorem \ref{t:lengthminusdiam} three times (using the fact that $\gamma_1,\gamma_2$ are Jordan arcs) gives
\begin{align}
    \ell(\Gamma) - \crd(\Gamma) &\leq \ell(\Gamma) - \diam(\Gamma) + \ell(\gamma_1) + \ell(\gamma_2)\nonumber\\
    &\leq C(A)\sum_{Q\in\scH}\beta_\Gamma(Q)^2\diam(Q) + \diam(\gamma_1) + \diam(\gamma_2).\label{e:curvetips}
\end{align}
This means we only need to show that $\diam(\gamma_1) + \diam(\gamma_2) \lesssim_A \sum_{Q\in\scH}\beta_\Gamma(Q)^2\diam(Q)$. Because the arguments for both arcs are the same, we only consider $\gamma_1$.

Let $\epsilon = \diam(\gamma_1) > 0$. Let $Q_1,\ldots, Q_k$ be balls with diameter going from $\diam(Q_0)$ to $\epsilon$ such that among all balls of their given radius, their center is closest to $x$. We have that $k\simeq \log(\diam(Q_0)/\epsilon)$. If any of these balls satisfies $\beta_\Gamma(Q_j) > \beta_0$, then
\begin{equation*}
    \diam(\gamma_1) \leq \frac{\beta_\Gamma(Q_j)^2}{\beta_0^2}\diam(\gamma_1) \lesssim \beta_\Gamma(Q_j)^2\diam(Q_j)
\end{equation*}
so that $\diam(\gamma_1)$ satisfies the desired bound. Hence, we assume that $\beta_\Gamma(Q_j) \leq \beta_0$ for all $1 \leq j \leq k$. Let $L_j$ be an infimizing line in the definition of $\beta_\Gamma(Q_j)$. We measure the angle that $L_j$ makes with the $x_1$-axis by translating it to intersect $0$, then measuring the angle between these lines in the (at most 2-dimensional) plane containing them.\\
\textbf{Case 1:} Assume that $L_i$ makes an angle larger than 10$\beta_0$ with the $x_1$ axis for some $i$. Since the angle between $L_0$ and $L_i$ is bounded by $C\sum_{j=1}^i\beta_\Gamma(Q_j)$ and the best line for $Q_0$ is within an angle $\beta_0$ of the $x_1$-axis, we have $\sum_{j=1}^k\beta_\Gamma(Q_j) \gtrsim \beta_0 \gtrsim 1$. The Cauchy-Schwarz inequality implies
\begin{align*}
    1 \lesssim \left( \sum_{j=1}^k\beta_\Gamma(Q_j) \right)^2 \lesssim \left( \sum_{j=1}^k\beta_\Gamma(Q_j)^22^{-j}\right) \cdot \left( \sum_{j=1}^k 2^j \right)  \simeq 2^k\sum_{i=j}^k\beta_\Gamma(Q_j)^2 2^{-j}.
\end{align*}
Therefore, $\sum_{j=1}^k\beta(Q_j) 2^{-j} \gtrsim 2^{-k} \gtrsim \epsilon$ so that
\begin{equation*}
    \epsilon = \diam(\gamma_1) \lesssim \sum_{j=1}^k \beta_\Gamma(Q_j)^2\diam(Q_j)
\end{equation*}
as desired.\\
\textbf{Case 2:} Now, assume that all of the lines $L_j,\ 1\leq j \leq k$ make angle less than $10\beta_0$ with the $x_1$-axis. Consider a subarc $\gamma_1'\subseteq\gamma_1$ that is contained in and connects the boundary components of the annulus 
\begin{equation*}
    B\left(x,\frac{1}{5}\diam(\gamma_1)\right) \bigg\backslash B\left(x, \frac{1}{10}\diam(\gamma_1)\right).
\end{equation*}
Because $\gamma_1'$ and $\gamma_1$ have comparable diameters, it suffices to bound $\diam(\gamma_1')$.

Given any $p\in\gamma_1'$, one of the following two statements holds:
\begin{enumerate}[label=(\roman*)]
    \item Every ball $Q = B(x,A2^{-n})$ with $p\in B(x,2^{-n})$ and $\diam(Q) \leq \frac{1}{10}\diam(\gamma_1)$ satisfies $\beta_\Gamma(Q) \leq \beta_0$.
    \item There exists a ball $Q_p$ of the above form such that $\beta_\Gamma(Q_p) > \beta_0$.
\end{enumerate}
We let $E\subseteq \gamma_1'$ be the set of points $p$ where a ball $Q_p$ as in (ii) exists. Since $\gamma_1'$ is rectifiable, it has tangents almost everywhere. Bishop provides the following two lemmas 

\begin{lem}(\protect\cite{Bi20} Lemma 4.1 for $\scH$)\label{l:bis-crossing}
If $p\in\gamma_1^\prime\backslash E$ and $p$ is a tangent point of $\Gamma$, then $p$ has a ``crossing property": If $Q\in\scH$ has $p\in \frac{1}{2}Q$ with $\diam(Q) \leq \frac{\diam(\gamma_1)}{10}$ then $\gamma_1$ must ``cross" $Q$ in the sense that $\gamma_1$ must connect the two components of $\partial Q\cap W$ where $W$ is a cylinder of radius $\frac{\diam(Q)}{10}$ containing $p$.
\end{lem}
\begin{lem}(\protect\cite{Bi20} Lemma 4.2)\label{l:Egammaprime}
$\ell(E) = \ell(\gamma_1^\prime).$
\end{lem}
We have changed the statement of Lemma \ref{l:bis-crossing} only by replacing the setting from $\R^n$ to $\ell_2$ and replacing dyadic cubes with balls in a multiresolution family. Bishop proves this lemma by constructing a ``dividing'' hypersurface which $\gamma_1$ can only cross once because $\gamma_1$ is a Jordan arc. It is straightforward to modify Bishop's construction by replacing $n-1$-dimensional planes in $\R^n$ with corresponding hyperplanes in $\ell_2$. Given Lemma \ref{l:bis-crossing}, Lemma \ref{l:Egammaprime} follows directly from Bishop's original argument. For the proofs of these results, we direct the reader to \cite{Bi20} (especially see Figure 6 there for a good picture of Lemma \ref{l:bis-crossing}).

Given these lemmas, we can complete the proof of Theorem \ref{t:thmB} by noting that Lemma \ref{l:Egammaprime} implies that $\gamma_1^\prime$ is nearly covered by the balls $\{Q_p\}_{p\in E}$ so that there is subcollection of distinct balls $\{Q_{p_j}\}_j$ with
\begin{equation*}
    \diam(\gamma_1^\prime) \leq 5\sum_{j}\diam(Q_{p_j}) \lesssim_{\beta_0} \sum_j \beta_\Gamma(Q_{p_j})^2\diam(Q_{p_j}) \leq \sum_{Q\in\scH} \beta_\Gamma(Q)^2\diam(Q).
\end{equation*}
Combining this result with (\ref{e:curvetips}) completes the proof of Theorem \ref{t:thmB}.
\end{proof}
\end{section}

\pagebreak

\bibliographystyle{alpha}
\bibliography{database}
\end{document}